\documentclass{article}
\usepackage{amssymb,latexsym,amsmath,amscd,amsthm,amsfonts,wasysym,amsmath}
\usepackage{graphicx}
\usepackage{enumerate}

\newtheorem{theorem}{Theorem}[section]
\newtheorem{lemma}[theorem]{Lemma}
\newtheorem{corollary}[theorem]{Corollary}
\newtheorem{proposition}[theorem]{Proposition}

\theoremstyle{definition}
\newtheorem{definition}[theorem]{Definition}
\newtheorem{remark}[theorem]{Remark}
\newtheorem{example}[theorem]{Example}
\newtheorem{notation}[theorem]{Notation}

\newcommand{\Coker}{{ \rm Coker }}
\newcommand{\End}{{ \rm End }}

\newcommand{\Ext}{{ \rm Ext }}
\newcommand{\Hom}{{ \rm Hom }}

\newcommand{\Res}{{ \rm Res }}

\newcommand{\Ker}{{ \rm Ker }}
\newcommand{\Mod}{{ \rm Mod }}

\newcommand{\Tor}{{ \rm Tor }}

\newcommand{\rad}{{ \rm rad }}

\newcommand{\add}{{ \rm add }}

\newcommand{\g}{\hbox{-}}

\newcommand{\hueca}[1]{\mathbb{#1}}
\renewcommand{\mod}{{\rm mod }}

\newcommand{\h}{{\rm I\hskip-2pt h}}
\newcommand{\cirmin }{\circleddash}

\newcommand{\rightmap}[1]{\smash{\mathop{\hbox to
20pt{\rightarrowfill}}\limits^{#1}}}
\newcommand{\leftmap}[1]{\smash{\mathop{\hbox to
20pt{\leftarrowfill}}\limits^{#1}}}

\newcommand{\longrightmap}[1]{\smash{\mathop{\hbox to
4cm{\rightarrowfill}}\limits^{#1}}}
\newcommand{\longleftmap}[1]{\smash{\mathop{\hbox to
4cm{\leftarrowfill}}\limits^{#1}}}
\newcommand{\medrightmap}[1]{\smash{\mathop{\hbox to
2cm{\rightarrowfill}}\limits^{#1}}}
\newcommand{\medleftmap}[1]{\smash{\mathop{\hbox to
2cm{\leftarrowfill}}\limits^{#1}}}

\newcommand{\shortlmapup}[1]
{\uparrow\rlap{$\vcenter{\hbox{$\scriptstyle#1$}}$}}

\newcommand{\shortlmapdown}[1]
{\downarrow\rlap{$\vcenter{\hbox{$\scriptstyle#1$}}$}}

\newcommand{\dobleflecha}[2]{\ \smash{\mathop{
   \raise 3pt \hbox to 40pt{\rightarrowfill}\hskip-40pt
\lower 3pt
   \hbox to 40pt{\rightarrowfill}}\limits^{#1}_{#2}}\ }

\input xy
\xyoption{all}

\begin{document}
   \date{}
   \setcounter{page}{1}

\title{\bf Tame and wild theorem for the category of
filtered by standard modules}
\vskip-1cm
\author{R. Bautista, E. P\'erez and L. Salmer\'on}
\vskip-1cm

\maketitle

\renewcommand{\thefootnote}{}



  \begin{abstract}
\noindent
We introduce the notion of interlaced weak ditalgebras and apply
reduction procedures to
their module categories to prove a tame-wild dichotomy for the category
${\cal F}(\Delta)$ of $\Delta$-filtered modules for an arbitrary finite homological system $({\cal P},\leq,\{\Delta_i\}_{i\in {\cal P}})$.  This includes the case of standardly stratified algebras.  Moreover, in the
tame case, we show
that given a fixed dimension
$d$, for every $d$-dimensional indecomposable module $M\in
{\cal F}(\Delta)$, with the only possible exception of those lying in a finite
number
of isomorphism classes,
the module $M$ coincides with its Auslander-Reiten translate in ${\cal
F}(\Delta)$.
Our proofs rely on the equivalence of ${\cal F}(\Delta)$
with
the module category of some special type of ditalgebra.
\end{abstract}

\section{Introduction}
 Denote by $k$ a fixed ground field, which will be assumed
 to be algebraically closed all over this work.
 Whenever we consider a $k$-algebra or a bimodule,
 we always assume that the field $k$ acts centrally on them.
 Given an algebra $\Lambda$, we denote
 by $\Lambda\g\Mod$  the
 category of left $\Lambda$-modules,
 and by $\Lambda\g\mod$ its full subcategory of finitely generated  modules.

We recall that a \emph{preordered set} $({\cal P},\leq)$ is a non-empty set ${\cal P}$ equipped with a reflexive and transitive relation $\leq$.
  Two elements $i,j\in {\cal P}$ are \emph{equivalent} iff $i\leq j$ and $j\leq i$. In this case, we write $i\sim j$.
 We denote with $\overline{{\cal P}}={\cal P}/\sim$ the set of equivalence
  classes of ${\cal P}$ modulo the equivalence relation $\sim$. For any $i\in {\cal P}$, denote by $\overline{i}$ its equivalence class. Then, $\overline{\cal P}$ is a partially ordered set with the relation defined by $\overline{i}\leq \overline{j}$ iff $i\leq j$.

 We recall the following terminology from \cite{MSX} and \cite{hsb}.

 \begin{definition}\label{D: homological system y F(Delta)}
  Given a finite-dimensional $k$-algebra $\Lambda$, a \emph{(finite) homological system $({\cal P},\leq,\{\Delta_i\}_{i\in {\cal P}})$} for $\Lambda$ consists of
  a finite preordered set $({\cal P},\leq)$ and
  a family of pairwise non-isomorphic indecomposable finite-dimensional $\Lambda$-modules $\{\Delta_i\}_{i\in {\cal P}}$ satisfying the following two conditions:
  \begin{enumerate}
  \item $\Hom_\Lambda(\Delta_i,\Delta_j)\not=0$ implies $i\leq j$;
  \item $\Ext^1_\Lambda(\Delta_i,\Delta_j)\not=0$ implies $i\leq j$ and $i\not\sim j$.
  \end{enumerate}

  We write  $\Delta:=\{\Delta_i\mid i\in {\cal P}\}$ and denote by ${\cal F}(\Delta)$  the full subcategory of $\Lambda\g\mod$ consisting of the trivial module  and all those $M\in \Lambda\g\mod$ which admit a \emph{$\Delta$-filtration},  that is a filtration of submodules
$$0=M_t\subseteq M_{t-1}\subseteq \cdots\subseteq M_1\subseteq M_0=M$$
such that $M_j/M_{j+1}$ is isomorphic to some module in $\Delta$, for each $j\in [0,t-1]$.

 A finite homological system ${\cal H}=({\cal P},\leq,\{\Delta_i\}_{i\in {\cal P}})$ for a finite-dimensional $k$-algebra $\Lambda$ is called \emph{admissible} if $\Lambda\in {\cal F}(\Delta)$ and the number of isoclasses of indecomposable projective $\Lambda$-modules coincides with the cardinality of ${\cal P}$.

 In this case, the modules in the family $\Delta$ are determined up to isomorphism and are known as \emph{the standard modules};  $\Lambda$ is called a \emph{pre-standardly stratified algebra}, see \cite{hsb}(2.11).

We recall also that a pre-standardly stratified algebra $\Lambda$, with preordered index set $({\cal P},\leq)$, is called \emph{standardly stratified} if $({\cal P},\leq)$ is a  partial order. A standardly stratified algebra  $\Lambda$, equipped with the poset $({\cal P},\leq)$, is \emph{quasi-hereditary} iff  $\End_\Lambda(\Delta_i)\cong k$, for each $i\in {\cal P}$.
  \end{definition}

In the following, when we say that \emph{almost every} object in a class ${\cal
M}$ of
objects in a given category satisfies some property, we mean that every object
in
${\cal M}$ has this property, with the possible exception of those lying in a
finite union of isoclasses of ${\cal M}$.

\begin{definition}\label{D: tame and wild}
Let $\Lambda$ be a finite-dimensional algebra and ${\cal C}$ a full subcategory
of
$\Lambda\g\mod$ closed under direct summands and direct sums. Then,
\begin{enumerate}
 \item The category ${\cal C}$ is called \emph{tame} iff, for each dimension
$d$, there are
 rational algebras $\Gamma_1,\ldots,\Gamma_{t_d}$ and bimodules
$Z_1,\ldots,Z_{t_d}$, where each $Z_i$ is a
 $\Lambda\g\Gamma_i$-bimodule,
 which is free of finite rank as a $\Gamma_i$-module, such that almost
every indecomposable
$d$-dimensional
 module $M$ in ${\cal C}$  is of the form $M\cong Z_i\otimes_{\Gamma_i}S$,
 for some $i$ and some simple  $\Gamma_i$-module $S$.
  \item The category ${\cal C}$ is called \emph{strictly tame} iff,
  for each dimension $d$, there are
 rational algebras $\Gamma_1,\ldots,\Gamma_{t_d}$ and bimodules
$Z_1,\ldots,Z_{t_d}$, where each $Z_i$ is a
 $\Lambda\g\Gamma_i$-bimodule,
 which is free of finite rank as a $\Gamma_i$-module, such that almost
every indecomposable
$d$-dimensional
 module $M$ in ${\cal C}$  is of the form $M\cong Z_i\otimes_{\Gamma_i}N$,
 for some $i$ and some indecomposable  $\Gamma_i$-module $N$. Moreover, for each $i\in [1,t_d]$,
 the functor $Z_i\otimes_{\Gamma_i}-:\Gamma_i\g\mod\rightmap{}\Lambda\g\mod$
 preserves isoclasses and indecomposables, and its image lies within
 ${\cal C}$.
 \item The category ${\cal C}$ is \emph{wild} iff there is a $\Lambda\g
k\langle
x,y\rangle$-bimodule $Z$, free of finite rank as a right $k\langle
x,y\rangle$-module, such that $Z\otimes_{k\langle x,y\rangle}-:k\langle
x,y\rangle\g\mod\rightmap{}{\cal C}$
preserves indecomposables and isomorphism classes.
\end{enumerate}
\end{definition}

The preceding definitions are very natural generalizations of the usual definitions of tameness and wildness for finite-dimensional algebras over  algebraically closed fields. This notion of wildeness for subcategories of modules was already considered explicitely in \cite{S0}(\S14.2) and \cite{S1}.
The notion  of tameness in the first item, for particular cases, was somehow considered in
  \cite{S2} and \cite{BKM}. Precise definitions were given in \cite{K}(2.9), but here we do not require that ${\cal C}$ is of infinite representation type to be of tame type.

It is not hard to see that strict tameness implies tameness.

Given a general homological system $({\cal P},\leq,\{\Delta_i\}_{i\in {\cal P}})$ for a finite-dimensional algebra
$\Lambda$, it is well known that the
category ${\cal F}(\Delta)$ of $\Delta$-filtered modules is a subcategory of
$\Lambda\g\mod$ closed under direct sums and direct summands, see \cite{MSX}(3.16) and \cite{P}.
The question of whether a tame and wild dichotomy theorem holds for the
category ${\cal F}(\Delta)$ of filtered by standard modules for a general
quasi-hereditary algebra $\Lambda$ is very natural and was explicitely
raised in \cite{K}. In 2017, we uploaded to arXiv an affirmative answer to this question, see \cite{bpsqh}. This article is a revised version of that one, where some of the arguments are simplified, but furthermore, we generalize it to prove the tame-wild dichotomy for ${\cal F}(\Delta)$ not only for the case of  pre-standardly stratified algebras, but for arbitrary homological systems.
We prove the following.

\begin{theorem}\label{T: main theorem qh-algebas}
 Assume that the ground field $k$ is algebraically closed, let $\Lambda$ be a
finite-dimensional $k$-algebra and $({\cal P},\leq,\{\Delta_i\}_{i\in {\cal P}})$ any homological system for $\Lambda$. Then, the category of $\Delta$-filtered  modules  ${\cal F}(\Delta)$ is either tame or wild,
but not both. Moreover, the category ${\cal F}(\Delta)$ is tame iff it is
strictly tame.
\end{theorem}

An important precedent to our result is \cite{BKM}, where Th. Br\"ustle, S. Koenig, and V. Mazorchuk prove a dichotomy result in a special case of   quasi-hereditary algebra by explicitly classifying the tame and wild subcategories of filtered modules.
D. Simson presented in \cite{S2} another  tame-wild dichotomy result for an interesting subcategory of modules over a commutative uniserial algebra.

By a well known result of C.M. Ringel, for a general homological system $({\cal P},\leq,\{\Delta_i\}_{i\in {\cal P}})$, the category ${\cal F}(\Delta)$ admits almost split
sequences, see \cite{Rin}. Whenever we refer here to almost split sequences in ${\cal F}(\Delta)$, we mean precisely in the sense of Ringel. We will prove the following statement, which is similar to a theorem of
Crawley-Boevey
for the category $\Lambda\g  \mod$, when $\Lambda$ is an arbitrary
finite-dimensional
tame algebra over an algebraically closed field, see \cite{CB1}.

\begin{theorem}\label{T: Crawley para F(Delta)}
  Assume that the ground field $k$ is algebraically closed, let $\Lambda$ be a finite-dimensional $k$-algebra and $({\cal P},\leq,\{\Delta_i\}_{i\in {\cal P}})$ any homological system for $\Lambda$. If the category
  of $\Delta$-filtered modules ${\cal F}(\Delta)$ is tame, then,
for each $d\in\hueca{N}$, almost
every $d$-dimensional indecomposable module $M\in {\cal F}(\Delta)$ admits an
almost split
sequence in
${\cal F}(\Delta)$ of the form
$0\rightmap{}M\rightmap{}E\rightmap{}M\rightmap{}0.$
\end{theorem}

  In the quasi-hereditary case, Ziting Zeng has proved in \cite{Z}(6.5) that  the second Brauer-Thrall conjecture holds for ${\cal F}(\Delta)$.  That is,  if ${\cal F}(\Delta)$ is of infinite representation type, it admits
infinitely many non-isomorphic $d$-dimensional indecomposables,
for infinitely many $d\in \hueca{N}$. This makes the preceding result particularly interesting. The problem of existence of some special type of generic modules for the category ${\cal F}(\Delta)$ in the tame representation type case has been addressed in \cite{genqh}.

In the case of admissible homological systems, the proofs of our main results rely on the existence  of an equivalence from
${\cal F}(\Delta)$ to the module category of a suitable \emph{bocs with relations} (or, equivalently, a \emph{ditalgebra with relations}) which is constructed for pre-standardly stratified  algebras in
 \cite{hsb}. This construction generalizes the one given by
 S. Koenig, J. K\"ulshammer, and S. Ovsienko in \cite{KKO}, for the particular, but of cardinal importance, case of quasi-hereditary algebras.

The bocs with relations mentioned before can be presented as a quotient of an  \emph{interlaced weak ditalgebra $({\cal A},I)$}, see \cite{hsb}(5.22), a notion which will be defined rigorously in  \S\ref{S: Interlaced weak dits} and studied in the subsequent ones. Moreover, it will be a quotient of a ${\cal P}$-oriented interlaced weak ditalgebra $({\cal A},I)$ with special geometric features related to ${\cal P}$, which we describe below.

\begin{definition}\label{D: directed bigraph}
A \emph{biquiver} $\hueca{B}$ is a triple
$\hueca{B}=({\cal P},\hueca{B}_0,\hueca{B}_1)$ formed by a finite set ${\cal P}$
of \emph{points}, a finite set $\hueca{B}_0$ of \emph{full arrows}, and a
finite
set $\hueca{B}_1$ of \emph{dashed arrows}. Each \emph{arrow}
$\alpha\in\hueca{B}_0\cup \hueca{B}_1$ has a \emph{starting point}
$s(\alpha)\in {\cal P}$ and a \emph{terminal point} $t(\alpha)\in {\cal P}$.

Given a biquiver $\hueca{B}$ with $n$ points, we consider the
$k$-algebra $R$ defined as the product
$R:=k\times k\times\cdots\times k$ of $n$ copies of the ground field $k$.
Denote by $e_i$ the idempotent of $R$ with 1 in its $i^{th}$ coordinate and
0 in the others.
Then, we consider the vector space $W_0$ (resp. $W_1$) as the vector space with
basis
$\hueca{B}_0$ (resp. $\hueca{B}_1$). Set $W:=W_0\oplus W_1$. If we define
$e_j\alpha e_i=\delta_{i,s(\alpha)}\delta_{t(\alpha),j}\alpha$, we get a
natural structure of $R$-$R$-bimodule
on the space $W$, and $W=W_0\oplus W_1$  is an $R$-$R$-bimodule decomposition.
Then, we have the graded tensor algebra $T=T_R(W)$, with $[T]_0=T_R(W_0)$ and
$[T]_1=[T]_0W_1[T]_0$, which is called \emph{the tensor algebra of the biquiver  $\hueca{B}$}.

Notice that $T$ can be identified with the \emph{path algebra $k\hueca{B}$ of
the biquiver $\hueca{B}$}, with underlying vector space with basis the set of
paths (of any kind of arrows) of $\hueca{B}$ (including one trivial path for
each point $i\in {\cal P}$). Each primitive idempotent $e_i$  of $R$ is
identified with the trivial path at the point $i$.
The subspace $[k\hueca{B}]_u$ of $k\hueca{B}$ of homogeneous elements  of
degree
$u$ consists of the linear combinations of paths containing exactly $u$ dashed arrows.
\end{definition}

\begin{definition}\label{D: biquiver P-orientado}
Let ${\cal P}=({\cal P},\leq)$ be a finite preordered set and $\hueca{B}$ a biquiver with set of points ${\cal P}$, then we say that $\hueca{B}$ is \emph{${\cal P}$-oriented} iff
 \begin{enumerate}
  \item $i\leq j$, whenever there is a dashed arrow from $i$ to $j$;
  \item $\overline{i}<\overline{j}$, whenever there is a solid arrow from $i$ to $j$.
 \end{enumerate}
 A weak ditalgebra ${\cal A}=(T,\delta)$ will be called a \emph{${\cal P}$-oriented weak ditalgebra} iff its underlying tensor algebra $T$ is the tensor algebra of a ${\cal P}$-oriented biquiver.
 \end{definition}

 A  biquiver $\hueca{B}$ is called \emph{directed} iff it admits no oriented
(non-trivial) cycle (composed by any kind of arrows).

 Given  a finite preordered set ${\cal P}$, a biquiver  $\hueca{B}$ is \emph{${\cal P}$-quasi-directed} iff the only (non trivial) oriented
cycles (composed of any kind of arrows) consist of dashed arrows between points in
the same $\sim$ class. In case $\hueca{B}$ is ${\cal P}$-quasi-directed where ${\cal P}$ is a partially ordered set, then the only non trivial oriented
cycles consist of dashed loops. Thus, ${\cal P}$-oriented biquivers are ${\cal P}$-quasi-directed, but not necessarily directed.

Our proofs, for admissible homological systems, rest on the following two theorems. For the sake of precision, let us recall from \cite{BSZ} that a ditalgebra ${\cal Q}=(T,\delta)$  is a pair where $T$ is a differential tensor algebra with differential $\delta$. We denote by $Q:=[T]_0$ the subalgebra of $T$ formed by the  elements with  degree zero and by $V:=[T]_1$ the $Q$-$Q$-subbimodule of $T$ formed by the elements with degree one. The category of ${\cal Q}$-modules, denoted by ${\cal Q}\g\Mod$,  has objects the $Q$-modules and a morphism $f:M\rightmap{}N$ in ${\cal Q}\g\Mod$ is a pair $f=(f^0,f^1)$ with $f^0\in \Hom_k(M,N)$ and $f^1\in \Hom_{Q\g Q}(V,\Hom_k(M,N))$ such that $qf^0(m)=f^0(qm)+f^1(\delta(q))(m)$, for all $q\in Q$ and $m\in M$. The composition of ${\cal Q}\g\Mod$ is defined using again the differential $\delta$, see \cite{BSZ}(2.2). Given $M\in {\cal Q}\g\Mod$, we denote by $\End_{\cal Q}(M)$ the endomorphism algebra of $M$ in ${\cal Q}\g\Mod$.

By definition, the
\emph{right algebra} of ${\cal Q}$ is
$\Gamma:=\End_{\cal Q}(\,_QQ)^{op}$.
We have the inclusion $\phi:Q\rightmap{}\Gamma$, with
$\phi(q)=(\rho_q,0)$,
where $\rho_q$ denotes the right multiplication by $q$ in $Q$.
Thus $\Gamma$ is a $Q\g Q$-bimodule.
We denote by
${\cal I}(\Gamma)$   the
full subcategory of $\Gamma\g\mod$ with objects $M$ of the form
$M\cong \Gamma\otimes_QN$, for some $N\in Q\g\mod$, and call it the category
of \emph{induced modules}. In the following statement we collect some of the main results obtained in \cite{hsb} for pre-standardly stratified algebras,  generalizing those of \cite{KKO} for quasi-hereditary algebras, which play an essential role in our proofs of (\ref{T: main theorem qh-algebas}) and (\ref{T: Crawley para F(Delta)}).

The ditalgebra ${\cal Q}$ appearing in the next theorem is a quotient of a \emph{${\cal P}$-oriented weak ditalgebra ${\cal A}=(T,\delta)$}. Thus,  ${\cal A}$ has a
 ${\cal P}$-quasi-directed biquiver. In the quasi-hereditary case, we furthermore know that it is a directed biquiver.

\begin{theorem}\label{T: KKO}
 Let $k$ be an algebraically closed field and
  $\Lambda$  a finite-dimensional $k$-algebra with an admissible homological system $({\cal P},\leq,\{\Delta_i\}_{ i\in {\cal P}})$. Then,
 \begin{enumerate}
  \item There is a ditalgebra
 ${\cal Q}$, quotient of a ${\cal P}$-oriented  weak ditalgebra,  such that $\Lambda$ is Morita equivalent
 to the right algebra $\Gamma$ of ${\cal Q}$. In this situation,
 the right $Q$-module $\Gamma$ is finite-dimensional,
projective, and the following holds:
 \item Denote by  $\{S_i\}_{i\in {\cal P}}$ the non-isomorphic simple
 $Q$-modules. Then,
 the  algebra $\Gamma$ has the admissible homological system $({\cal P},\leq,\{\Delta'_i\}_{i\in {\cal P}})$, with
 $$\Delta'_i=\Gamma\otimes_QS_i, \hbox{ for each } i\in {\cal P}.$$
 \item There is an exact full and faithful functor
 $F:{\cal Q}\g\mod\rightmap{}\Gamma\g\mod,$ which restricts to
 an equivalence of categories
 $F:{\cal Q}\g\mod\rightmap{}{\cal I}(\Gamma),$ satisfying
 $$F(M)\cong \Gamma\otimes_QM, \hbox{ for each } M\in {\cal
Q}\g\mod.$$
 \item We have ${\cal I}(\Gamma)={\cal F}(\Delta')$. So, the
equivalence $F$  restricts to an equivalence of categories
 $F:{\cal Q}\g\mod\rightmap{}{\cal F}(\Delta'),$
 such that $\Delta'_i=F(S_i)$, for all $i\in {\cal P}$.
 \item There is a Morita equivalence $\Theta:\Lambda\g\Mod\rightmap{}\Gamma\g\Mod$ such that $\Theta(\Delta_i)\cong \Delta_i'$, for all $i\in {\cal P}$.
 \item There is an equivalence of categories $K:{\cal F}(\Delta)\rightmap{}{\cal Q}\g\mod$
which maps  short exact sequences onto conflations
and $K(\Delta_i)\cong S_i$, for each $i\in {\cal P}$.
 \end{enumerate}
\end{theorem}

The preceding statements and their proofs are scattered throughout \cite{hsb}: The construction of ${\cal Q}$ or, equivalently, of the ${\cal P}$-oriented interlaced weak  ditalgebra  $({\cal A}(\Delta),I(\Delta))$,  is done in \cite{hsb}\S5[see (5.22),(13.1),(13.2)]; the fact that the ideal $I(\Delta)$ is triangular is verified in \cite{hsb}(10.10);
for \emph{2}, see \cite{hsb}(13.7); for \emph{3}, see \cite{hsb}(12.11), and, to recall the exact structure of ${\cal Q}\g\mod$, see \cite{BSZ}\S6 or \cite{hsb}\S11; for \emph{4}, see \cite{hsb}(13.8); for \emph{5}, see \cite{hsb}(13.9); item \emph{6} follows from \cite{hsb}(11.12).

\begin{remark}\label{R: F(Delta) tame (wild) iff F(Delta') tame (wild)}
The equivalence $\Theta:\Lambda\g\Mod\rightmap{}\Gamma\g\Mod$ mentioned in the last theorem restricts to an equivalence of categories  $\Theta_\vert:{\cal F}(\Delta)\rightmap{}{\cal F}(\Delta')$ given by a tensor product by some progenerator  and its quasi-inverse $\Theta'_\vert:{\cal F}(\Delta')\rightmap{}{\cal F}(\Delta)$ has the same form. It follows that the category ${\cal F}(\Delta)$ is tame (resp. strictly tame, or wild) iff the category ${\cal F}(\Delta')$ is tame (resp. strictly tame, or wild). The argument that justifies this statement is similar to the one justifying that Morita equivalent finite-dimensional algebras are simultaneously tame (or strictly tame, or wild).
\end{remark}

The word \emph{ditalgebra} used before is an acronym for differential tensor
algebra (also refered to as \emph{semi-free differential graded algebra}). Their study was first introduced by the Kiev School in representation
theory, see \cite{RK} and \cite{R}, and is handled in \cite{KKO} and \cite{K} with the equivalent formulation in terms of
bocses.

The ditalgebra ${\cal Q}$ appearing in the last theorem is a special kind of  \emph{quotient of a ${\cal P}$-oriented  weak  $k$-ditalgebra,} which will be discussed in the
next sections.  We will prove the following statement, which  applies to them.

\begin{theorem}\label{T: main theorem dits} Assume that the ground field $k$ is
algebraically closed and that ${\cal P}$ is a finite preordered set. Then, every   quotient ${\cal A}/J$ of a ${\cal P}$-oriented  triangular
weak
$k$-ditalgebra ${\cal A}=(T,\delta)$,
 by an ideal $J$ of ${\cal A}$ generated by an ideal $I$ of the algebra  $A=[T]_0$  (such that $I$ is contained in the radical of $A$, is interlaced with ${\cal A}$
 and is ${\cal A}$-triangular), is either tame or wild, but
not  both.
\end{theorem}

The strategy for the proof of the preceding theorem will be the following. We work in the general context of triangular interlaced weak ditalgebras $({\cal A},I)$, described in the first sections \S2--\S5, and then we will concentrate on the ${\cal P}$-oriented case.
In this case, we will first reduce the study of the $({\cal A},I)$-modules with  dimension bounded by some natural number $d$, to the corresponding study for stellar interlaced weak ditalgebras, then we reduce this problem to the corresponding problem for seminested ditalgebras and, finally, we apply the known theory which reduces this problem to  the corresponding problem for minimal ditalgebras (developed by Drozd and Crawley-Boevey in \cite{D} and \cite{CB1}).

The proofs of our main results in the case of general homological systems follow from the corresponding statements for admissible homological systems and from the fact, proved by Mendoza, S\'aenz, and Xi in \cite{MSX}, that the category ${\cal F}(\Delta)$ associated to a general homological system is equivalent as an exact category to the category ${\cal F}(\Theta)$ associated to an appropriate admissible homological system. The last section of this article is devoted to this.

The reader is refered to \cite{BSZ} as a general reference for the
terminology not explained here and for a basic introduction to the subject. In particular, we use intensively reduction functors associated with admissible modules, a construction carefully studied in \cite{BSZ}.
In this article, we do not introduce essentially new types of reductions for ditalgebras, but, since we consider weak ditalgebras with relations, we have to
carefully monitor how the ideal of relations is modified by each classical reduction procedure.

We use freely the basic terminology of unital rings and modules employed in \cite{AF}, but we accept the possibility that in a unital ring $A$ we have $1=0$, so $A\g\Mod$ has only trivial modules.

\section{Interlaced weak ditalgebras and modules}\label{S: Interlaced weak dits}

In this section we introduce the notions of ``weak ditalgebras ${\cal A}$'',
of ``interlaced weak ditalgebras $({\cal A},I)$'',
and their ``module category $({\cal A},I)\g\Mod$''. The first one is a trivial
generalization of the notion of ditalgebra and the second one is a useful way
to handle some special kind of ditalgebras with relations
and their module categories.

\begin{definition}\label{D: derivation, differential}
Given a graded $k$-algebra $T=\bigoplus_{i\geq 0}[T]_i$,
a \emph{derivation} on $T$  is a linear map
$\delta:T\rightmap{}T$ such that $\delta([T]_i)\subseteq [T]_{i+1}$, for all
$i$, and such that for any homogeneous elements $a,b\in T$, the following
\emph{Leibniz rule} holds:
$$\delta(ab)=\delta(a)b+(-1)^{\deg(a)}a\delta(b).$$
The derivation $\delta$ is called a \emph{differential} if $\delta^2=0$.
\end{definition}

\begin{definition}\label{D: weak ditalg}
Given a $k$-algebra $A$ and an $A\g A$-bimodule $V$, we have the tensor algebra
$T=T_A(V)$
with canonical grading $[T]_i=V^{\otimes i}$, for $i\in \hueca{N}\cup\{0\}$.
Thus, $[T]_0=A$ and $[T]_1=V$. A \emph{graded tensor algebra} (or
\emph{$t$-algebra} $T$)
is a graded algebra isomorphic to some $T_A(V)$. We often identify them.

A \emph{weak ditalgebra} ${\cal A}$ is a pair ${\cal A}=(T,\delta)$, where $T$
is a graded tensor algebra and $\delta$ is a derivation on $T$. A
\emph{morphism of
weak
ditalgebras} $\eta:{\cal A}\rightmap{}{\cal A}'$ is a morphism of graded
algebras $\eta:T\rightmap{}T'$ such that $\delta'\eta=\eta\delta$.

A weak ditalgebra
$(T,\delta)$ is a \emph{ditalgebra} if $\delta$ is a differential.
\end{definition}

\begin{notation} Given a weak ditalgebra ${\cal A}=(T_A(V),\delta)$, consider
for any pair of $A$-modules $M$ and $N$ the set $U(M,N)$ defined as the
collection of pairs $f=(f^0,f^1)$, with $f^0\in \Hom_k(M,N)$ and
$f^1\in \Hom_{A\g A}(V,\Hom_k(M,N))$ such that, for any $a\in A$ and $m\in M$,
the following holds
$$af^0[m]=f^0[am]+f^1(\delta(a))[m].$$
Given $f^1\in \Hom_{A\g A}(V,\Hom_k(M,N))$  and
$g^1\in \Hom_{A\g A}(V,\Hom_k(N,L))$, we consider the morphism
$g^1*f^1\in \Hom_{A\g A}(V^{\otimes 2},\Hom_k(M,L))$ defined, for any
$\sum_ju_j\otimes v_j\in V\otimes_AV$, by
$$(g^1*f^1)(\sum_ju_j\otimes v_j)=\sum_j g^1(u_j)f^1(v_j):M\rightmap{}L.$$
\end{notation}

\begin{lemma}\label{L: compos de morfismos de U}
With the preceding notations,  assume that  $f=(f^0,f^1)\in U(M,N)$ and
$g=(g^0,g^1)\in U(N,L)$ satisfy that
$(g^1*f^1)(\delta^2(a))=0$, for any $a\in A$, then
$gf:=(g^0f^0,(gf)^1)\in U(M,L)$, where
$$(gf)^1(v):=g^0f^1(v)+g^1(v)f^0+(g^1*f^1)(\delta(v)),\hbox{ for } v\in V.$$
\end{lemma}

\begin{proof} For $a\in A$, $v\in V$, and $m\in M$, we have
$$\begin{matrix}
   (gf)^1(av)[m]&=&g^0f^1(av)[m]+g^1(av)f^0[m]+(g^1*f^1)(\delta(av))[m]\hfill\\
   &=& g^0(af^1(v))[m]+ag^1(v)f^0[m]+(g^1*f^1)(\delta(a)v+a\delta(v))[m]\hfill\\
   &=& ag^0f^1(v)[m]+ag^1(v)f^0[m]+a(g^1*f^1)(\delta(v))[m]\hfill\\
   &=& a(gf)^1(v)[m]\hfill
    \end{matrix}$$
   and
$$\begin{matrix}
(gf)^1(va)[m]&=&g^0f^1(va)[m]+g^1(va)f^0[m]+(g^1*f^1)(\delta(va))[m]\hfill\\
&=&
g^0(f^1(v)a)[m]+(g^1(v)a)(f^0[m])+(g^1*f^1)(\delta(v)a-v\delta(a))[m]\hfill\\
&=&g^0f^1(v)[am]+g^1(v)f^0[am]+(g^1*f^1)(\delta(v))[am]\hfill\\
&=&(((gf)^1(v))a)[m],\hfill\\
  \end{matrix}$$
so $(gf)^1$ is a morphism of $A$-$A$-bimodules. Moreover,
$$\begin{matrix}
   a(gf)^0[m]&=& a(g^0(f^0[m]))\hfill\\
   &=& g^0(af^0[m])+g^1(\delta(a))f^0[m]\hfill\\
   &=& g^0(f^0[am]+f^1(\delta(a))[m])+g^1(\delta(a))f^0[m]\hfill\\
   &=& (gf)^0[am]+(gf)^1(\delta(a))[m]-(g^1*f^1)(\delta^2(a))[m]\hfill\\
  \end{matrix}$$
and, since $(g^1*f^1)(\delta^2(a))=0$, we obtain $gf\in U(M,L)$, as claimed.
\end{proof}

\begin{definition}\label{D: weak ditalg interlaced an ideal}
 Let ${\cal A}=(T,\delta)$ be a weak ditalgebra, with $T=T_A(V)$, and $I$ an
ideal of $A$.
 The weak ditalgebra ${\cal A}$ is called  \emph{interlaced with the ideal $I$}
iff
 \begin{enumerate}
  \item $\delta^2(A)\subseteq I[T]_2+VIV+[T]_2I$, and
  \item $\delta^2(V)\subseteq I[T]_3+VI[T]_2+[T]_2IV+[T]_3I$.
 \end{enumerate}
 Equivalently, $\delta^2(A)$ and $\delta^2(V)$ are contained in the ideal of $T$ generated by $I$.

 The pair $({\cal A},I)$ is called then an \emph{interlaced weak
ditalgebra}.
\end{definition}

\begin{proposition}\label{P: cat de mods de (wdit,I) }
Assume that $({\cal A},I)$ is an interlaced weak ditalgebra.
Then, we construct a category $({\cal A},I)\g\Mod$ as follows:
The class of objects of $({\cal A},I)\g\Mod$ is the class of left $A/I$-modules $M$; given the $A/I$-modules $M$ and $N$,
the set of morphisms from $M$ to $N$ in $({\cal A},I)\g\Mod$ is, by definition,
$\Hom_{({\cal A},I)}(M,N):=U(M,N)$;
finally, given  $f=(f^0,f^1)\in U(M,N)$ and
$g=(g^0,g^1)\in U(N,L)$, their composition $gf$ in this category is defined by
the
formulas in (\ref{L: compos de morfismos de U}).

If $I=0$, then ${\cal A}$ is a ditalgebra, and $({\cal A},{\{0\}})\g\Mod$ coincides with
${\cal A}\g\Mod$, as in \cite{BSZ}\S2.
\end{proposition}

\begin{proof} The composition is well defined by Lemma (\ref{L: compos de
morfismos de U}) and the first
condition on the interlaced weak  ditalgebra $({\cal A},I)$.
Clearly, for each $M\in ({\cal A},I)\g\Mod$, the pair $(1_M,0)$ plays the role
of an identity morphism.
Let us verify the associativity of the composition.

Take $f\in U(M,N)$, $g\in U(N,L)$, and $h\in U(L,K)$. It is clear that
$(h(gf))^0=((hg)f)^0$. Take $v\in V$ and write $\delta(v)=\sum_iv_iv'_i$,
$\delta(v_i)=\sum_ju_{i,j}u'_{i,j}$, and $ \delta(v'_i)=\sum_sw_{i,s}w'_{i,s}$,
then we have
$$\begin{matrix}
   (h(gf))^1(v)&=&h^0(gf)^1(v)+h^1(v)g^0f^0+\sum_ih^1(v_i)(gf)^1(v'_i)\hfill\\
   &=& h^0g^0f^1(v)+h^0g^1(v)f^0+\sum_ih^0g^1(v_i)f^1(v'_i)+h^1(v)g^0f^0\hfill\\
   &&+\sum_ih^1(v_i)g^0f^1(v'_i)+\sum_ih^1(v_i)g^1(v'_i)f^0\hfill\\
   &&+\sum_{i,s}
h^1(v_i)g^1(w_{i,s})f^1(w'_{i,s})\hfill\\
  \end{matrix}$$
  and
$$\begin{matrix}
   ((hg)f)^1(v)&=& (hg)^0f^1(v)+(hg)^1(v)f^0+\sum_i(hg)^1(v_i)f^1(v'_i)\hfill\\
   &=& h^0g^0f^1(v)+h^0g^1(v)f^0+h^1(v)g^0f^0+\sum_ih^1(v_i)g^1(v'_i)f^0\hfill\\

&&+\sum_ih^0g^1(v_i)f^1(v'_i)+\sum_ih^1(v_i)g^0f^1(v'_i)\hfill\\
&&+\sum_{i,j}h^1(u_{i,j}
)g^1(u'_{i,j})f^1(v'_i).\hfill\\
  \end{matrix}$$
The difference
$\Delta=\sum_{i,j}h^1(u_{i,j})g^1(u'_{i,j})f^1(v'_i)-\sum_{i,s}h^1(v_i)g^1(w_{i,
s})f^1(w'_{i,s})$ is zero because
$\delta^2(v)=\delta(\sum_iv_iv'_i)=\sum_{i,j}u_{i,j}u'_{i,j}v'_i-\sum_{i,s}v_iw_
{i,s}w'_{i,s}$
and, by the second condition on the interlaced weak ditalgebra $({\cal A},I)$,
we obtain $\Delta[m]=0$, for $m\in M$.
\end{proof}

Given an interlaced weak ditalgebra $({\cal A},I)$, we will denote
by $({\cal A},I)\g\mod$ the full subcategory of $({\cal A},I)\g\Mod$ formed
by its finite-dimensional objects.

\section{Quotients of weak ditalgebras}

In this section, we extend the concept of ideal for ditalgebras, see \cite{BSZ}\S8, to the more
general
case of weak ditalgebras and we relate the category $({\cal A},I)\g\Mod$,
for an interlaced weak ditalgebra $({\cal A},I)$, to the category
${\cal A}/J\g\Mod$ of modules over the ditalgebra ${\cal A}/J$, where $J$
is the ideal of ${\cal A}$ generated by $I$.

\begin{definition}\label{D: ideal of a w.ditalg}
Assume that ${\cal A}=(T,\delta)$ is a weak ditalgebra, where $T=T_A(V)$. An
\emph{ideal $J$ of ${\cal A}$} is an ideal $J$ of $T$ such that
\begin{enumerate}
 \item $J\cap A$ and $J\cap V$ together generate the ideal $J$ of $T$;
 \item $\delta(J)\subseteq J.$
\end{enumerate}
In this case, $J$ is a homogeneous ideal of $T=T_A(V)$, and is a proper ideal of $A$ iff $J\cap A$ is a proper ideal of $T$. Thus,  $T/J$ is a graded
$k$-algebra. The first condition guarantees that the algebra $T/J$ can be
identified
with the tensor algebra $T_{A/(J\cap A)}(V/(J\cap V))$, see \cite{BSZ}(8.4);
the second one implies
that $\delta$ induces a derivation $\overline{\delta}$ on $T/J$. Thus, the
pair ${\cal A}/J:=(T/J,\overline{\delta})$ is a weak ditalgebra: the
\emph{quotient of the weak ditalgebra ${\cal A}$ by the ideal $J$}. The
canonical projection $\eta:T\rightmap{}T/J$ is a morphism of weak ditalgebras
$\eta:{\cal A}\rightmap{}{\cal A}/J$.
\end{definition}

\begin{definition}\label{D: triangular ideal and balanced ideal}
 Let ${\cal A}=(T,\delta)$ be a weak ditalgebra with $T=T_A(V)$. Assume that $I$ is an ideal of $A$. Then,
\begin{enumerate}
 \item We say that $I$ is an
 \emph{${\cal A}$-triangular ideal of $A$} iff there is a sequence of
$k$-subspaces
 $0=H_0\subseteq H_1\subseteq\cdots\subseteq H_t=I$
 such that
 $\delta(H_i)\subseteq AH_{i-1}V+VH_{i-1}A$,  for all $i\in [1,t].$
 \item We say that $I$ is an \emph{${\cal A}$-balanced ideal} of $A$ iff
$\delta(I)\subseteq IV+VI.$
\end{enumerate}
Clearly, every ${\cal A}$-triangular ideal of $A$ is an ${\cal A}$-balanced
ideal of $A$.
\end{definition}

\begin{remark}\label{R: defs de triangularidad son equivalentes}
  Let ${\cal A}=(T,\delta)$ be a weak ditalgebra with $T=T_A(V)$. Assume that $I$ is an ideal of $A$ and that
  $0=H_0\subseteq H_1\subseteq\cdots\subseteq H_t\subseteq I$
 is a vector space filtration of $I$ such that
 $\delta(H_i)\subseteq AH_{i-1}V+VH_{i-1}A, \hbox{ for all } i\in [1,t],$
 and the ideal of $A$ generated by $H_t$ is $I$,
 then $0=H_0\subseteq H_1\subseteq\cdots\subseteq H_t\subseteq H_{t+1}=I$ is a triangular filtration of $I$ as in the preceding definition. Indeed, the ideal $I$ is generated as a vector space by  elements of the form $ahb$, with $a,b\in A$ and $h\in H_t$. Then,  we have $\delta(ahb)=\delta(a)hb+a\delta(h)b+ah\delta(b)$, where
 $\delta(a)hb\in VH_tA$, $ah\delta(b)\in AH_tV$, and
 $a\delta(h)b\in a(AH_{t-1}V+VH_{t-1}A)b\subseteq AH_tV+VH_tA$. Hence, the definition of a triangular ideal given in (\ref{D: triangular ideal and balanced ideal}) coincides  with the definition of \cite{BSZ}(8.22).
\end{remark}

\begin{lemma}\label{L: sobre ideales de weak dits generated by I}
Assume that ${\cal A}=(T,\delta)$ is a weak ditalgebra, with $T=T_A(V)$,
and that $I$ is some ideal of $A$. Consider the $A$-$A$-subbimodule
$I_V=IV+\delta(I)+VI$ of $V$ and assume that we are in one of the following two
cases:
\begin{enumerate}
\item ${\cal A}$ is interlaced with the ideal $I$
\item $I$ is an ${\cal A}$-balanced ideal.
\end{enumerate}
Then, the ideal $J$ of $T$ generated by $I$ and $I_V$ is an
ideal of the weak ditalgebra ${\cal A}$. In particular, $I=J\cap A$ and
$I_V=J\cap V$.
If $I$ is ${\cal A}$-balanced, then $I_V=IV+VI$.
\end{lemma}

\begin{proof} By \cite{BSZ}(8.3 and 8.5), we have $J\cap A=I$ and $J\cap V=I_V$. Since
$\delta(I)\subseteq I_V\subseteq J$ and
$\delta^2(I)\subseteq J$, by the Leibniz rule, we have $\delta(I_V)\subseteq J$.
Then, using again the Leibniz rule, we get  $\delta(J)\subset J$.
\end{proof}

\begin{definition}
In the preceding situation, we say that $J$ is \emph{the ideal of the weak
ditalgebra ${\cal A}$ generated by the ideal} $I$ of $A$.
\end{definition}

\begin{remark}\label{R: cocientes por ideales balanceados}
 Let ${\cal A}=(T,\delta)=(T_A(V),\delta)$ be a weak ditalgebra
 and $I$ an ideal of $A$. Assume that we are in the situation of
 (\ref{L: sobre ideales de weak dits generated by I}) and let $J$ be the
ideal of ${\cal A}$ generated by $I$.
Write
$\overline{A}=A/I$,
$\overline{V}:=V/(IV+\delta(I)+VI)$ and consider the canonical morphism
$\pi: T\rightmap{}T_{\overline{A}}(\overline{V})$ induced by the
projections $\pi_0:A\rightmap{}\overline{A}$ and
$\pi_1:V\rightmap{}\overline{V}$. From \cite{BSZ}(8.4), we obtain an
isomorphism
$\overline{\eta}: T_{\overline{A}}(\overline{V})\rightmap{}T/J$ such
that $\overline{\eta}\pi=\eta$, where $\eta:{\cal A}\rightmap{}{\cal A}/J$ is
the canonical projection to the quotient weak ditalgebra ${\cal A}/J$. With the
help of this isomorphism, we can transfer the derivation of ${\cal A}/J$
onto a derivation $\overline{\delta}$ on $T_{\overline{A}}(\overline{V})$.
Notice that $\overline{\delta}$ can also be obtained by \cite{BSZ}(1.8) from
the linear maps
$\overline{\delta}_0:\overline{A}\rightmap{}[\overline{T}]_1$ and
$\overline{\delta}_1:\overline{V}\rightmap{}[\overline{T}]_2$ induced by the
composition $\pi\delta$,
when restricted respectively to $A$ and $V$.
We write $\overline{\cal A}:=(\overline{T},\overline{\delta})$ and we identify
it with ${\cal A}/J$.
\end{remark}

\begin{lemma}\label{modules of interlaced ditalgebras}
 Let ${\cal A}=(T,\delta)=(T_A(V),\delta)$ be a weak ditalgebra
 interlaced with an ideal $I$ of $A$, and let $J$ be
the ideal of ${\cal A}$ generated by  $I$.
Then, the quotient ${\cal A}/J$ is a ditalgebra and we have
 $$\begin{matrix}A/I\g\Mod&\rightmap{ \ \ L_{{\cal A}/J} \ \ }&
 {\cal A}/J\g\Mod&\leftmap{ \ \Psi \ }&({\cal A},I)\g\Mod\end{matrix},$$
where ${\cal A}/J\g\Mod$ is the category of ${\cal A}/J$-modules, $L_{{\cal
A}/J}$ is the canonical
embedding defined by  $L_{{\cal
A}/J}(M)=M$, for any $A/I$-module,  and
$L_{{\cal A}/J}(f^0)=(f^0,0)$, for any morphism $f^0$ of $A/I\g \Mod$,
and $\Psi$ is an isomorphism of categories. The composition
$L_{({\cal A},I)}:=\Psi^{-1}L_{{\cal A}/J}$ is also called \emph{the canonical
embedding}.
\end{lemma}

\begin{proof} Since ${\cal A}$ is interlaced with $I$, we have
$\delta^2(T)\subseteq J$,
thus $\overline{\delta}^2=0$ and ${\cal A}/J$ is a ditalgebra.

 By definition, $\Psi$ is the identity on objects. It maps a morphism
$(f^0,f^1)\in U(M,N)$ onto the pair $\Psi(f^0,f^1)=(f^0,\overline{f}^1)$, where
$\overline{f}^1:\overline{V}\rightmap{}\Hom_k(M,N)$ is the morphism induced by
$f^1:V\rightmap{}\Hom_k(M,N)$, which exists because $IM=0$ and $IN=0$, so
 $f^1(IV+\delta(I)+VI)=0$.
 Consider the canonical projections
$\pi_0:A\rightmap{}\overline{A}$
and $\pi_1:V\rightmap{}\overline{V}$, write
$\overline{v}:=\pi_1(v)$ and
$\overline{a}=\pi_0(a)$,
for $v\in V$ and  $a\in A$.
Then, $(f^0,f^1)\in U(M,N)$ iff
$af^0(m)=f^0(am)+f^1(\delta(a))(m)$, for $a\in A$ and $m\in M$, or, equivalently
$\overline{a}f^0(m)=f^0(\overline{a}m)+\overline{f}^1(\overline{\delta}
(\overline{a}))(m)$,
for $\overline{a}\in \overline{A}$ and $m\in M$, that is iff
$(f^0,\overline{f}^1)\in \Hom_{{\cal A}/J}(M,N)$.
Similarly, given $(f^0,f^1)\in U(M,N)$ and $(g^0,g^1)\in U(N,L)$, and $v\in V$
with $\delta(v)=\sum_iv_iu_i$,
the  expressions  $g^0\overline{f}^1(\overline{v})+
\overline{g}^1(\overline{v})f^0+\sum_i\overline{g}^1(\overline{v}_i)\overline{f}
^1(\overline{u}_i)$
and $g^0f^1(v)+g^1(v)f^0+\sum_ig^1(v_i)f^1(u_i)$ coincide.
Thus $\Psi$ preserves compositions, and clearly it preserves identities.
\end{proof}

Notice that the preceding argument provides an alternative proof, to the
one given before,
of the fact that $({\cal A},I)\g\Mod$ is a category.

\begin{remark}\label{R: balanced vs interlaced}
Assume that ${\cal A}=(T_A(V),\delta)$ is a weak ditalgebra, that $I$ is an
${\cal A}$-balanced
ideal of $A$, and denote by $J$ the ideal of ${\cal A}$ generated by $I$. Then,
the following are equivalent statements:
\begin{enumerate}
 \item ${\cal A}$ is
 interlaced with $I$;
 \item $\delta^2(A)\subseteq J$ and $\delta^2(V)\subseteq J$;
 \item $\delta^2(T)\subseteq J$.
\end{enumerate}
For instance, assuming 2, we know that $\delta^2(A)\subseteq TIT+TI_VT$, so $\delta^2(A)\subseteq V^2I+VIV+IV^2+I_VV+VI_V$. But $I_V=VI+IV$,  because $I$ is ${\cal A}$-balanced, so $\delta^2(A)\subseteq V^2I+VIV+IV^2$; similarly, we have that $\delta^2(V)\subseteq V^3I+V^2IV+VIV^2+IV^3$, and 1 holds.
\end{remark}

\begin{definition}\label{D: morphism of interlaced ditalgebras}
A \emph{morphism $\phi:({\cal A},I)\rightmap{}({\cal A}',I')$ of interlaced
weak ditalgebras}
is a morphism of weak ditalgebras $\phi:{\cal A}\rightmap{}{\cal A}'$ such that
$\phi(I)\subseteq I'$.
\end{definition}

The following is easy to show (see \cite{BSZ}(2.4)).

\begin{lemma}\label{L: funtores inducidos por restriccion}
 Any morphism of interlaced weak ditalgebras $\phi:({\cal
A},I)\rightmap{}({\cal A}',I')$ induces, by restriction, a functor
 $F_\phi:({\cal A}',I')\g\Mod\rightmap{}({\cal A},I)\g\Mod$. Consider the
restrictions
 $\phi_0:A\rightmap{}A'$ and $\phi_1:V\rightmap{}V'$ of $\phi$. Then, for $M\in
({\cal A}',I')\g\Mod$, the $A$-module $F_\phi(M)$ is  obtained by restriction
of scalars through the map $\phi_0$; for $f=(f^0,f^1)\in \Hom_{({\cal
A}',I')}(M,N)$, we have $F_\phi(f)=(f^0,f^1\phi_1)$.

If $\phi$ is surjective, then $F_\phi$ is faithful. Moreover, if
$\psi:({\cal A}',I')\rightmap{}({\cal A}'',I'')$ is another morphism of
interlaced weak ditalgebras,
then $F_{\psi\phi}=F_{\phi}F_\psi$.
\end{lemma}

\begin{remark} The equivalent notions of \emph{normal bocs} ${\cal B}$ and \emph{ditalgebra} ${\cal A}$, and their categories of modules, are discussed in detail in  \cite{BSZ}\S3, see also \cite{K}\S4. There, we can see that the coassociativity of the comultiplication $\mu$ of a normal bocs ${\cal B}=(A,U,\mu,\epsilon)$ implies that the derivation $\delta$ of the corrresponding ditalgebra ${\cal A}=(T_A(V),\delta)$ indeed satisfies $\delta^2=0$.

If $({\cal A},I)$ is an interlaced weak ditalgebra with underlying weak ditalgebra ${\cal A}=(T,\delta)$, we have the quotient ditalgebra
$\overline{\cal A}=(T/J,\overline{\delta})$, where $J$ is the ideal of ${\cal A}$ generated by $I$ and $\overline{\delta}$ is the differential induced by the derivation $\delta$ on $T/J$, as remarked before in (\ref{R: cocientes por ideales balanceados}).

Now, given an interlaced weak ditalgebra $({\cal A},I)$, one could consider the corresponding \emph{interlaced weak normal bocs} $({\cal B},I)$, where $I$ is the given ideal of $A=[T]_0$ and ${\cal B}=(A,U,\mu,\epsilon)$ is constructed as in \cite{BSZ}(3.3), using the derivation $\delta$ of ${\cal A}$. This would be a \emph{weak normal bocs} since the comultiplication $\mu$ is not necessarily coassociative. Nevertheless, we would have a normal bocs $\overline{\cal B}=(\overline{A},\overline{U},\overline{\mu},\overline{\epsilon})$, where $\overline{A}=A/I$, $\overline{U}=U/(IU+UI)$, and $\overline{\mu}$, $\overline{\epsilon}$ are induced by $\mu$,  $\epsilon$, respectively. The normal bocs $\overline{\cal B}$ corresponds to the ditalgebra $\overline{\cal A}$.

We find it more desirable to work with a derivation $\delta$ (which does not necessarily satisfies $\delta^2=0$) than to deal with a not necessarily coassociative comultiplication.
\end{remark}

\section{Layered weak ditalgebras}

In this section we consider some sufficient conditions on a ditalgebra with
relations
$\widehat{\cal A}=(\widehat{T},\widehat{\delta})$ which allow us to lift its differential
$\widehat{\delta}$
to a derivation of an interlaced weak ditalgebra $({\cal A},I)$ in such a way
that
$\widehat{\cal A}={\cal A}/J$, where $J$ is the ideal of ${\cal A}$ generated by
an ${\cal A}$-balanced ideal $I$ of $A$.

\begin{definition}\label{D: layered weak ditalgebra}
A $t$-algebra $T$ has \emph{layer $(R,W)$} iff $R$ is a $k$-algebra and $W$ an
$R$-$R$-bimodule equipped with an $R$-$R$-bimodule
decomposition $W_0\oplus W_1$ such that
 $R\cup W_0\subseteq [T]_0$, $W_1\subseteq [T]_1$, and $T$ is freely generated by the
pair $(R,W)$, see \cite{BSZ}(4.1).

 In this case, we have isomorphisms of algebras $T\cong T_R(W)$ and $A\cong
T_R(W_0)$,  and an isomorphism of $A$-$A$-bimodules
$V\cong A\otimes_R W_1\otimes_RA$ which we shall
consider as  identifications.

A weak ditalgebra ${\cal A}=(T,\delta)$ has \emph{layer $(R,W)$} iff
 $T$ admits the layer $(R,W)$ and, moreover, $\delta(R)=0$.
\end{definition}

The following statement can be proved as in \cite{BSZ}(4.4).

\begin{lemma}\label{L: construyendo diferenciales}
Assume that $T$ is a $t$-algebra with layer $(R,W)$.
 Suppose that $\delta_0:W_0\rightmap{}[T]_1$ and $\delta_1:W_1\rightmap{}[T]_2$
are morphisms of $R$-$R$-bimodules. Then, there is a unique derivation
$\delta:T\rightmap{}T$, extending $\delta_0$ and $\delta_1$, such that ${\cal
A}=(T,\delta)$ is a weak ditalgebra with layer
$(R,W)$.
\end{lemma}

The following lemma is probably known, but we did not find a reference.

\begin{lemma}\label{L: tecnico de nucleos}
Let $W_0$ and $W_1$ be  $R$-$R$-bimodules, where $R$ is a semisimple
$k$-algebra,
let $I$ be an ideal of $A:=T_R(W_0)$, and set  $\widehat{A}:=A/I$.
Consider the canonical projection $\pi:A\rightmap{}\widehat{A}$.
Then, the kernel of the morphism
$$\pi\otimes 1\otimes\pi:A\otimes_R W_1\otimes_R
A\rightmap{}\widehat{A}\otimes_RW_1\otimes_R\widehat{A}$$
is $I\otimes_R W_1\otimes_R A+A\otimes_R W_1\otimes_R I$.
\end{lemma}

\begin{proof} By assumption, we have the exact sequences
$$0\rightmap{}I\otimes_RW_1\otimes_RA\rightmap{ \ i\otimes 1\otimes 1 \ }
A\otimes_RW_1\otimes_RA\rightmap{ \ \pi\otimes 1\otimes 1 \
}\widehat{A}\otimes_R W_1\otimes_RA\rightmap{}0$$
and
$$0\rightmap{}\widehat{A}\otimes_RW_1\otimes_RI\rightmap{ \ 1\otimes 1\otimes i \ }
\widehat{A}\otimes_RW_1\otimes_RA\rightmap{ \ 1\otimes 1\otimes \pi \
}\widehat{A}\otimes_R W_1\otimes_R\widehat{A}\rightmap{}0.$$
 Thus, if $z\in A\otimes_RW_1\otimes_RA$ is such that $(\pi\otimes
1\otimes\pi)[z]=0$, we obtain that
 $$(1\otimes 1\otimes\pi)(\pi\otimes 1\otimes 1)[z]=0.$$
 Thus, $z':=(\pi\otimes
1\otimes 1)[z]\in \widehat{A}\otimes_RW_1\otimes_RI\subseteq
\widehat{A}\otimes_RW_1\otimes_R A$. Moreover, the surjective map
$\pi\otimes 1\otimes
1_I:A\otimes_RW_1\otimes_RI\rightmap{}\widehat{A}\otimes_RW_1\otimes_RI$
maps some $z''\in A\otimes_RW_1\otimes_RI\subseteq A\otimes_RW_1\otimes_RA$
onto
$z'$.
Therefore, $(\pi\otimes 1\otimes 1)[z'']=z'$. Hence, $(\pi\otimes 1\otimes
1)(z-z'')=z'-z'=0$,
and $z'''=z-z''\in I\otimes_RW_1\otimes_RA$, and $z=z'''+z''\in
I\otimes_R W_1\otimes_R A+A\otimes_R W_1\otimes_R I$.
\end{proof}

\begin{proposition}\label{P: lifting of differential}
Let $W_0$ and $W_1$ be $R$-$R$-bimodules, where $R$ is a semisimple $k$-algebra.
Define $A:=T_R(W_0)$ and assume that $I$ is an ideal of $A$.
Set $\widehat{A}:=A/I$,
$\widehat{V}:=\widehat{A}\otimes_RW_1\otimes_R\widehat{A}$,
and
$\widehat{T}:=T_{\widehat{A}}(\widehat{V})$. Assume that
$\widehat{\cal A}=(\widehat{T},\widehat{\delta})$
is a ditalgebra with $\widehat{\delta}([R+I]/I)=0$. Then,
 there is a derivation $\delta:T\rightmap{}T$,
 where $T=T_A(V)$, with $V:=A\otimes_RW_1\otimes_RA$, such that:
 \begin{enumerate}
  \item The pair ${\cal A}=(T,\delta)$ is a weak ditalgebra with layer $(R,W)$;
  \item The canonical maps $\pi_0:A\rightmap{}\widehat{A}$ and $\pi_1:=\pi_0\otimes
1\otimes \pi_0:V\rightmap{}\widehat{V}$
induce a morphism of weak ditalgebras $\pi:{\cal A}\rightmap{}\widehat{\cal A}$;
\item The kernel of $\pi_1$ is precisely $\Ker\,{\pi_1}=IV+VI$;
  \item The ideal $I$ of $A$ is ${\cal A}$-balanced and $\Ker\,\pi$ coincides
with the ideal $J$ of ${\cal A}$ generated by $I$. The weak ditalgebra ${\cal
A}$ is interlaced with $I$ and  $\widehat{\cal A}\cong {\cal A}/J$.
 \end{enumerate}
\end{proposition}

\begin{proof}
 It is clear that $\pi_0$ and $\pi_1$ induce a surjective morphism
 of $k$-algebras $\pi:T\rightmap{}\widehat{T}$. Since $\widehat{\delta}([R+I]/I)=0$, we
obtain that
 $\widehat{\delta}$ is a morphism of $R$-$R$-bimodules.
 Since $k$ is perfect, $R\otimes_kR$
is also semisimple and
 $W_0$ and $W_1$ are projective $R$-$R$-bimodules. Thus, we have the following
commutative diagrams of
$R$-$R$-bimodules
 $$\begin{matrix}
    W_0&\rightmap{\pi_\vert}&\widehat{A}&&\\
    \shortlmapdown{\delta_0}&&\shortlmapdown{\widehat{\delta}}&&\\
    V&\rightmap{\pi_\vert}&\widehat{V}&\rightmap{}&0\\
   \end{matrix}\hbox{ \ \ \  \ \ \ }
   \begin{matrix}
    W_1&\rightmap{\pi_\vert}&\widehat{V}&&\\
    \shortlmapdown{\delta_1}&&\shortlmapdown{\widehat{\delta}}&&\\
    V\otimes_AV&\rightmap{\pi_\vert}&\widehat{V}\otimes_{\widehat{A}}\widehat{V}&\rightmap{}&0.\\
   \end{matrix}$$
  Then, by (\ref{L: construyendo
diferenciales}), we can extend
   these maps to a derivation $\delta:T\rightmap{}T$ in such a way that (1)
and (2) are satisfied.
   Item (3) follows from (\ref{L: tecnico de nucleos}). The ideal $I$ of $A$ is
${\cal A}$-balanced because
   $\pi \delta(I)=\widehat{\delta}\pi(I)=0$, so $\delta(I)\subseteq \Ker\,\pi_1$.
   From (\ref{L: sobre ideales de weak dits generated by I}), we know that $J$
   is an ideal of ${\cal A}$.

   We have $\overline{A}=A/I=\widehat{A}$ and, by (3), the map $\pi_1$ induces an
isomorphism
   $\overline{\pi}_1:\overline{V}=V/(IV+VI)\rightmap{}\widehat{V}$, see (\ref{R: cocientes por ideales balanceados}). They determine
an
isomorphism of graded algebras
$\overline{\pi}:\overline{T}=T_{\overline{A}}(\overline{V})\rightmap{}
\widehat{T}=T_{\widehat{A}}
(\widehat{V})$. Thus $\pi$ essentially
 coincides with the canonical projection $T\rightmap{}T/J$, once we
have made the
   identification described in (\ref{R: cocientes por ideales balanceados}).
Thus, the
   kernel of $\pi$ is $J$ and $\widehat{\cal A}\cong {\cal A}/J$. Finally,
   $\pi\delta^2(T)=\widehat{\delta}^2\pi(T)=0$, so $\delta^2(T)\subseteq J$
   and,  from (\ref{R: balanced vs interlaced}),
   ${\cal A}$ is interlaced with $I$.
\end{proof}

\section{Triangular interlaced weak ditalgebras}

In this section we give a brief summary of the basic properties of the category
of $({\cal A},I)$-modules which follow from triangularity  conditions on the
layer of $({\cal A},I)$. We transfer terminology and basic properties from
the case of triangular ditalgebras, see \cite{BSZ}\S5.

\begin{definition}\label{D: triangular weak ditalg}
Assume that ${\cal A}=(T,\delta)$ is a weak ditalgebra with layer $(R,W)$.
We say that this
layer is \emph{triangular} if
\begin{enumerate}
 \item There is a filtration of $R$-$R$-subbimodules
$0=W_0^0\subseteq W_0^1\subseteq \cdots \subseteq W_0^r=W_0$ such that
$\delta(W_0^{i+1})\subseteq A_iW_1A_i$, for all $i\in [0,r-1]$,
where $A_i$ denotes the $R$-subalgebra $A$ generated by $W_0^i$.
\item There is a filtration of $R$-$R$-subbimodules
$0=W_1^0\subseteq W_1^1\subseteq \cdots \subseteq W_1^s=W_1$ such that
$\delta(W_1^{i+1})\subseteq AW_1^iAW_1^iA$, for all $i\in [0,s-1]$.
\end{enumerate}
The layer is called \emph{additive triangular} iff each $W_0^i$ is a direct
summand of $W_0^{i+1}$, for all $i$. ${\cal A}$ is called
\emph{additive triangular} if it has an additive triangular layer.
\end{definition}

\begin{definition}\label{D: triangular interlaced weak ditalgebras}
 Assume that $({\cal A},I)$ is an interlaced weak ditalgebra with layer
$(R,W)$, then
 $({\cal A},I)$ is called a \emph{triangular  interlaced
weak ditalgebra} iff
$(R,W)$ is triangular, as in the last definition, and $I$ is
an ${\cal
A}$-triangular ideal of $A$, as in (\ref{D: triangular ideal and balanced
ideal}).
Given a preordered set ${\cal P}$, a triangular interlaced weak ditalgebra $({\cal A},I)$ will be called
${\cal P}$-\emph{oriented} iff the underlying weak ditalgebra ${\cal A}$
is ${\cal P}$-oriented, see (\ref{D: biquiver P-orientado}).
\end{definition}

\begin{proposition}\label{P: lema de Roiter basico}
 Let $({\cal A},I)$ be a triangular interlaced weak ditalgebra with layer
$(R,W)$ and adopt the notation of (\ref{D: triangular weak ditalg}).
Assume that, for all $i\in [1,r]$,  the algebra $A_i$ is freely generated
by the pair $(R,W_0^i)$ and that the product map
$$A_i\otimes_RW_1\otimes_RA_i\rightmap{}A_iW_1A_i$$ is an isomorphism  (this is
the case if the layer of ${\cal A}$ is additive triangular).
Suppose that $M$ and $N$ are $R$-modules, $f^0\in \Hom_{R}(M,N)$ is an
isomorphism, and $f^1\in \Hom_{R\g R}(W_1,\Hom_k(M,N))$. Then,
\begin{enumerate}
 \item If $N$ is an $A$-module with underlying $R$-module $N$ and $IN=0$,
 then there is an $A$-module structure on $M$, with underlying $R$-module $M$,
 such that $IM=0$ and $(f^0,f^1)\in \Hom_{({\cal A},I)}(M,N)$.
 \item If $M$ is an $A$-module with underlying $R$-module $M$ and $IM=0$,
 then there is an $A$-module structure on $N$, with underlying $R$-module $N$,
 such that $IN=0$ and $(f^0,f^1)\in \Hom_{({\cal A},I)}(M,N)$.
\end{enumerate}
\end{proposition}

\begin{proof} This is similar to the proof of \cite{BSZ}(5.3). For instance,
for (1) we
take a morphism of $R$-modules $g^0:N\rightmap{}M$ with $f^0g^0=1_N$.
Using the triangular filtration of $W_0$, we construct inductively morphisms
$\phi_j:W_0^j\otimes_R M\rightmap{}M$ of $R$-modules satisfying
$\phi_j(w\otimes m)=g^0[wf^0(m)]-g^0[f^1(\delta(w))(m)]$, for $w\in W_0^j$
  and $m\in M$, where the extension of $f^1$ to a morphism
  $A_{j-1}W_1A_{j-1}\rightmap{}\Hom_k(M,N)$ is denoted with the
  same symbol $f^1$. Then, $\phi_r$ gives to $M$ a structure of an $A$-module
  and from its recipe we derive the formula
  $$f^0(wm)=f^0g^0(wf^0(m)-f^1(\delta(w))[m])=wf^0(m)-f^1(\delta(w))[m],$$
  for any $w\in W_0$. Then the same identity holds for $w\in A$.

  Now, adopt the notation of (\ref{D: triangular ideal and balanced ideal}).
  So, given $a_1\in H_1$, we get $f^0[a_1m]=a_1f^0[m]-f^1(0)[m]=0$,
  and the injectivity of $f^0$ implies that $a_1m=0$. If we assume that
$a_jm=0$, for
  $a_j\in H_j$ and $m\in M$, then for $a_{j+1}\in H_{j+1}$ we have
  $f^0(a_{j+1}m)=a_{j+1}f^0(m)-f^1(\delta(a_{j+1}))[m]=0$,
  because $\delta(a_{j+1})\in AH_jV+VH_jA$. Thus, $IM=0$.
\end{proof}

\begin{definition}\label{D: Roiter weak ditalgebra}
A triangular interlaced weak ditalgebra $({\cal A},I)$ with layer $(R,W)$  is
called a \emph{Roiter interlaced weak ditalgebra} iff the following
property is satisfied: for any isomorphism $f^0$ of $R$-modules
$f^0:M\rightmap{}N$ and any $f^1\in \Hom_{R\g R}(W_1,\Hom_k(M,N))$, if one of
$M$ or $N$ has a structure of a left $A/I$-module, then the other one admits also
a structure of a left $A/I$-module such that $(f^0,f^1)\in \Hom_{({\cal
A},I)}(M,N)$.
\end{definition}

Notice that, by (\ref{P: lema de Roiter basico}), any ${\cal P}$-oriented  triangular interlaced weak ditalgebra
$({\cal A},I)$ is a Roiter interlaced weak ditalgebra.

The following three statements, and their proofs, are similar to
(5.7), (5.8), and (5.12) of \cite{BSZ}. With a slightly different language, these results are proved in \cite{KKO} for the case of directed interlaced weak ditalgebras, see the following section.

\begin{proposition}\label{P: lema de Roiter con sec y retr}
  Let $({\cal A},I)$ be a Roiter interlaced weak ditalgebra with layer $(R,W)$.
  Suppose that
  $f=(f^0,f^1)\in \Hom_{({\cal A},I)}(M,N)$. Then,
  \begin{enumerate}
   \item If $f^0$ is a retraction in $R\g\Mod$, then there is a morphism
$h:M'\rightmap{}M$
in $({\cal A},I)\g\Mod$ such that $h^0$ is an isomorphism and $(fh)^1=0$;
  \item If $f^0$ is a section in $R\g\Mod$, then there is a morphism
$g:N\rightmap{}N'$
in $({\cal A},I)\g\Mod$ such that $g^0$ is an isomorphism and $(gf)^1=0$.
  \end{enumerate}
\end{proposition}

\begin{corollary}\label{C: lema de Roiter}
  Let $({\cal A},I)$ be a Roiter interlaced weak ditalgebra and suppose
that
  $f$ is a morphism in $({\cal A},I)\g\Mod$. Then, $f$ is an isomorphism if and
only if $f^0$ is bijective.
\end{corollary}

\begin{proposition}\label{P: idempotents split}
   Let $({\cal A},I)$ be a Roiter interlaced weak ditalgebra, then
idempotents split in $({\cal A},I)\g\Mod$.
   That is, for any idempotent $e\in \End_{({\cal A},I)}(M)$, there is an
isomorphism $h:M_1\oplus M_2\rightmap{}M$ in
   $({\cal A},I)\g\Mod$ such that
   $$h^{-1}eh=\begin{pmatrix}
            1_{M_1}&0\\ 0&0\\
           \end{pmatrix}.$$
\end{proposition}

\section{Reductions and wildness}

In this section, we adapt standard reductions for ditalgebras, as presented in \cite{BSZ}, to the case of triangular interlaced weak
ditalgebras. By the end of this section, we relate them to the notion of
wildness.

\begin{lemma}\label{L: context of reduction by a quotient}
 Let $({\cal A},I)$ be a triangular interlaced weak ditalgebra, where ${\cal
A}=(T,\delta)$ admits the triangular layer $(R,W)$. Assume that
we have a surjective morphism of $k$-algebras $\phi_{\flat}:R\rightmap{}R'$,
$R'$-$R'$-bimodules $W'_0$ and $W'_1$; and
surjective morphisms of $R$-$R$-bimodules $\phi_r:W_r\rightmap{}W_r'$,
for $r\in \{0,1\}$. Consider the layered $t$-algebra $T':=T_{R'}(W'_0\oplus
W'_1)$. Then, the algebra
$A':=T_{R'}(W'_0)$ can be identified with $[T']_0$ and
$V':=A'\otimes_{R'}W'_1\otimes_{R'}A'$ can be identified with $[T']_1$.
Consider also the morphism of graded $k$-algebras $\phi:T\rightmap{}T'$
determined by $\phi_{\flat}$, $\phi_0$, and $\phi_1$. Finally, assume that we
have commutative diagrams
$$\begin{matrix}
  W_0&\rightmap{\delta}&V\\
  \shortlmapdown{\phi_0}&&\shortlmapdown{\phi_{\vert V}}\\
  W'_0&\rightmap{\delta'_0}&V'\\
  \end{matrix}
  \hbox{ \hskip2cm }
\begin{matrix}
  W_1&\rightmap{\delta}&[T]_2\\
  \shortlmapdown{\phi_1}&&\shortlmapdown{\phi_{\vert [T]_2}}\\
  W'_0&\rightmap{\delta'_1}&[T']_2\\
  \end{matrix} $$
  where $\delta'_0$ and $\delta'_1$ are morphisms of $R'$-$R'$-bimodules.
  Then,
  \begin{enumerate}
   \item There is a derivation $\delta':T'\rightmap{}T'$ extending
$\delta'_0$ and $\delta'_1$ with $\delta'\phi=\phi\delta$;
\item ${\cal A}'=(T',\delta')$ is a weak ditalgebra interlaced with the ideal
$I':=\phi(I)$ of $A'$;
\item $({\cal A}',I')$ is a triangular interlaced weak ditalgebra with
triangular layer
$(R',W'_0\oplus W'_1)$;
\item $\phi:({\cal A},I)\rightmap{}({\cal A}',I')$ is a morphism of
interlaced weak ditalgebras and the induced functor
$F_\phi:({\cal A}',I')\g\Mod\rightmap{}({\cal A},I)\g\Mod$ is faithful and
preserves dimension;
\item The image of $F_\phi$ consists of the modules in $({\cal A},I)\g\Mod$
annihilated by the ideal $K_0:=\Ker(\phi_A:A\rightmap{}A')$ of $A$;
\item Set $K_1:=\Ker(\phi_V:V\rightmap{}V')$ and assume that
$K_1=K_0V+VK_0+\delta(K_0)$, then the functor $F_\phi$ is full;
\item If $K=\Ker\phi$, $K_0=\Ker\phi_A$, and $K_1=\Ker\phi_V$, as before, we have:
\begin{enumerate}
\item
$K_0=A\Ker\phi_{\flat}A+A\Ker\phi_0A$;
\item  $K=T\Ker\phi_{\flat}T+T\Ker\phi_0T+T\Ker\phi_1T,$
and the ideal $K$ of $T$ is generated by $K_0$ and $K_1$.
\end{enumerate}
\end{enumerate}
\end{lemma}

\begin{proof} The existence of $\phi$ follows from the universal property of
the tensor algebra $T$.
(1) follows from (\ref{L: construyendo diferenciales}) and the Leibniz rule.
(2): Since $I$ is triangular,
it generates an ideal $J$ of ${\cal A}$, using the surjectivity of $\phi$,
we get that $\phi(I)$ generates the ideal $\phi(J)$ of ${\cal A}'$. Thus
$(\delta')^2(T')=\phi(\delta^2(T))\subseteq \phi(J)$, and ${\cal A}'$ is
interlaced with $I'$.
(3): It is easy to see that the image under $\phi$ of the filtrations provided
by the triangularity of
the layer $(R,W_0\oplus W_1)$ of ${\cal A}$ give triangular filtrations for the
layer $(R',W'_0\oplus W'_1)$ of ${\cal A}'$.
 Similarly, the image under
$\phi$ of the filtration provided by the ${\cal A}$-triangularity of the ideal
$I$ gives an ${\cal A}'$-triangular filtration for $I'$.
(4): Clearly, $F_\phi$ is
faithful and preserves dimension.

(5) Notice that for $M\in ({\cal A}',I')\g\Mod$, we have $K_0F_\phi(M)=0$.
If $N\in ({\cal A},I)\g\Mod$ is such that $K_0N=0$, then $N$ has a canonical
structure of an $A'$-module, let us denote such $A'$-module by $N'$,
thus $N$ is the $A$-module obtained from $N'$ by restriction through
$\phi:A\rightmap{}A'$.
Moreover, $I'N'=0$ because $IN=0$ and $I'=\phi(I)$, thus $F_\phi(N')=N$.

(6) Assume that $(f^0,f^1)\in \Hom_{({\cal A},I)}(F_\phi(M),F_\phi(N))$. Take
$a\in K_0$, then
from the identity $af^0(m)=f^0(am)+f^1(\delta(a))[m]$ for each $m\in M$, we get
$f^1(\delta(a))=0$. Thus $f^1(K_0V+VK_0+\delta(K_0))=0$. The morphism of
$A$-$A$-bimodules
$f^1:V\rightmap{}\Hom_k(F_\phi(M),F_\phi(N))$ induces
 a morphism of $A$-$A$-bimodules
$\overline{f}^1: V/(K_0V+VK_0+\delta(K_0))\rightmap{}\Hom_k(M,N)$.
Let us denote by $\hat{\phi}:V/(K_0V+VK_0+\delta(K_0))\rightmap{}V'$ the
isomorphism
given by our assumption on $K_1$. Thus, $g^1:=\overline{f}^1\hat{\phi}^{-1}
:V'\rightmap{}\Hom_k(F_\phi(M),F_\phi(N))$ is a morphism of $A$-$A$-bimodules,
which is
also a morphism
of $A'$-$A'$-bimodules $g^1:V'\rightmap{}\Hom_k(M,N)$ with $g^1\phi_V=f^1$.
Given $a\in A$, we have
$$\begin{matrix}\phi(a)f^0(m)=af^0(m)&=&f^0(am)+f^1(\delta(a))[m]\hfill\\
&=&f^0(\phi(a)m)+g^1(\phi(\delta(a)))[m]\hfill\\
&=&
f^0(\phi(a)m)+g^1(\delta'(\phi(a)))[m].\hfill\\
\end{matrix}$$
Thus $(f^0,g^1)\in \Hom_{({\cal A}',I')}(M,N)$ and $F_\phi(f^0,g^1)=(f^0,f^1)$.

(7): We have that $\Ker\phi_{\flat}$ is an ideal of $R$ and
$\Ker\phi_0$ is an $R$-$R$-subbimodule of $W_0$ such that
$\Ker\phi_{\flat}\Ker\phi_0+\Ker\phi_0\Ker\phi_{\flat}\subseteq
\Ker\phi_0$.
Then, by \cite{BSZ}(8.3), $N_0:=A\Ker\phi_{\flat}A+A\Ker\phi_0A$ is an ideal
of $A=T_R(W_0)$ which is $(R,W_0)$-compatible. Then, by \cite{BSZ}(8.4), we have
a
commutative square
$$\begin{matrix}
   A&\longrightmap{\nu_A}&A/N_0\\
   \parallel&&\shortlmapdown{\overline{\phi_A}}\\
   T_R(W_0)&\begin{matrix}\rightmap{\phi_A}&A'=T_{R'}(W_0')&
   \cong\end{matrix}&T_{R/\Ker\phi_{\flat}}(W_0/\Ker\phi_0),\\
  \end{matrix}$$
where $\nu_A$ is the canonical projection and $\overline{\phi_A}$
is an isomorphism. So $N_0=K_0$.

Similarly, since $\Ker\phi_{\flat}W+W\Ker\phi_{\flat}\subseteq
\Ker\,\phi_0+\Ker\,\phi_1$, we have that
$N:=T\Ker\phi_{\flat}T+T\Ker[\phi_0\oplus\phi_1]T$ is an ideal
of $T=T_R(W)$ which is $(R,W)$-compatible. Then, we have a commutative square
$$\begin{matrix}
   T&\longrightmap{\nu_T}&T/N\\
   \parallel&&\shortlmapdown{\overline{\phi}}\\
   T_R(W)&\begin{matrix}\rightmap{\phi}&T_{R'}(W')&
   \cong\end{matrix}&T_{R/\Ker\phi_{\flat}}([W_0\oplus
W_1]/\Ker[\phi_0\oplus\phi_1]),\\
  \end{matrix}$$
where $\overline{\phi}$
is an isomorphism. So $N=K$, and $K$ is generated by $K_0$
and $K_1$.
\end{proof}

In the following construction, we say that \emph{$({\cal A}^d,I^d)$ is obtained from $({\cal A},I)$ by deletion of the idempotent $1-e$}.

\begin{proposition}[deletion of idempotents]\label{P: idemp deletion}
Assume that $({\cal A},I)$ is a triangular interlaced weak ditalgebra with
layer $(R,W)$.
Assume that $e\in R$ is a central non-trivial idempotent  of $R$. Consider the canonical
projections
$\phi_{\flat}:R\rightmap{}eRe$, and $\phi_r:W_r\rightmap{}eW_re$, for $r\in
\{0,1\}$.
Set $T^d:=T_{eRe}(eWe)$. Then, as in the first paragraph of the last lemma, we
have
a morphism of graded $k$-algebras $\phi:T\rightmap{}T^d$ and we have the ideal
$I^d=\phi(I)$ of $A^d$.
Then, there is a
triangular interlaced weak ditalgebra $({\cal A}^d,I^d)$ with layer $(R^d,W^d)$
where $R^d=eRe$,
$W^d_0=eW_0e$, and $W^d_1=eW_1e$. The morphism $\phi:({\cal
A},I)\rightmap{}({\cal A}^d,I^d)$
of interlaced weak ditalgebras induces
 a full and faithful functor $F^d:=F_\phi:({\cal
A}^d,I^d)\g\Mod\rightmap{}({\cal A},I)\g\Mod$
whose image consists of the objects annihilated by $1-e$.
\end{proposition}

\begin{proof} We can identify $W_r^d$ with $W_r/((1-e)W_r+W_r(1-e))$, for $r\in
\{0,1\}$.
Since  $\delta$ is a morphism of $R$-$R$-bimodules, there are commutative
diagrams
$$\begin{matrix}
   W_0&\rightmap{\delta}&A\otimes_RW_1\otimes_RA\\
   \shortlmapdown{\phi_0}&&\shortlmapdown{\phi}\\
   W^d_0&\rightmap{\delta_0^d}&A^d\otimes_{R^d}W^d_1\otimes_{R^d}A^d\\
  \end{matrix}
  $$

  $$\begin{matrix}
   W_1&\rightmap{\delta}&A\otimes_RW_1\otimes_RA\otimes_RW_1\otimes_RA\\
   \shortlmapdown{\phi_1}&&\shortlmapdown{\phi}\\

W^d_1&\rightmap{\delta_1^d}&A^d\otimes_{R^d}W^d_1\otimes_{R^d}A^d\otimes_{R^d}
W_1^d\otimes_{R^d}A^d\\
  \end{matrix}$$
indeed $\phi\delta[(1-e)W_r+W_r(1-e)]=0$, for $r\in \{0,1\}$.
Then, from  (\ref{L: context of reduction by a quotient}),  we have
a derivation
$\delta^d:T^d\rightmap{}T^d$
such that ${\cal A}^d=(T^d,\delta^d)$ is a weak ditalgebra with layer
$(R^d,W^d)$, and
$\phi:({\cal A},I)\rightmap{}({\cal A}^d,I^d)$ is a morphism of triangular
interlaced weak ditalgebras.
Moreover,
an object in $M\in ({\cal A},I)\g\Mod$ is
annihilated by $K_0=\Ker\phi_{\vert A}=A(1-e)A$ iff $(1-e)M=0$. Recall that
$K_1=\Ker(\phi_{\vert V}:V\rightmap{}V^d)$.  Since $\delta(K_0)\subseteq
K_0V+VK_0$, by (\ref{L: sobre ideales de weak dits generated by I}), the ideal
$J$ of $T$ generated by $K_0$ and $K_0V+VK_0$
is an ideal of ${\cal A}$. But, from (\ref{L: context of reduction by a
quotient})(7),
we get $K=\Ker\phi=J$, so
$K_1=K\cap V=K_0V+VK_0$, and again from the last lemma, $F_\phi$ is full and
faithful.
\end{proof}

\begin{proposition}[regularization]\label{P: regularization}
 Assume that $({\cal A},I)$ is a triangular interlaced weak ditalgebra with
layer $(R,W)$.
Assume that we have $R$-$R$-bimodule decompositions $W_0=W'_0\oplus W''_0$ and
$W_1=\delta(W'_0)\oplus W''_1$, set $W'':=W_0''\oplus W''_1$. Consider the
identity map
$\phi_{\flat}:R\rightmap{}R$, the canonical projections
$\phi_j:W_j\rightmap{}W''_j$, for $j\in
\{0,1\}$, and the tensor algebra
 $T^r=T_R(W'')$. Then, we have a morphism of graded
algebras $\phi:T\rightmap{}T^r$
and the ideal $I^r=\phi(I)$ of $A^r$. Then, there is a triangular
interlaced weak ditalgebra
$({\cal A}^r,I^r)$ with layer $(R^r,W^r)$, where $R^r=R$, $W_0^r=W''_0$,
and $W^r_1=W''_1$.  The morphism $\phi:({\cal A},I)\rightmap{}({\cal A}^r,I^r)$
of interlaced weak ditalgebras induces a full and faithful functor
$F^r:=F_\phi:({\cal
A}^r,I^r)\g\Mod\rightmap{}({\cal A},I)\g\Mod$. Moreover,
if $({\cal A},I)$ is a Roiter interlaced weak ditalgebra,
then  $M\in ({\cal A},I)\g\Mod$ is isomorphic to an object in the image of
$F^r$ iff $\Ker\,\delta\cap W'_0$ annihilates $M$. In particular, if
this intersection is zero, $F^r$ is an equivalence of categories.
\end{proposition}

\begin{proof} It is clear that the morphisms
$\delta'_0:=\phi_{\vert V}\delta_{\vert W''_0}:W''_0\rightmap{}V^r$ and
 $\delta'_1:=\phi_{\vert [T]_2}\delta_{\vert W''_1}:W''_1\rightmap{}[T^r]_2$
provide commutative squares as in the statement of (\ref{L: context of
reduction
by a quotient}),
so we can apply this result to obtain a derivation $\delta^r$ on $T^r$ such
that
$\phi:({\cal A},I)\rightmap{}({\cal A}^r,I^r)$ is a morphism of interlaced weak
ditalgebras
and ${\cal A}^r$ admits the layer $(R,W^r)$. Moreover, we know that the induced
functor $F^r:({\cal A}^r,I^r)\g\Mod\rightmap{}({\cal A},I)\g\Mod$ is faithful.

Here, the kernel of $\phi_{\vert A}:A\rightmap{}A^r$ is $K_0=AW'_0A$.
From (\ref{L: context of reduction by a quotient})(7), we know that
$K=\Ker\phi=TW'_0T+T\delta(W'_0)T$. By (\ref{L: sobre ideales de weak dits
generated by I}),
the ideal $J$ of $T$ generated by $K_0$ and $K_0V+\delta(K_0)+VK_0$
is an ideal of ${\cal A}$. Since $K=J$,
the kernel of $\phi_{\vert V}:V\rightmap{}V^r$ is
$K_1=K\cap V=K_0V+\delta(K_0)+VK_0$.
From  (\ref{L: context of reduction by a quotient}), we obtain that $F^r$
is a full functor.

Finally, take $M\in ({\cal A},I)\g\Mod$ such that $(\Ker\,\delta\cap W'_0)M=0$.
Consider the action map $\psi\in \Hom_{R\g R}(W_0,\Hom_k(M,M))$ of $W_0$ on the
$A$-module $M$ and denote by $\psi'$ its restriction to $W'_0$. The condition
on $M$ implies that we can factor $\psi'$ through $\delta$, so there is a
morphism of $R$-$R$-bimodules $f_1^1:\delta(W'_0)\rightmap{}\Hom_k(M,M)$ such
that $\psi'=f_1^1\delta$. Consider the morphism of $R$-$R$-bimodules
$f^1:=(f^1_1,0):\delta(W'_0)\oplus W''_1\rightmap{}\Hom_k(M,M)$.
If $({\cal A},I)$ is a Roiter interlaced weak ditalgebra,
 we obtain an
isomorphism $(1_M,f^1):\underline{M}\rightmap{}M$ in $({\cal A},I)\g\Mod$.
We claim that $\underline{M}$ is in the image of $F^r$. In order to apply
(\ref{L: context of reduction by a quotient})(5), we
want to show that $\underline{M}$ is annihilated by $K_0$, or equivalently by
$W'_0$, and is therefore in the image of $F^r$. Take $w\in W'_0$, $m\in
\underline{M}$, and denote by $w\cdot m$ the action corresponding to the
structure of the $A$-module $\underline{M}$, then
$$w\cdot m=wm-f^1(\delta(w))[m]=wm-f^1_1(\delta(w))[m]=wm-wm=0.$$
\end{proof}

\begin{lemma}[factoring out a direct summand of $W_0$]\label{L: reduction by a
quotient}
 Let $({\cal A},I)$ be a triangular interlaced weak ditalgebra, where ${\cal
A}=(T,\delta)$ admits the triangular layer $(R,W)$. Assume that there is a
decomposition of $R$-$R$-bimodules $W_0=W'_0\oplus W''_0$, such that
$W'_0\subseteq I$ and $\delta(W'_0)\subseteq AW'_0V+VW'_0A$.
Set $T^q=T_R(W^q)$, where $W^q_0=W''_0$, $W^q_1=W_1$, and
$W^q=W_0^q\oplus W_1^q$.
Then, there is a derivation $\delta^q$ on $T^q$ such that
${\cal A}^q:=(T^q,\delta^q)$ is a weak ditalgebra with triangular layer
$(R,W^q)$.
The $t$-algebra $A^q=T_R(W''_0)$ can be identified with the quotient algebra
$A/I'$,
where $I'=AW'_0A$ is the ideal of $A$ generated by $W'_0$,
so we can consider $I^q:=I/I'$ as an ideal of $A^q$. Then  $({\cal A}^q,I^q)$
is a triangular
interlaced weak ditalgebra, and there is a morphism of interlaced weak
ditalgebras
$\phi:({\cal A},I)\rightmap{}({\cal A}^q,I^q)$, with $\phi(I)=I^q$, which
induces an equivalence of
categories
$$F^q:=F_\phi:({\cal A}^q,I^q)\g\Mod\rightmap{}({\cal A},I)\g\Mod.$$
\end{lemma}

\begin{proof} Consider the identity map $\phi_{\flat}:R\rightmap{}R$,
the canonical projection $\phi_0:W_0\rightmap{}W^q_0$, the identity map
$\phi_1:W_1\rightmap{}W^q_1$, and the induced morphism of graded algebras
$\phi:T\rightmap{}T^q$. Consider the algebras $A=T_R(W_0)$ and
$A^q=T_R(W''_0)$,
and the $A\g A$-bimodules $V=[T]_1=AW_1A$
and $V^q=[T^q]_1=A^qW_1A^q$. Then, we have commutative squares
$$\begin{matrix}
  W_0&\rightmap{\delta}&V\\
  \shortlmapdown{\phi_0}&&\shortlmapdown{\phi_{\vert V}}\\
  W^q_0&\rightmap{\delta^q_0}&V^q\\
  \end{matrix}
  \hbox{ \hskip2cm }
\begin{matrix}
  W_1&\rightmap{\delta}&[T]_2\\
  \shortlmapdown{\phi_1}&&\shortlmapdown{\phi_{\vert [T]_2}}\\
  W^q_1&\rightmap{\delta^q_1}&[T^q]_2\\
  \end{matrix} $$
  where $\delta$ is the derivation of ${\cal A}$, and
$\delta^q_0:=\phi_{\vert
V}\delta_{\vert W^q_0}$ and
  $\delta^q_1=\phi_{\vert [T]_2}\delta_{\vert W^q_1}$ are morphisms of
$R$-$R$-bimodules. Then we can apply
  (\ref{L: context of reduction by a quotient}) to this situation to obtain a
derivation $\delta^q$ on
$T^q$ such that
  $\phi:({\cal A},I)\rightmap{}({\cal A}^q,I^q)$ is a morphism of interlaced
weak
ditalgebras, where $I^q=\phi(I)$, and $({\cal A}^q,I^q)$ admits the triangular
layer $(R,W^q)$.
Moreover, we have and induced faithful functor
$$F^q:=F_\phi:({\cal A}^q,I^q)\g\Mod\rightmap{}({\cal A},I)\g\Mod.$$
Using the definition of $\phi$, by (\ref{L: context of reduction by a
quotient})(7),
 we find that the kernel of the
restriction $\phi_{\vert A}:A\rightmap{}A^q$ is $K_0=I'$, and that the kernel
of $\phi:T=T_R(W)\rightmap{}T_R(W^q)=T^q$ is the ideal $K$ of $T$
generated by $W'_0$. Our assumption on $\delta(W'_0)$
implies that $\delta(I')\subseteq I'V+VI'$. Thus the ideal $K$ coincides
with the ideal $J'$ of $T$ generated by $I'$ and $I'V+VI'$. This implies, see
(\ref{L: sobre ideales de weak dits generated by I}),
that $K_0=K\cap A=I'$ and $K_1=K\cap V=I'V+VI'$,
where $K_1$ is the kernel of the restriction map
$\phi_{\vert V}:V\rightmap{}V^q$.
Therefore, by (\ref{L: context of reduction by a quotient})(5\&6), the functor
$F^q$ is full and dense.
\end{proof}

\begin{lemma}[absorption]\label{L: absorption}
 Let $({\cal A},I)$ be a triangular interlaced weak ditalgebra,
 where ${\cal A}=(T,\delta)$ admits the triangular  layer $(R,W)$.
 Assume that there is a decomposition of $R$-$R$-bimodules
 $W_0=W'_0\oplus W''_0$ with $\delta(W'_0)=0$.
 Then, we can consider another layer $(R^a,W^a)$ for the same weak ditalgebra
${\cal A}$,
 where $R^a$ is the subalgebra of $T$ generated by $R$ and $W'_0$, thus
 we can identify it with $T_R(W'_0)$, $W_0^a=R^aW''_0R^a$, and
 $W_1^a=R^aW_1R^a$. We denote by ${\cal A}^a$  the same ditalgebra
 ${\cal A}$ equipped with its new layer $(R ^a,W^a)$; in particular,
 we have $\delta^a=\delta$. We say that
 ${\cal A}^a$ is obtained from ${\cal A}$ \emph{by absorption of
 the bimodule $W'_0$}. The layer $(R^a,W^a)$ is triangular and we obtain a
 triangular interlaced weak ditalgebra $({\cal A}^a,I^a)$, with $I^a=I$.
The identity is a morphism of interlaced weak
ditalgebras
$\phi:({\cal A},I)\rightmap{}({\cal A}^a,I^a)$,  which
induces an isomorphism  of categories
$$F^a:=F_\phi:({\cal A}^a,I^a)\g\Mod\rightmap{}({\cal A},I)\g\Mod.$$
\end{lemma}

\begin{proof} It is easy to show, see \cite{BSZ}(8.20).
\end{proof}

In \cite{BSZ}\S12, the construction of a new ditalgebra ${\cal A}^X$, for a
given ditalgebra ${\cal A}$ and an admissible module $X$ is detailed.
The following statement claim is that, given a layered weak ditalgebra
${\cal A}=(T,\delta)$, the construction of a  layered weak ditalgebra
${\cal A}^X$ can be made in a very similar way. We reproduce  here its main
constituents for terminology and
precision purposes, and we restrict a little the generality of the construction  in \cite{BSZ} for the sake of simplicity.

\begin{proposition}\label{P: reduction by a B-module}
 Assume that ${\cal A}=(T,\delta)$ is a weak ditalgebra with layer $(R,W)$ such
that there is an $R$-$R$-bimodule decomposition $W_0=W'_0\oplus W''_0$ with
$\delta(W'_0)=0$. Suppose that $X$ is an \emph{admissible}
$B$-module, where $B=T_R(W'_0)$. This means that the algebra
$\Gamma=\End_{B}(X)^{op}$
admits
a splitting $\Gamma=S\oplus P$, where $S$ is a subalgebra of  $\Gamma$,
$P$ is an ideal of $\Gamma$, the direct sum is an $S$-$S$-bimodule
decomposition
of $\Gamma$,
and the right $S$-modules $X$ and $P$ are finitely generated projective.
By $(x_i,\nu_i)_{i\in I}$ and $(p_j,\gamma_j)_{j\in J}$ we denote a finite dual
basis of the right $S$-modules $X$ and $P$, respectively.

There is a comultiplication $\mu:P^*\rightmap{}P^*\otimes_SP^*$ induced by the
multiplication of $P$, which is coassociative (see \cite{BSZ}(11.7)). For
$\gamma\in P^*$,
we have $$\mu(\gamma)=\sum_{i,j\in J}\gamma(p_ip_j)\gamma_j\otimes \gamma_i.$$

The action of $P$ on $X$ induces (see \cite{BSZ}(10--11)) a morphism of
$S$-$R$-bimodules
$\lambda:X^*\rightmap{}P^*\otimes_SX^*$ and a morphism of $R$-$S$-bimodules
$\rho: X\rightmap{}X\otimes_SP^*$ with
$$\lambda(\nu)=\sum_{i\in I,j\in J}\nu(x_ip_j)\gamma_j\otimes \nu_i\hbox{ \ and
\ }
\rho(x)=\sum_{j\in J}xp_j\otimes\gamma_j.$$

Set $\underline{W}_0=BW''_0B$ and $\underline{W}_1=BW_1B$. Recall that we can
identify
$A=T_R(W_0)$ with $T_B(\underline{W}_0)$ and  $T$ with
$T_B(\underline{W})$, where $\underline{W}=\underline{W}_0\oplus
\underline{W}_1$, see
\cite{BSZ}(12.2).

We have the $S$-$S$-bimodules
$$\begin{matrix}W_0^X=X^*\otimes_B
\underline{W}_0\otimes_BX=X^*\otimes_RW''_0\otimes_RX \\
\hbox{  \ \ \, \  \ } \\
W_1^X=(X^*\otimes_B\underline{W}_1\otimes_BX)\oplus
P^*=(X^*\otimes_RW_1\otimes_R X)\oplus P^*.\\
\end{matrix}$$
Consider the tensor algebra $T^X=T_S(W^X)$, where $W^X=W_0^X\oplus W_1^X$.
 Following \cite{BSZ}(12.8),
observe that
for $\nu\in X^*$ and $x\in X$, there is a linear map
$$\sigma_{\nu,x}:T\rightmap{}T^X$$
such that $\sigma_{\nu,x}(b)=\nu(bx)$, for $b\in B$, and  given
$w_1,w_2,\ldots,w_n\in \underline{W}$, we have
$\sigma_{\nu,x}(w_1w_2\cdots w_n)$ is given by
$$\sum_{i_1,i_2,\ldots,i_{n-1}}\nu\otimes w_1\otimes x_{i_1}\otimes
\nu_{i_1}\otimes w_2\otimes x_{i_2}\otimes \nu_{i_2}\otimes w_3\otimes
\cdots\otimes x_{i_{n-1}}\otimes \nu_{i_{n-1}}\otimes w_n\otimes x.$$
There is a derivation $\delta^X$ on $T^X$ determined by
$\delta^X(\gamma)=\mu(\gamma)$, for $\gamma\in P^*$ and
$$\delta^X(\nu\otimes w\otimes x)=\lambda(\nu)\otimes w\otimes x+
\sigma_{\nu,x}(\delta(w))+(-1)^{\deg w+1}\nu\otimes w\otimes \rho(x),$$
for $w\in \underline{W}_0\cup \underline{W}_1$, $\nu\in X^*$, and $x\in X$.

Then, ${\cal A}^X=(T^X,\delta^X)$ is a weak ditalgebra.
If the layer $(R,W)$ of ${\cal A}$ is triangular and $X$ is a triangular
admissible $B$-module,
as in \cite{BSZ}(14.6), then the layer $(S,W^X)$ of ${\cal A}^X$ is also
triangular.
\end{proposition}

\begin{proof} This is similar to the first steps in the proof of
\cite{BSZ}(12.9).
For the triangularity statement, choose appropriate dual basis for $X$ and $P$
as in
\cite{BSZ}(14.7) and then follow the proof of \cite{BSZ}(14.10), where we can
assume that $W'_0$ belongs to the triangular filtration of $W_0$, so $B$  is an
initial subalgebra of ${\cal A}$.
\end{proof}

\begin{proposition}[reduction with an admissible module $X$]\label{P: (AX, IX)}
Assume that $({\cal A},I)$ is a triangular interlaced weak ditalgebra, where
 ${\cal A}=(T,\delta)$ is a weak ditalgebra with layer $(R,W)$ such
that there is an $R$-$R$-bimodule decomposition $W_0=W'_0\oplus W''_0$ with
$\delta(W'_0)=0$. Suppose that $X$ is a triangular admissible
$B$-module, where $B=T_R(W'_0)$. As usual, denote by $A^X=[T^X]_0$.
Consider the ideal $I^X$ of $A^X$ generated by the elements
$\sigma_{\nu,x}(h)$,
for $\nu\in X^*$,  $h\in I$, and $x\in X$. Then,
\begin{enumerate}
 \item The pair $({\cal A}^X,I^X)$ is a triangular interlaced weak ditalgebra;
\item There is a
functor $F^X:({\cal A}^X,I^X)\g\Mod\rightmap{}({\cal A},I)\g\Mod$ such
that for $M\in ({\cal A}^X,I^X)\g\Mod$,
the underlying $B$-module of $F^X(M)$ is $X\otimes_SM$ and
$$a\cdot (x\otimes m)=\sum_ix_i\otimes \sigma_{{\nu_i},x}(a)m,$$
for $a\in A$, $x\in X$, and $m\in M$. Moreover,
given the morphism $f=(f^0,f^1)\in \Hom_{({\cal A}^X,I^X)}(M,N)$,
we have $F^X(f)=(F^X(f)^0,F^X(f)^1)$ given by
$$\begin{matrix}
   F^X(f)^0[x\otimes m]&=&x\otimes f^0(m)+\sum_jxp_j\otimes f^1(\gamma_j)[m]\\
   F^X(f)^1(v)[x\otimes m]&=& \sum_ix_i\otimes f^1(\sigma_{\nu_i,x}(v))[m]\hfill\\
  \end{matrix}$$
  for $v\in V$, $x\in X$, and $m\in M$.

\item There is a constant $c_X\in \hueca{N}$ such that,
for any $M\in ({\cal A}^X,I^X)\g\mod$, we have
$\dim_kF^X(M)\leq c_X\dim_kM.$
\item For any $N\in ({\cal A},I)\g\Mod$ which is isomorphic  as a $B$-module
to some $B$-module of the form $X\otimes_SM$, for some $M\in S\g\Mod$, we have
$N\cong F^X(\overline{M})$ in $({\cal A},I)\g\Mod$, for some
$\overline{M}\in ({\cal A}^X,I^X)\g\Mod$.
\end{enumerate}
\end{proposition}

\begin{proof} We choose appropriate dual basis for the right $S$-modules
$P$ and $X$ as in
\cite{BSZ}(14.7).

\medskip

\noindent(1): We show first that $I^X$ is an ${\cal A}^X$-triangular ideal. By
assumption, we have a vector space filtration
$0=H_0\subseteq H_1\subseteq \cdots\subseteq H_\ell=I$ such that
$\delta(H_i)\subseteq AH_{i-1}V+VH_{i-1}A$, for each $i\in [1,\ell]$.
Then, we have the vector space filtration
$$0=H_0^X\subseteq H_1^X\subseteq \cdots\subseteq H_{2\ell_X(\ell+1)}^X\subseteq I^X,$$
where each space $H_m^X$ is generated by the set
$$\{\sigma_{\nu,x}(h)
\mid \nu\in X^*, h\in I, x\in X, \hbox{ with } \h(\nu)+2\ell_X\h(h)+\h(x)\leq
m\},$$
and the heights are taken relative to the filtrations of $X^*$, $I$, and $X$.
Then, we can see that the computations of \cite{BSZ}\S14 also verify
that the filtration of $I^X$ makes $I^X$ an ${\cal A}^X$-triangular ideal.
See also (\ref{R: defs de triangularidad son equivalentes}) and the solution to Exercise (14.11) of \cite{BSZ}.

In order to see that ${\cal A}^X$ is interlaced with $I^X$, we need to show the
inclusion
 $(\delta^X)^2(T^X)\subseteq J^X$, where $J^X$ is the ideal of ${\cal A}^X$
generated by $I^X$.
For this, it will be enough to show that $(\delta^X)^2(W_0^X)\subseteq J^X$ and
$(\delta^X)^2(W_1^X)\subseteq J^X$. Equivalently, we need to show that
$(\delta^{X})^2(\sigma_{\nu,x}(w))\in J^X$, for any $\nu\in X^*$,
$w\in \underline{W}_0\cup \underline{W}_1$, and $x\in X$; and that
$(\delta^X)^2(\gamma)\in J^X$,
for any $\gamma\in P^*$. The argument is essentially contained in the proof of
\cite{BSZ}(12.9),
where, under the assumption $\delta^2=0$ it is proved that $(\delta^X)^2=0$. In
our situation here, it is again easy to see that $(\delta^X)^2(\gamma)=0$, for
$\gamma\in P^*$, so we look at the other case, where we have to calculate
$(\delta^X)^2(\sigma_{\nu,x}(w))$.
The same computations given there, show that
$$(\delta^X)^2(\sigma_{\nu,x}(w))=\sigma_{\nu,x}(\delta^2(w)).$$
Since $\delta^2(w)\in J$, where $J$ is the ideal of ${\cal A}$ generated by
$I$, from \cite{BSZ}(12.8)(3),
we obtain
$\sigma_{\nu,x}(\delta^2(w))\in
\sigma_{\nu,x}(I[T]_2+VIV+[T]_2I)\subseteq
I^X[T^X]_2+V^XI^XV^X+[T^X]_2I^X$, for $w\in \underline{W}_0$, and
$\sigma_{\nu,x}(\delta^2(w))\in \sigma_{\nu,x}(I[T]_3+VI[T]_2
+[T]_2IV+[T]_3I)\subseteq I^X[T^X]_3+V^XI^X[T^X]_2
+[T^X]_2I^XV^X+[T^X]_3I^X)$, for $w\in \underline{W}_1$.

\medskip
\noindent(2): For the existence of $F^X:({\cal
A}^X,I^X)\g\Mod\rightmap{}({\cal A},I)\g\Mod$,
we need to have in mind the construction of the functor $F^X$ in
\cite{BSZ}(12.10), where given
an $A^X$-module $M$, a structure of an $A$-module can be defined on
$X\otimes_SM$ by the recipe
$$a\cdot (x\otimes m)=\sum_{i}x_i\otimes\sigma_{\nu_i,x}(a)m$$
for  $a\in A$, $x\in X$, and $m\in M$. From this formula,
we get that $I^XM=0$ implies that $I(X\otimes_SM)=0$. Then,
the computations in the proof of
\cite{BSZ}(12.10) show that there is a functor
 $F^X:({\cal A}^X,I^X)\g\Mod\rightmap{}({\cal A},I)\g\Mod$, with
$F^X(M)=X\otimes_SM$.

 \medskip
 \noindent(3): Denote by $c_X$ the cardinality of the dual basis of the right
 $S$-module $X$. Then, for $M\in ({\cal A}^X,I^X)\g\mod$, we have
 $$\dim_kF^X(M)=\dim_kX\otimes_SM\leq \dim_k(S^{c_X}\otimes_SM)\leq
 \dim_kM^{c_X}\leq c_X\dim_kM.$$

 \medskip
 \noindent(4):
Finally, we verify the ``density claim'' for the functor $F^X$.
If $N\in ({\cal A},I)\g\Mod$ and $\varphi:N\rightmap{}X\otimes_SM$ is an isomorphism of $B$-modules, we can transfer the structure of $A$-module of $N$ onto $X\otimes_SM$ through $\varphi$ and obtain an object $\overline{X\otimes_SM}\in ({\cal A},I)\g\Mod$ with underlying $B$-module $X\otimes_SM$ and such that $N\cong \overline{X\otimes_SM}$ in
$({\cal A},I)\g\Mod$. Now,
recall from \cite{BSZ}(16.1), that given an $A$-module with underlying
$B$-module of the
form $X\otimes_SM$, where $M$ is an $S$-module,
a structure of $A^X$-module can be defined on $M$ by the formula
$$(\nu\otimes w\otimes x)*m=\sigma(\epsilon\otimes 1)[\nu\otimes w \circ
(x\otimes m)]$$
where $\nu\in X^*$, $x\in X$,  $w\in \underline{W}_0$, $\epsilon:X^*\otimes_B
X\rightmap{}S$ is the evaluation map determined by $\epsilon(\nu\otimes
x)=\nu(x)$, and $\sigma:S\otimes_SM\rightmap{}M$ is the product map;
here, $\circ$ denotes the given $A$-module structure on $X\otimes_SM$.
Let us show by induction on $n$ that, for $w_1,\ldots,w_n\in \underline{W}_0$,
 $x\in X$, $\nu\in X^*$, and $m\in M$, we have
$$\sigma_{\nu,x}(w_1w_2\cdots w_n)*m=
\sigma(\epsilon \otimes 1)[\nu\otimes w_1w_2\cdots w_n\circ (x\otimes m)].$$
Suppose that $n>1$, write $w=w_1$ and $t=w_2\cdots w_n$, and assume that the
statement
holds for $n-1$.
Asume that $t\circ(x\otimes m)=\sum_sx_s\otimes m_s$. Then, applying the
induction hyphotesis and \cite{BSZ}(12.8)(3), we have
$$\begin{matrix}
\sigma_{\nu,x}(wt)*m&=&\sum_i(\sigma_{\nu,x_i}(w)\sigma_{\nu_i,x}(t))*m\hfill\\
&=& \sum_i\sigma_{\nu,x_i}(w)*\sigma(\epsilon\otimes 1)[\nu_i\otimes t\circ
(x\otimes m)]\hfill\\
   &=& \sum_{i,s}(\nu\otimes w\otimes x_i)*(\nu_i(x_s)m_s)\hfill\\
   &=& \sum_s\sigma(\epsilon\otimes 1)[\nu\otimes w\circ (\sum_ix_i\otimes
\nu_i(x_s)m_s)]\hfill\\
   &=& \sum_s\sigma(\epsilon\otimes 1)[\nu\otimes w\circ (x_s\otimes
m_s)]\hfill\\
  &=& \sigma(\epsilon\otimes 1)[\nu\otimes wt\circ (x\otimes m)].\hfill\\
  \end{matrix}$$
Hence $\sigma_{\nu,x}(a)*m=\sigma(\epsilon\otimes 1)[\nu\otimes a\circ
(x\otimes
m)]$, for $a\in A$.
Then, from $I(X\otimes_SM)=0$ it follows that $I^XM=0$. From \cite{BSZ}(16.1), we obtain that the $({\cal A}^X,I^X)$-module $M$ satisfies that $F^X(M)=\overline{X\otimes_SM}\cong N$ in $({\cal A},I)\g\Mod$.
\end{proof}

Notice that the existence of some $0\not= M\in ({\cal A},I)\g\Mod$ as in
(\ref{P: (AX, IX)})(4) implies that $I^X$ is a proper ideal of $A^X$, because
$({\cal A}^X,I^X)\g\Mod$ can not be trivial.

Now we briefly discuss an additional condition on the admissible $B$-module $X$
which
guarantees that
$F^X$ is full and faithful.

\begin{remark}\label{R: preserv proyectividad del bimod W bajo reducciones}
Assume that $\underline{\cal A}'=\underline{\cal A}^{z_1\cdots z_t}$ is an interlaced weak ditalgebra with layer $(R',W')$, obtained by applying successively a finite sequence  of reductions of type $z_1,\ldots,z_t\in \{a,d,r,q,X\}$ from an interlaced weak ditalgebra $\underline{\cal A}$ with layer $(R,W)$, where $R$ is a finite product of fields. Then, the $R'$-$R'$-bimodule $W'$ is projective.

Indeed, if $\underline{\cal A}^z$ is the interlaced weak ditalgebra with layer $(R^z,W^z)$, obtained by a reduction of type $z\in \{a,d,r,q,X\}$ from an interlaced weak ditalgebra $\underline{\cal A}$ with layer $(R,W)$, where $W$ is a
projective $R$-$R$-bimodule, it is not hard to verify that the $R^z$-$R^z$-bimodule $W^z$ is projective.

The layers of the weak ditalgebras ${\cal A}$ appearing in our arguments from \S6 to the end of this work are typically pairs
$(R,W)$, where $R$ is a minimal algebra, as in (\ref{D: source point}), and $W$ is a projective $R$-$R$-bimodule.
\end{remark}

\begin{remark}\label{R: complete B-mods} Assume that $B=T_R(W'_0)$ is a tensor agebra and let $X$ be an admissible $B$-module. Thus, we have a splitting $\Gamma=\End_B(X)^{op}
=S\oplus P$, as in (\ref{P: reduction by a B-module}). We have the ditalgebra $(B,0)$, with trivial differential, and the ditalgebra $(B,0)^X=(T_S(P^*),\delta)$, where $\delta$ is the differential determined by the comultiplication $\mu:P^*\rightmap{}P^*\otimes_SP^*$. Recall from \cite{BSZ}, that the admissible $B$-module $X$ is called \emph{complete} if the functor
$$F^X:(B,0)^X\g\Mod\rightmap{}B\g\Mod$$
is full and faithful.  From (17.4), (17.5), and
(17.11) of \cite{BSZ},
we know that we obtain complete
 admissible $B$-modules $X$ in the following cases:
\begin{enumerate}
 \item  $X$ is a finite direct sum of non-isomorphic finite-dimensional
indecomposables in
 $B\g\mod$;
 \item $X$ is the $B$-module obtained from the regular $S$-module $S$ by
restriction
 through a given epimorphism of $k$-algebras $\phi:B\rightmap{}S$;
 \item $X=X_1\oplus X_2$, where $X_1$ and $X_2$ are complete triangular
 admissible $B$-modules such that $\Hom_{B}({\cal I}_{X_i},{\cal I}_{X_j})=0$,
 for $i\not=j$, and
  ${\cal I}_{X_i}$ denotes the class of $B$-modules of the form $X_i\otimes_SN$, for some $N\in S\g\Mod$.
\end{enumerate}
All the complete admissible $B$-modules we shall consider in this paper are constructed using 1, 2, and 3.
\end{remark}

\begin{proposition} Assume that $B=T_R(W'_0)$ and $X$ is a complete admissible $B$-module. Then, for any ideal $I_0$ of $B$,  the functor $F^{\prime X}$ induces a full and faithful functor
$$F^{\prime X}:((B,0)^X,I_0^X)\g\Mod\rightmap{}B/I_0\g\Mod.$$
\end{proposition}

\begin{proof} By definition, the category $((B,0)^X,I_0^X)\g\Mod$ coincides with the full subcategory of $(B,0)^X\g\Mod$ formed by the $S$-modules $M$ such that $I_0^XM=0$. We have the commutative diagram
$$\begin{matrix}
  ((B,0)^X,I_0^X)\g\Mod&\rightmap{ \ F^{\prime X} \ }&B/I_0\g\Mod\\
  \shortlmapdown{}\hbox{\,\,\,\,}&&\shortlmapdown{}\\
   (B,0)^X\g\Mod&\rightmap{ \ F^{X} \ }&B\g\Mod,\\
  \end{matrix}$$
where the vertical functors, as well as $F^X$, are full and faithful functors. Therefore, the functor $F^{\prime X}$ is also full and faithful.
\end{proof}

 \begin{proposition}\label{P: completez de X para la reduccion}
Under the assumptions of (\ref{P: (AX, IX)}),
consider the ideal $I_0:=B\cap I$ of $B=T_R(W_0')$. Then,
we can consider the triangular interlaced weak ditalgebra
$({\cal B},I_0)$, where ${\cal B}=(B,0)$.
Then, for any complete admissible $B$-module $X$, we have a full and faithful functor
$$F^{\prime X}:({\cal B}^X,I_0^X)\g\Mod\rightmap{}({\cal
B},I_0)\g\Mod=(B/I_0)\g\Mod,$$
and a full and faithful functor
$F^X:({\cal A},I^X)\g\Mod\rightmap{}({\cal A},I)\g\Mod$.
\end{proposition}

\begin{proof}
The proof of the fact that $F^X$ is full and faithful, for any complete
admissible $B$-module $X$, is similar to the proof of \cite{BSZ}(13.5), now using $F^{\prime X}$.
\end{proof}

\begin{definition}\label{D: wild weak dit} An interlaced weak ditalgebra
$({\cal A},I)$ over the field $k$ is called \emph{wild} iff there is an
$A/I$-$k\langle x,y \rangle$-bimodule $Z$, which is free of finite rank as a
right
$k\langle x,y \rangle$-module, such that the composition functor
$$k\langle x,y \rangle\g\Mod\rightmap{ \ Z\otimes_{k\langle x,y \rangle}- \ }A/I
 \g\Mod\rightmap{ \ L_{({\cal A},I)} \ }({\cal A},I)\g\Mod$$
 preserves isomorphism classes of indecomposables. Here $L_{({\cal A},I)}$
 denotes the canonical embedding functor mapping each morphism $f^0$ onto
$(f^0,0)$.
 In this case, we say that \emph{$Z$ produces the wildness of} $({\cal A},I)$.
\end{definition}

The following statement is just \cite{BSZ}(22.7) rewritten for interlaced weak
ditalgebras. The proof given there works here too.

\begin{lemma}\label{L: parameriz de funtores de reduccion}
 Assume that $H:({\cal A}',I')\g\Mod\rightmap{}({\cal A},I)\g\Mod$ is a functor
obtained as a finite composition of functors of type $F^X$, for some admissible
module $X$, or $F_\phi$, for some morphism $\phi:({\cal
A},I)\rightmap{}({\cal A}',I')$ of interlaced weak ditalgebras. Then, $H$
induces by restriction a functor $\underline{H}$ which makes the right square
of the following diagram commutative. If $D$ is any $k$-algebra and $Z$
is an $(A'/I')$-$D$-bimodule, then $H(Z)$ is an $(A/I)\g D$-bimodule
and the first square in the following diagram
commutes up to isomorphism
$$\begin{matrix}
   D\g\Mod&\rightmap{ \ Z\otimes_D- \ }&(A'/I')\g\Mod&
   \rightmap{ \ L_{({\cal A}',I')} \ }&({\cal A}',I')\g\Mod\hfill\\
   \parallel&&\shortlmapdown{\underline{H}}&&\shortlmapdown{H}\\
    D\g\Mod&\rightmap{ \ H(Z)\otimes_D- \ }&(A/I)\g\Mod&
    \rightmap{ \ L_{({\cal A},I)} \ }
    &({\cal A},I)\g\Mod.\\
  \end{matrix}$$
  If $Z$ is a projective right $D$-module, so is $H(Z)$. In the
particular case $D=A'/I'$, we get that $\underline{H}\cong
H(A'/I')\otimes_{(A'/I')}-$
is exact and preserves direct sums.
\end{lemma}

\begin{proposition}\label{P: Reducciones vs wildness}
Assume that a triangular interlaced weak ditalgebra $({\cal A}^z,I^z)$
is obtained from a Roiter interlaced weak ditalgebra $({\cal A},I)$ by
some of the procedures described in this section, that is $z\in \{d,r,q,a,X\}$.
Then, $({\cal A}^z,I^z)$ is  a Roiter interlaced weak ditalgebra.
The associated full and faithful functor
$$F^z:({\cal A}^z,I^z)\g\Mod\rightmap{}({\cal A},I)\g\Mod$$
preserves isomorphism classes and indecomposables.
Moreover, if $({\cal A}^z,I^z)$ is wild (with wildness produced by an
$(A^z/I^z)$-$k\langle x,y\rangle$-bimodule $Z$) then so is $({\cal A},I)$
(with wildness produced by the
$(A/I)$-$k\langle x,y\rangle$-bimodule $F^z(Z)$).
\end{proposition}

\begin{proof} The fact that $({\cal A}^z,I^z)$ is a Roiter interlaced weak
ditalgebra
can be proved in a similar way as in the proof of the same fact for ditalgebras
as
in (9.3) and (16.3) of \cite{BSZ}.
The last statement follows from the preceding lemma, as in
(22.8) and (22.10) of \cite{BSZ}. In case $z=X$,
we use (\ref{P: (AX, IX)})(4).
\end{proof}

\section{Stellar weak ditalgebras}

In this section we shall see that for a special kind of interlaced weak
ditalgebras $({\cal A},I)$, the stellar ones, we can apply reduction procedures
and reach
after a finite number of steps a seminested ditalgebra, as in \cite{BSZ}(23.5).
The corresponding reduction functor covers finite-dimensional
$({\cal A},I)$-modules with dimension bounded by some number $d\in \hueca{N}$.

\begin{definition}\label{D: source point}
Recall that a \emph{minimal algebra} $R$
is a finite product of algebras $R=\prod_{i\in {\cal P}}R_i$, where each $R_i$
is either a rational $k$-algebra or is isomorphic to the field $k$. So, we have a decomposition of the unit of $R$ as a sum $1=\sum_{i\in {\cal P}}e_i$ of  primitive orthogonal idempotents.
 Assume ${\cal
A}$ is a triangular weak ditalgebra with layer $(R,W)$, where $R$ is a minimal algebra. Then, the \emph{points of} ${\cal A}$ are the mentioned idempotents (or the set ${\cal P}$ of their subscripts). A \emph{multiple source $\Omega$ of}
${\cal A}$ is a subset $\Omega$ of ${\cal P}$ such that, if we define $e_{_\Omega}:=\sum_{\omega\in  \Omega}e_\omega$, the following are satisfied:
$e_{_\Omega}W_0=0$, $e_{_\Omega}W_1=e_{_\Omega}W_1e_{_\Omega}$, and $Re_\omega=ke_\omega$, for all $\omega\in \Omega$.

The elements of $\Omega$, or the corresponding family of idempotents, $\{e_\omega\}_{\omega\in \Omega}$ are called \emph{the sources of the given multiple source $\Omega$}.
\end{definition}

\begin{definition}\label{D: stellar weak ditalgebra}
A \emph{stellar weak ditalgebra}
${\cal A}$ is a weak ditalgebra with triangular layer
$(R,W)$, with $R$ a minimal algebra and $W$ a projective $R$-$R$-bimodule, such that  there is a multiple source $\Omega$ of ${\cal A}$, and $W_0=W_0e_{_\Omega}$. For each $\omega\in \Omega$, the algebra $T_{Re_\omega}(W_0e_\omega)$ is called \emph{the star of ${\cal A}$ with center $e_\omega$}.
\end{definition}

Notice that if  ${\cal A}$ is a stellar weak ditalgebra with layer
$(R,W)$, then we have  $A=R\oplus W_0$.

The following elementary lemma, where $R_h$ denotes the localization of the ring $R$ with respect to the element $h\in R$, will be useful to us.

\begin{lemma}\label{L: lema de localizados y sumandos directos}
Given a principal ideal domain $R$ and a finitely generated $R$-module $U$ with  a filtration  in $R\g\mod$
$$0=U_0\subseteq U_1\subseteq \cdots\subseteq U_\ell=U,$$
there is $h\in R$ such that $R_h\otimes_RU$ is a free $R_h$-module and
 $$0= R_h\otimes_RU_0\subseteq R_h\otimes_RU_1
 \subseteq \cdots\subseteq R_h\otimes_RU_\ell=R_h\otimes_RU$$
 is an \emph{additive filtration} of $R_h\otimes_RU$ in $R_h\g\mod$, that is each term of the filtration is a direct summand of the next one.
\end{lemma}

\begin{proof} \emph{Step 1: The case $\ell=1$.} Here, we just have to show
that for any $U\in R\g\mod$ there is $h\in R$ such that $R_h\otimes_R U$ is
a free $R_h$-module.

This is a standard procedure, we have an exact sequence in $R\g\mod$
$$R^s\rightmap{H}R^r\rightmap{}U\rightmap{}0,$$
with $H$ a matrix in $R^{r\times s}\subseteq K^{r\times s}$,
where $K$ is the field of fractions of $R$. Then, there are  invertible
matrices $P$ and $Q$ with coefficients in $K$ such that
$$PHQ=\begin{pmatrix}
       I_d&0\\ 0&0\\
      \end{pmatrix},$$
where $d$ is the rank of $H$. Consider $h\in R$ such that $P$ and $Q$ have
entries in $R_h$, then
$R_h\otimes_RU=\Coker(1_{R_h}\otimes H)\cong R_h^{r-d}$ is free in $R_h\g\mod$.

 \medskip
 \emph{Step 2: The case $\ell =2$.} We take $U'\subseteq U$ and consider the
following exact
 sequence
 in $R\g\mod$
 $$0\rightmap{}U'\rightmap{}U\rightmap{}U/U'\rightmap{}0.$$
 Then, apply the first step to obtain $h\in R$ with $R_h\otimes_R(U\oplus
U'\oplus [U/U'])$
 free in $R_h\g\mod$. So, $R_h\otimes_RU'$, $R_h\otimes_RU$, and
$R_h\otimes_R[U/U']$
 are free modules too. The exact sequence obtained from the preceding one by
tensoring
 by $R_h$ splits and we are done.

 \medskip
 \emph{Step 3: The general case.} This is an easy induction.
\end{proof}

\begin{remark}\label{R: podemos aditivizar finitas filtraciones}
Notice that,  in the last lemma, whenever $U$ is not a torsion $R$-module, we have $R_h\otimes_RU\not=0$.

Moreover,  given a finite sequence of finitely generated $R$-modules $U^1,\ldots,U^t$  with  filtrations in $R\g\mod$
$$0=U_0^t\subseteq U_1^t\subseteq \cdots\subseteq U_{\ell_t}^t=U^t$$
with $t\in [1,m]$, there is $h\in R$, such that after tensoring by the same $R_h$ the preceding filtrations, we obtain additive filtrations with free terms in $R_h\g\mod$
$$0= R_h\otimes_RU^t_0\subseteq R_h\otimes_RU^t_1
 \subseteq \cdots\subseteq R_h\otimes_RU^t_{\ell_t}=R_h\otimes_RU^t.$$
 \end{remark}

\begin{lemma}[multiple unravelling with stars]\label{L: unravelling multiple}
 Let $({\cal A},I)$ be a triangular interlaced weak ditalgebra with
 triangular layer $(R,W)$. Suppose that ${\cal A}$ is stellar with stars centers $\{e_\omega\}_{\omega\in \Omega}$.  We have $R=\prod_{i\in {\cal P}}R_ie_i$, where each $Re_i$ is either
isomorphic to $k$ or to some rational algebra.
We can assume that ${\cal P}=J\uplus J'$ where $Re_i=ke_i$, for $i\in J$, and
$Re_j=R_je_j$ with $R_j=k[x]_{g_j}$, for $j\in J'$.
   Notice that,  $\Omega\subseteq J$.

 Take  $d\in \hueca{N}$ and  non-zero elements  $h_j\in R_j$, for $j\in J'$.
   Then,  there is complete triangular admissible $R$-module $X$ such that   $({\cal A}^X,I^X)$ is a triangular interlaced weak ditalgebra
  with
   triangular layer $(S,W^X)$, where $S$ is a minimal algebra of the form  $$S=\left[\prod_{s\in J''}kf_s\right]\times \left[\prod_{j\in
J'}e_j(R_j)_{h_j}\right]\times \prod_{j\in J}ke_j.$$
Where $\{f_s\}_{s\in J''}$ is a new finite family of primitive idempotents of $S$.
The weak ditalgebra ${\cal A}^X$ is stellar with stars centers
 $\{e_\omega\}_{\omega\in \Omega}$. We have that $I^X\subseteq W_0^X$, whenever $I\subseteq W_0$.
   Moreover, there is  a full and faithful functor
$$F^X:({\cal A}^X,I^X)\g\Mod\rightmap{}({\cal A},I)\g\Mod$$
such that for any $M\in ({\cal A},I)\g\Mod$ with $\dim_k{M}\leq d$, there is
some
$N\in ({\cal A}^X,I^X)\g\Mod$ with $F^X(N)\cong M$.
\end{lemma}

\begin{proof}  We have $R=Re\times Re'$, where $e=\sum_{i\in J}e_i$
 and $e'=\sum_{j\in J'}e_j$. Consider the algebra $C=\prod_{j\in J'}e_jR_j/\langle h_j^d\rangle $, which admits only a finite number
 of isoclasses of indecomposable finite-dimensional $C$-modules, say
 represented by the $C$-modules $\{Z_s\}_{s\in J''}$.
 Consider $Z:=\bigoplus_{s\in J''}Z_s$ and the $R$-module
 $$X=Z\bigoplus\left[ \bigoplus_{j\in J'}e_j(R_j)_{h_j}\right]\bigoplus Re.$$
 Then, we have the splitting
 $\End_R(X)^{op}=S\oplus P$, where
 $$S=\left[\prod_{s\in J''}kf_s\right]\times \left[\prod_{j\in
J'}e_j(R_j)_{h_j}\right]\times Re,$$
  $f_s\in \End_R(X)^{op}$ is the idempotent corresponding to the indecomposable
direct
summand
 $Z_s$ of $X$, and $P=\rad \End_R(Z)^{op}$.

 Consider the filtration of $P$ given by its powers
 $$0=P^{(\ell_P+1)}\subseteq P^{(\ell_P)}\subseteq \cdots\subseteq P^{(1)}=P,$$
 thus $P^{(i)}P^{(j)}\subseteq P^{(i+j)}$ for all $i,j\in[1,\ell_P]$ with
$i+j\leq \ell_P$,
 and $P^{(i)}P^{(j)}=0$, otherwise.
 It determines a filtration of the $R$-module  $Z$
 $$0={\cal Z}_0\subseteq {\cal Z}_1\subseteq\cdots\subseteq {\cal Z}_{\ell_Z}=Z,$$
 with ${\cal Z}_jP\subseteq {\cal Z}_{j-1}$, for all $j\in[1,\ell_Z]$.
 From (\ref{R: complete B-mods}), we know that the $R$-module $X$ is complete triangular admissible with filtration
 $$0=X_0\subseteq X_1\subseteq \cdots\subseteq X_{\ell_X}=X,$$
 where $X_1=\left[\oplus_{j\in J'}e_j(R_j)_{h_j}\right]\oplus Re\oplus {\cal Z}_1$,
$X_2=X_1\oplus {\cal Z}_2, \ldots,
 X_{\ell_X}=X_1\oplus {\cal Z}_{\ell_Z}$.

 Then, by (\ref{P: (AX, IX)}), we have a
 triangular interlaced weak ditalgebra $({\cal A}^X,I^X)$
with triangular layer $(S,W^X)$ given by
 $$W_0^X=X^*\otimes_RW_0\otimes_RX \hbox{ \  and \
}W_1^X=(X^*\otimes_RW_1\otimes_RX)\oplus P^*.$$
 Since $X$ is complete, by (\ref{P: completez de X para la reduccion}),
 the functor $F^X:({\cal A}^X,I^X)\g\Mod\rightmap{}({\cal
A},I)\g\Mod$ is full and
faithful. As remarked in (\ref{R: preserv proyectividad del bimod W bajo reducciones}), $W^X$ is a projective $S\g S$-bimodule.

 Since $e_{_\Omega} W_0=0$, for $\omega\in \Omega$, we have
  $$e_\omega[W_0^X]= e_\omega[X^*]\otimes_RW_0\otimes_RX=e_\omega k\otimes_RW_0\otimes_R X=
  e_\omega k\otimes e_\omega W_0\otimes_RX=0.$$
 Thus, $e_{_\Omega}(W_0^X)=0$. Moreover, since  $e_{_\Omega}P^*=0$ and $e_{_\Omega}W_1=e_{_\Omega}W_1e_{_\Omega}$,
 we have:
 $$\begin{matrix}
    e_{_\Omega}W_1^X
    &=&
    \bigoplus_{\omega\in \Omega}e_\omega(X^*)\otimes_RW_1\otimes_RX\hfill\\
    &=&
    \bigoplus_{\omega\in \Omega}ke_\omega\otimes_RW_1\otimes_RX\hfill\\
    &=&
    \bigoplus_{\omega,\omega'\in \Omega}ke_\omega\otimes_Re_\omega W_1e_{\omega'}\otimes_RX\hfill\\
    &=&
    \bigoplus_{\omega,\omega'\in \Omega}ke_\omega\otimes_Re_\omega W_1e_{\omega'}\otimes_Rke_{\omega'} \hfill\\
    &=&
    \bigoplus_{\omega,\omega'\in \Omega}e_\omega (X^*)\otimes_R W_1\otimes_RXe_{\omega'}=e_{_\Omega}W_1^Xe_{_\Omega}. \hfill\\
   \end{matrix}$$
  So, $\Omega$ is a multiple source  in ${\cal A}^X$.
 Since $W_0=W_0e_{_\Omega}$, we also have
 $$\begin{matrix}W_0^Xe_{_\Omega}&=&X^*\otimes_RW_0\otimes_RXe_{_\Omega}\hfill\\
&=&
 \bigoplus_{\omega\in \Omega}X^*\otimes_RW_0\otimes_Rke_\omega\hfill\\
 &=&
 \bigoplus_{\omega\in \Omega}X^*\otimes_RW_0e_\omega \otimes_RX=W_0^X.\hfill
 \end{matrix}$$
 So ${\cal A}^X$ is a stellar weak ditalgebra with stars centers $\{e_\omega\}_{\omega\in \Omega}$.

 Let us show that if $I\subseteq W_0$, then $I^X\subseteq W_0^X$.
 More precisely, that  $I^X=A^X(X^*\widehat{\otimes}_RI\widehat{\otimes}_RX)A^X=X^*\widehat{\otimes}_RI\widehat{\otimes}_RX$, where  $X^*\widehat{\otimes}_RI\widehat{\otimes}_RX$ denotes the image of the canonical map $X^*\otimes_RI\otimes_RX\rightmap{}X^*\otimes_RW_0\otimes_RX$.
 Indeed, by assumption we have $W_0=W_0e_{_\Omega}=\oplus_wW_0e_\omega$ and also $e_{_\Omega}W_0=0$, so for $h\in I\subseteq W_0$, we have $h=\sum_\omega he_\omega$. Then, given generators $\nu\otimes h\otimes y\in X^*\widehat{\otimes}_RI\widehat{\otimes}_RX$
 and $\nu'\otimes w_0\otimes y'\in X^*\otimes_RW_0\otimes_RX$, we have
 $(\nu\otimes h\otimes y)(\nu'\otimes w_0\otimes y')=\sum_\omega (\nu\otimes he_\omega\otimes y)(\nu'\otimes w_0\otimes y')\in
    \sum_\omega(X^*\widehat{\otimes}_RI\widehat{\otimes}_Rke_\omega)(e_\omega k\otimes_RW_0\otimes_RX)=0$. Similarly,
    we obtain that
$(\nu'\otimes w_0\otimes y')(\nu\otimes h\otimes y)\in
    \sum_\omega(X^*\otimes_RW_0\otimes_Rke_\omega)(e_\omega k\widehat{\otimes}_RI\widehat{\otimes}_RX)=0$.

Now, take any $M\in ({\cal A},I)\g\Mod$ with
$\dim_k{M}\leq d$.
Let us show that there is some
$N\in ({\cal A}^X,I^X)\g\Mod$ with $F^X(N)\cong M$.
By (\ref{P: (AX, IX)})(4), it will be enough to see
that there is some
left $S$-module  $N_0$ such that $M\cong X\otimes_SN_0$  as $R$-modules. Since
$M$ satisfies that
$\dim_k e'M\leq \dim_k M\leq d$, we know that $M=e'M\oplus eM$, and
$e'M\cong [\bigoplus_{s\in J''}n_sZ_s]\oplus M_c$,
where every $h_j$ acts invertibly on $e_jM_c$. Then, we can consider the left $S$-module
$$N_0=[\bigoplus_{s\in J''}n_skf_s]\bigoplus M_c\bigoplus eM.$$
Hence, $M\cong X\otimes_SN_0$ as $R$-modules, and there is $N\in ({\cal
A}^X,I^X)\g\Mod$ with $F^X(N)\cong M$.
\end{proof}

\begin{lemma}\label{L: transformando I de submodulo en sumando directo de W0}
 Let $({\cal A},I)$ be a triangular interlaced weak ditalgebra with
 triangular layer $(R,W)$. Suppose that ${\cal A}$ is stellar. Moreover, assume that  $I\subseteq W_0$ and
 take any $d\in \hueca{N}$.
   Then,  there is a triangular interlaced weak ditalgebra
   $({\cal A}^X,I^X)$ with
   triangular layer $(S,W^X)$, such that ${\cal A}^X$ is stellar and $I^X$ is a direct summand of
   $W_0^X$ as an $S$-$S$-bimodule.
Moreover, there is  a full and faithful functor
$$F^X:({\cal A}^X,I^X)\g\Mod\rightmap{}({\cal A},I)\g\Mod$$
such that for any $M\in ({\cal A},I)\g\Mod$ with $\dim_k{M}\leq d$, there is
some
$N\in ({\cal A}^X,I^X)\g\Mod$ with $F^X(N)\cong M$.
\end{lemma}

\begin{proof} We have $R=\prod_{i\in {\cal P}}R_ie_i$, where each $Re_i$ is either
isomorphic to $k$ or to some rational algebra.
We can assume that ${\cal P}=J\uplus J'$ where $Re_i=ke_i$, for $i\in J$, and
$Re_j=R_je_j$ with $R_j=k[x]_{g_j}$, for $j\in J'$.
  Then we get $R=Re\times Re'$, where $e=\sum_{i\in J}e_i$
 and $e'=\sum_{j\in J'}e_j$. Notice that, if $\{e_\omega\}_{\omega\in \Omega}$ are the centers of the stars of ${\cal A}$, we have   $\Omega\subseteq J$.

By assumption $I\subseteq W_0$.
 For each $j\in J'$ and $\omega\in \Omega$, the bimodule
 $e_jW_0e_\omega=e_jW_0e_{_\Omega}e_\omega$ is an
$R_j=R_j\otimes_kk$-module.
 From (\ref{R: podemos aditivizar finitas filtraciones}), there is some
$h_j\in R_j$ such that, for all $\omega\in \Omega$, we have that
$(R_j)_{h_j}\otimes_{R_j}e_jW_0e_\omega$
is a free $(R_j)_{h_j}$-module and  $(R_j)_{h_j}\otimes_{R_j}e_jIe_\omega$ is a
direct summand of
$(R_j)_{h_j}\otimes_{R_j}e_jW_0e_\omega$. So we  obtain
  basis $\hueca{B}_I(j,\omega)\subseteq \hueca{B}_0(j,\omega)$ of the free
$(R_j)_{h_j}$-modules
  $(R_j)_{h_j}\otimes_{R_j}e_jIe_\omega$ and
$(R_j)_{h_j}\otimes_{R_j}e_jW_0e_\omega$.

Now, we can apply (\ref{L: unravelling multiple}) to the parameters $d$ and $\{h_j\}_{j\in J'}$ to obtain the weak interlaced ditalgebra $({\cal A}^X,I^X)$ with layer $(S,W^X)$ as specified in that statement, and we have a full and faithful functor with the wanted properties.

So we only have to show that $I^X$ is a direct summand of $W_0^X$ as $S$-$S$-bimodules. For this, we keep the notation of the proof of (\ref{L: unravelling multiple}). Since $I\subseteq W_0$, we already know that
$I^X\subseteq W_0^X$,
$e_{_\Omega} W^X_0=0$, $W^X_0=W^X_0e_{_\Omega}$, and $e_{_\Omega}W_1^X=e_{_\Omega}W_1^Xe_{_\Omega}$.
 Thus, we get $I^X=I^Xe_{_\Omega}$.
  For $j\in J'$ and $\omega\in \Omega$, we have
 $$e_j[W_0^X]e_\omega=e_j[X^*]\otimes_RW_0\otimes_RXe_\omega=
e_j(R_j)_{h_j}\otimes_RW_0\otimes_Rke_\omega\cong
    e_j(R_j)_{h_j}\otimes_RW_0e_\omega.$$
 Since each $e_j(R_j)_{h_j}$-module filtration
 $e_j[I^X]e_\omega\subseteq e_j[W_0^X]e_\omega$ is isomorphic to the additive filtration
 $e_j(R_j)_{h_j}\otimes_RIe_\omega\subseteq e_j(R_j)_{h_j}\otimes_RW_0e_\omega$ of
 free modules described before, we
can find  bases $\hueca{B}_{I^X}(j,\omega)\subseteq \hueca{B}^X_0(j,\omega)$ of
$e_j[I^X]e_\omega$ and $e_j[W_0^X]e_\omega$. If $j\in J$,
we choose any vector space basis $\hueca{B}_{I^X}(j,\omega)\subseteq \hueca{B}^X(j,\omega)$
of $e_jI^Xe_\omega$ and $e_jW_0^Xe_\omega$ respectively;
similarly, we choose any vector
space basis $\hueca{B}_{I^X}(f_s,\omega)\subseteq \hueca{B}^X(f_s,\omega)$
of $f_sI^Xe_\omega$ and $f_sW_0^Xe_\omega$, respectively, for any $s\in J''$.
Thus, $I^Xe_\omega$  is a direct summand of
$W_0^Xe_\omega$ as $S$-$S$-bimodules, because they are respectively freely generated by the sets
$$[\bigcup_{j\in {\cal P}}\hueca{B}_{I^X}(j,\omega)]\bigcup[\bigcup_{s\in J''}
\hueca{B}_{I^X}(f_s,\omega)]
 \subseteq
[\bigcup_{j\in {\cal P}}\hueca{B}^X_0(j,\omega)]\bigcup[\bigcup_{s\in J''} \hueca{B}^X_0(f_s,\omega)].$$
It follows that $I^X=\bigoplus_{\omega\in\Omega}I^Xe_\omega$  is a direct summand of
$W_0^X=\bigoplus_{\omega\in \Omega}W_0^Xe_\omega$ as $S$-$S$-bimodules, and we are done.
\end{proof}

\begin{proposition}\label{P: factoring out sumandos directos de W0}
 Let $({\cal A},I)$ be triangular interlaced weak ditalgebra with
triangular layer $(R,W)$.
  Moreover, assume that ${\cal A}$ is stellar and that there is a
 decomposition of $R$-$R$-bimodules $W_0=W'_0\oplus W''_0$,
 where $W'_0$ generates the ideal $I$ of $A$ and
 $\delta(W'_0)\subseteq IV+VI$. Set $W^q_0:=W''_0$,
 $W^q_1:=W_1$, $W^q:=W^q_0\oplus W^q_1$, and $T^q:=T_R(W^q)$.
 Then, there is a derivation $\delta^q$ on $T^q$ such that ${\cal
A}^q:=(T^q,\delta^q)$ is a weak ditalgebra and $({\cal A}^q,0)$ is a triangular
interlaced weak ditalgebra with ${\cal A}^q$ stellar and
 triangular layer $(R,W^q)$. Moreover, there is  a morphism
 $\phi:({\cal A},I)\rightmap{}({\cal
A}^q,0)$ of interlaced weak ditalgebras which induces an equivalence of
categories
 $$F^q=F_\phi:({\cal A}^q,0)\g\Mod\rightmap{}({\cal A},I)\g\Mod.$$
\end{proposition}

\begin{proof}
It follows from (\ref{R: preserv proyectividad del bimod W bajo reducciones}), and form  (\ref{L: reduction by a
quotient}) with $I=I'$.
\end{proof}

 Notice that the following statement is about ditalgebras, not weak
ditalgebras, so we can use freely the terminology and results of \cite{BSZ}. Anyway, we extend the terminology to this context.
A \emph{directed element} of a given $R$-$R$-bimodule, where $R$ is a minimal algebra, is any non-zero element $w$ such that $w=e_jwe_i$, for some  primitive idempotents $e_i,e_j$ of $R$.

\begin{definition}\label{D: seminested weak ditalgebra}
Let ${\cal A}$ be a layered weak ditalgebra. Then,
\begin{enumerate}
 \item A layer $(R,W)$ of ${\cal A}$ is called \emph{seminested} iff $R$ is a minimal $k$-algebra, the layer $(R,W)$ is triangular,
the $R$-$R$-bimodule $W_1$ is freely generated by a finite directed subset $\hueca{B}_1$ of
$W_1$,
 the $R$-$R$-bimodule filtration ${\cal F}(W_0)=\{W_0^j\}_{j=0}^{\ell_0}$ of $W_0$ is
freely generated by a set filtration ${\cal
F}(\hueca{B}_0)=\{\hueca{B}_0^j\}_{j=0}^{\ell_0}$ of some finite directed subset
$\hueca{B}_0$ of $W_0$. See \cite{BSZ}(23.2);
\item A layered weak ditalgebra ${\cal A}$ is called \emph{seminested} iff
its layer is seminested. A weak seminested ditalgebra ${\cal A}$ is called \emph{minimal} iff
its layer $(R,W)$ is seminested with $W_0=0$.
\end{enumerate}
\end{definition}

\begin{lemma}\label{L: sobre ditalgebras estelares triangulares}
 Let ${\cal A}$ be a stellar ditalgebra with triangular layer $(R,W)$ and
such that $W_1$ is freely generated by some finite directed subset. Then, given
any $d\in \hueca{N}$,
there is a stellar seminested ditalgebra ${\cal A}'$ and a full and faithful
functor
$$\begin{matrix}F: {\cal A}'\g\Mod&\rightmap{}&{\cal A}\g\Mod\end{matrix}$$
such that for any $M\in {\cal A}\g\Mod$ with $\dim_k{M}\leq d$, there is some
$N\in {\cal A}'\g\Mod$ with $F(N)\cong M$.
\end{lemma}

\begin{proof} We proceed as in the proof of
(\ref{L: transformando I de submodulo en sumando directo de W0}).
We have $R=\prod_{i\in {\cal P}}R_ie_i$, where each $Re_i$ is either
isomorphic to $k$ or to some rational algebra.
We can assume that ${\cal P}=J\uplus J'$ where $Re_j=ke_j$, for $j\in J$, and
$Re_j=R_je_j$ with $R_j=k[x]_{g_j}$, for $j\in J'$. Thus, if $\{e_\omega\}_{\omega\in \Omega}$ denote
the centers of the stars, $\Omega\subseteq  J$.
  Then we get $R=Re\times Re'$, where $e=\sum_{j\in J}e_j$
 and $e'=\sum_{j\in J'}e_j$.

 From the triangularity of $(R,W)$, we have an $R$-$R$-bimodule filtration
 $$0=W_0^0\subseteq W_0^1\subseteq\cdots\subseteq W_0^\ell=W_0.$$
 For each $j\in J'$ and $\omega\in \Omega$, the bimodule $e_jW_0e_\omega$ is an
$R_j=R_j\otimes_kk$-module.
 Then, by  (\ref{R: podemos aditivizar finitas filtraciones}), we can find
$h_j\in R_j$ such that, for all $\omega\in \Omega$, the
 $(R_j)_{h_j}$-module filtration
 $$0=(R_j)_{h_j}\otimes_RW_0^0e_\omega\subseteq (R_j)_{h_j}\otimes_RW_0^1e_\omega
 \subseteq\cdots\subseteq (R_j)_{h_j}\otimes_RW_0^\ell e_\omega
 =(R_j)_{h_j}\otimes_RW_0e_\omega$$
  of the free $(R_j)_{h_j}$-module
$(R_j)_{h_j}\otimes_{R_j}e_jW_0e_\omega$
is additive. Then, we obtain  set
filtrations
 $$\emptyset=\hueca{B}_0^0(j,\omega)\subseteq \hueca{B}_0^1(j,\omega)\subseteq\cdots\subseteq
\hueca{B}_0^\ell(j,\omega)=\hueca{B}_0(j,\omega)$$
 of basis $\hueca{B}_0^t(j,\omega)$ of the free $(R_j)_{h_j}$-modules
$(R_j)_{h_j}\otimes_{R_j}e_jW_0^te_\omega$.

Fix $d\in \hueca{N}$ and  apply (\ref{L: unravelling multiple}) to the weak ditalgebra $({\cal A},0)$ and the family $\{h_j\}_{j\in J'}$ to obtain the weak interlaced ditalgebra $({\cal A}^X,I^X)$ with layer $(S,W^X)$ as specified in that statement, and we have the full and faithful functor with the wanted properties.
Moreover, we know that
$e_{_\Omega}W_0^X=0$, $e_{_\Omega}W^X_1=e_{_\Omega}W^X_1e_{_\Omega}$, and $W_0^X=W_0^Xe_{_\Omega}$.

With the notation used in (\ref{L: unravelling multiple}), we
 we only have to show that the layer $(S,W^X)$ is seminested.
 In the proof of  \cite{BSZ}(14.10), the triangular filtrations of
$W_0^X$ and $W_1^X$ are exhibited. The typical
term $[W_0^X]^m$ of the triangular filtration of the $S\g S$-bimodule $W_0^X$,
defined for $m\in [0,2\ell_X(\ell_0+1)]$, is the sum
 $$[W_0^X]^m=\sum_{r+2\ell_Xs+t\leq m}[X^*]_r\widehat{\otimes}_RW_0^s
 \widehat{\otimes}_RX_t,$$
 where each summand $[X^*]_r\widehat{\otimes}_RW_0^s\widehat{\otimes}_RX_t$
denotes the image of the canonical map
 $[X^*]_r\otimes_RW_0^s\otimes_RX_t\rightmap{}X^*\otimes_RW_0\otimes_RX$.
  Here, $X_t$ denotes the term $t$ of the filtration of $X$ specified in the proof of (\ref{L: unravelling multiple}), and $[X^*]_r$ denotes the term $r$ of the corresponding dual filtration of $X^*$, see \cite{BSZ}\S14.
 For each $j\in J'$ and $\omega\in \Omega$, we have  isomorphisms of
 $(R_j)_{h_j}$-modules
$$e_jW_0^Xe_\omega=e_jX^*\otimes_RW_0\otimes_R Xe_\omega
=e_j(R_j)_{h_j}\otimes_RW_0\otimes_Rke_\omega\cong
e_j(R_j)_{h_j}\otimes_RW_0e_\omega$$
For $r\not=\ell_X$, we have $e_j[X^*]_r=0$; and $e_j[X^*]_{\ell_X}=e_jX^*=e_j(R_j)_{h_j}$, so
 $$e_jX^*\otimes_RW^s_0\otimes_RX_te_\omega
=e_j(R_j)_{h_j}\otimes_RW^s_0\otimes_Rke_\omega\cong e_j(R_j)_{h_j}
\otimes_RW_0^se_\omega,$$
for each $s,t$ with $\ell_X+2\ell_Xs+t\leq m$. They determine the following
commutative squares
$$\begin{matrix}
   \bigoplus_{s,t}e_jX^*\otimes_RW_0^s\otimes_RX_te_\omega&\rightmap{\sigma_j^m}&
   \bigoplus_{s,t}e_j(R_j)_{h_j}\otimes_RW_0^se_\omega\\
   \shortlmapdown{}&&\shortlmapdown{}\\
   e_jW_0^Xe_\omega&\rightmap{\sigma_j}&e_j(R_j)_{h_j}\otimes_R W_0e_\omega\\
  \end{matrix}$$
where $\sigma_j^m$ and $\sigma_j$ are isomorphisms and the vertical arrows are
the canonical maps. Then, the restriction of $\sigma_j$ to the image of
the canonical maps gives an isomorphism of free $(R_j)_{h_j}$-modules,
so an isomorphism of freely generated  $S$-$S$-bimodules,
$$\overline{\sigma}_j^m:e_j[W_0^X]^me_\omega\rightmap{}e_j(R_j)_{h_j}\otimes_R
W_0^{s(m)}e_\omega,$$
 for a suitable $s(m)$ independent of $\omega$. Clearly, $s(m)\leq s(m+1)$.
This implies that from the additivity of the filtration of the
$S$-$S$-bimodule $e_j(R_j)_{h_j}\otimes_RW_0e_\omega$ that we constructed before
we can derive the additivity of the filtration
$$0=e_j[W_0^X]^0e_\omega\subseteq e_j[W_0^X]^1e_\omega\subseteq\cdots
\subseteq e_j[W_0^X]^me_\omega\subseteq\cdots\subseteq e_jW_0^Xe_\omega.$$
 Since each one of these $S$-$S$-bimodules $e_j[W_0^X]^me_\omega$ is freely generated,
 we can find a filtration of bases, where each $(\hueca{B}_0^X)(j,\omega)^m$ freely generates
   the $S$-$S$-bimodule $e_j[W_0^X]^me_\omega$, as follows:
 $$\emptyset=(\hueca{B}^X_0)(j,\omega)^0\subseteq(\hueca{B}^X_0)(j,\omega)^1\subseteq\cdots
\subseteq(\hueca{B}^X_0)(j,\omega)^m\subseteq\cdots\subseteq
(\hueca{B}^X_0)(j,\omega).$$
 These filtrations can be completed to a filtration of sets,
 where each $(\hueca{B}_0^X)(\omega)^m$ freely generates the $S$-$S$-bimodule
 $[W_0^X]^me_\omega$, as follows
  $$\emptyset=(\hueca{B}^X_0)(\omega)^0\subseteq(\hueca{B}^X_0)(\omega)^1\subseteq\cdots
\subseteq(\hueca{B}^X_0)(\omega)^m\subseteq\cdots
\subseteq\hueca{B}^X_0(\omega),$$
 where $\hueca{B}_0^X(\omega)$ freely generates the $S$-$S$-bimodule $W_0^Xe_\omega$.
 This completion is done by suitable choices
  of vector space basis of each $e_j[W_0^X]^me_\omega$, for $j\in J$,
  and of each $f_s[W_0^X]^me_\omega$, for $s\in J''$, as in the proof of
 (\ref{L: transformando I de submodulo en sumando directo de W0}). Taking the union over $\omega\in \Omega$, we get a set filtration ${\cal F}(\hueca{B}_0^X)$ of a directed subset $\hueca{B}_0^X$ of $W_0^X$ which freely generates it.

 It follows that ${\cal A}^X$ is a seminested ditalgebra. Indeed,  we already
know that $W_1^X$ is freely generated because $W_1$ is so.
\end{proof}

\begin{lemma}\label{L: cambio a ditalgebras estelares triangulares con W1 free}
 Let ${\cal A}$ be a stellar triangular  ditalgebra.
 Then, given any $d\in \hueca{N}$,
there is a stellar ditalgebra ${\cal A}'$ with triangular layer $(R',W')$,
such that $W'_1$ is freely generated by some finite directed subset,
 and a full and faithful functor
$$\begin{matrix}F: {\cal A}'\g\Mod&\rightmap{}&{\cal A}\g\Mod\end{matrix}$$
such that for any $M\in {\cal A}\g\Mod$ with $\dim_k{M}\leq d$, there is some
$N\in {\cal A}'\g\Mod$ with $F(N)\cong M$.
\end{lemma}

\begin{proof} It is similar to the proof of the preceding lemma:
if the triangular layer of ${\cal A}$ is $(R,W)$, we ``localize''
some of the factors  of $R$ in order to transform simultaneously every $R_j$-module  $e_jW_1e_\omega$ into a free
module $(R_j)_{h_j}\otimes_{R_j}e_jW_1e_\omega$, for $\omega\in \Omega$.
\end{proof}

 \begin{theorem}\label{T: de stellar wditalg to seminested ditalg}
 Let $({\cal A},I)$  be a  triangular interlaced  weak ditalgebra with
 ${\cal A}$ stellar, triangular layer $(R,W)$, and  stars centers $\{e_\omega\}_{\omega\in \Omega}$.
Suppose that $e_\omega\not\in I$, for all $\omega\in \Omega$, and take any $d\in \hueca{N}$.
Then, there is a seminested ditalgebra ${\cal A}'$ and a
composition of full and faithful reduction functors
$$\begin{matrix}F: ({\cal A}',0)\g\Mod&\rightmap{}&({\cal
A},I)\g\Mod\end{matrix}$$
such that any $ M\in ({\cal A},I)\g\Mod$ with $\dim_k M\leq d$ is of the form
$F(N)\cong M$, for some
$N\in ({\cal A}',0)\g\Mod$.
 \end{theorem}

 \begin{proof} Since $e_{_\Omega}W_0=0$ and $W_0=W_0e_{_\Omega}$, we know that
$A=R\oplus W_0$ and  $V=W_1\oplus (W_1\otimes_RW_0)\oplus W_0e_{_\Omega}W_1e_{_\Omega}$. In fact, the following formula holds
$$(*)\hbox{\hskip1.cm}I=(I\cap R)\oplus (I\cap W_0).$$
Indeed, we have $R=\prod_{i\in {\cal P}}Re_i$ and $I=[\bigoplus_{i\in {\cal P}\setminus\Omega}Ie_i]\oplus [\bigoplus_{\omega\in \Omega}Ie_\omega]$.
For $i\not\in \Omega$,   since $W_0e_i=0$ and $I\subseteq R\oplus W_0$,
we have $Ie_i=I\cap Re_i=(I\cap R)e_i$.
Now, take $\omega\in  \Omega$ and $a\in Ie_\omega$, so $a=r+v$, with $r=re_\omega\in Re_\omega$ and $v=ve_\omega\in W_0e_\omega$.
If $r\not=0$, since $e_\omega W_0=0$, we get $ce_\omega=re_\omega+e_\omega v=e_\omega a\in I$, with $0\not=c\in k$.
So, $e_\omega\in I$, a contradiction.
Thus we get $Ie_\omega=I\cap W_0e_\omega=(I\cap W_0)e_\omega$ and the formula $(*)$ holds.

\medskip
\noindent\emph{Case 1: There is some $j\in {\cal P}\setminus\Omega$ with $e_j\in I\cap R$.}
\medskip

Fix such idempotent $e_j$.
Thus, we have  $e_jM=0$, for each
  $M\in ({\cal A},I)\g\Mod$. Then, we can consider the  interlaced weak
ditalgebra $({\cal A}^d,I^d)$ obtained
  by deletion of the idempotent $e_j$
  and the corresponding full and faithful functor $F^d:({\cal
A}^d,I^d)\g\Mod\rightmap{}({\cal A},I)\g\Mod$, where $R^d$ has less factors
than $R$ and  $F^d$ is dense. Moreover, $({\cal A}^d,I^d)$ is a
stellar weak ditalgebra with stars centers $\{e_\omega\}_{\omega\in \Omega}$ and $e_\omega\not\in I^d$, for all $\omega\in \Omega$. We can repeat this
  process a finite number of times, if necessary, and so we can assume that
$R\cap I$ contains no  primitive idempotent of $R$. For the rest of this proof, we assume that.

 \medskip
 \noindent\emph{Case 2:} $I\cap R\not=0$.
 \medskip

 We will reduce this \emph{Case 2} to the following \emph{Case 3}. Since $I_0:=I\cap R\not=0$, we have the following non-empty  set of indexes
  $${\cal S}=\{s\in {\cal P}\setminus\Omega\mid I_0e_s\not=0\}.$$
 For $s\in {\cal S}$, since  $e_s\not\in I$,  we know that $0\not=I_0e_s$ is a proper ideal of $Re_s$.

 Define
$e=\sum_{s\in {\cal S}}e_s$ and $f=1-e-e_{_\Omega}$. So $1=f+e+e_{_\Omega}$ is a decomposition of
$1\in R$ as a sum of  orthogonal idempotents.  We have the ideal  $I_0\subseteq Re$ and the quotient algebra $Re/I_0\cong \prod_{s\in {\cal S}}Re_s/I_0e_s$ has finite representation type.
Let $\{Z_t\}_{t\in J}$ be a complete (finite) system of representatives of the isoclasses
of
the indecomposable $Re/I_0$-modules. Take $Z=\bigoplus_{t\in J}Z_t$ and consider the following $R$-module
$$X=Z\oplus Rf\oplus Re_{_\Omega},$$
where $Z$ is an $R$-module by restriction through the surjective morphism of algebras $R\rightmap{}Re/I_0$. So, we have $I_0X=0$.
We have the splitting $\Gamma=\End_R(X)^{op}=S\oplus P$, where
$$S=\prod_{t\in J}kf_t\times Rf\times \prod_{\omega\in \Omega}ke_\omega,$$
 $P=\rad\End_R(Z)^{op}$, and $f_t$ is the idempotent corresponding to
$Z_t$ in  the given decomposition of $X$. Thus $S$ is a subalgebra of $\Gamma$,
the given decomposition of $\Gamma$ is an $S$-$S$-bimodule decomposition,
and $X$ and $P$ are finitely generated projective right $S$-modules.
 In fact $X$ is a triangular admissible $R$-module and we can
 form the triangular interlaced weak ditalgebra
 $({\cal A}^X,I^X)$ with layer $(S, W_0^X\oplus W_1^X)$, where
 $$W_0^X=X^*\otimes_RW_0\otimes_RX \hbox{  \ and \  }
 W_1^X=(X^*\otimes_RW_1\otimes_RX)\oplus P^*,$$
 as in (\ref{P: (AX, IX)}).
 By (\ref{P: completez de X para la reduccion}), the admissible $R$-module $X$ is complete  and
 we have a full and faithful functor
 $$F^X:({\cal A}^X,I^X)\g\Mod\rightmap{}({\cal A},I)\g\Mod.$$
 Let us show that $F^X$ covers every finite-dimensional $({\cal A},I)$-module.
Recall from  (\ref{P: (AX, IX)}) that
 $M\in ({\cal A},I)\g\Mod$ is of the form $M\cong F^X(N)$ for some $N\in ({\cal
A}^X,I^X)\g\Mod$ iff the
 underlying $R$-module of $M$ is isomorphic to $X\otimes_SN_0$, for some
$S$-module $N_0$. Assume
 that $M\in ({\cal A},I)\g\Mod$ is finite-dimensional.
 We have that $I_0fM=0$, so $fM$ is a left $Re/I_0$-module and, therefore, it
has the form
 $fM\cong \bigoplus_{t\in J}n_tZ_t$. Then, we have the left $S$-module
 $$N_0=\bigoplus_{t\in J}n_t(kf_t)\bigoplus fM\bigoplus e_{_\Omega}M$$
  which satisfies that $M=fM\oplus eM\oplus e_{_\Omega}M\cong X\otimes_SN_0$, as left
$R$-modules.
 Hence $F^X$ covers $M$, and we have verified the announced property of $F^X$.

 Now, let us examine more closely the triangular interlaced weak ditalgebra  $({\cal A}^X,I^X)$, which is
stellar  with stars centers $\{e_\omega\}_{\omega\in \Omega}$  because we have
  $e_{_\Omega}W_0^X=0$, $W_0^Xe_{_\Omega}=W_0^X$,
  $e_{_\Omega}W_1^X= e_{_\Omega}W_1^Xe_{_\Omega}$,
  and $Se_\omega=ke_\omega$, for $\omega\in \Omega$. We claim that $I^X\subseteq W_0^X$, so we are reduced to the situation examined in \emph{Case 3}, and the composition of $F^X$ with the functor constructed in that case  for the weak ditalgebra $({\cal A}^X,I^X)$ and a given $d$ gives us the functor we were looking for $({\cal A},I)$ and $d$.

  Indeed, we have $I^X\subseteq W_0^X$, because by the formula $(*)$, the ideal $I^X$ is generated by the elements of the form $\sigma_{\nu,x}(h)$, with $h\in I_0\cup (I\cap W_0)$. If $h\in I_0$, we have $\sigma_{\nu,x}(hx)=0$, because $I_0X=0$. If $h\in I\cap W_0$, then $\sigma_{\nu,x}(h)=\nu\otimes h\otimes x\in W_0^X$. So the ideal $I^X$ is generated by elements in $W_0^X$, which implies that $I^X\subseteq W_0^X$, because $W_0^Xe_{_\Omega}=W_0^X$ and $e_{_\Omega}W_0^X=0$.

 \medskip
\noindent\emph{Case 3: $I\cap R=0$.}
\medskip

 From the formula $(*)$, we get $I\subseteq  W_0$. After an application of
 (\ref{L: transformando I de submodulo en sumando directo de W0}),
 we may asume that $I$ is in fact a direct summand of $W_0$. Then, we can apply
  (\ref{P: factoring out sumandos directos de W0}) to obtain a
 triangular interlaced weak ditalgebra $({\cal A}',0)$ and an
 equivalence of categories
 $F:({\cal A}',0)\g\Mod\rightmap{}({\cal A},I)\g\Mod$
 which preserves dimensions. Now, take $d\in \hueca{N}$, and apply
 first (\ref{L: cambio a ditalgebras estelares triangulares con W1 free})
 and then  (\ref{L: sobre ditalgebras estelares triangulares}) to obtain
 a seminested ditalgebra
 ${\cal A}''$ and a full and faithful functor
 $G:({\cal A}'',0)\g\Mod\rightmap{}({\cal A}',0)\g\Mod$
 such that any $L\in ({\cal A}',0)\g\Mod$ with $\dim_k L\leq d$ is of the form
$G(N)\cong L$,
 for some $N\in ({\cal A}'',0)\g\Mod$. Hence, if $M\in ({\cal A},I)\g\Mod$
satisfies $\dim_k M\leq d$,
 we know the existence of $L\in ({\cal A}',0)\g\Mod$ with $F(L)\cong M$ and
$\dim_k L\leq \dim_k M\leq d$. Therefore,
 there is some $N\in ({\cal A}'',0)\g\Mod$ with $FG(N)\cong M$.
 \end{proof}

  \section{Reductions with sources and restrictions}\label{Reducs with source}

  In this section we go back to the reduction procedures and functors
described in \S7, but in more specific situations which consider a triangular interlaced weak ditalgebra $({\cal A},I)$ with a fixed
multiple source $\Omega$. In this case,  the interlaced weak ditalgebra
obtained from the preceding one by deletion of the idempotent $e_{_\Omega}$ admits a realization $({\cal A}^\cirmin ,I^\cirmin )$  within the original weak ditalgebra $({\cal A},I)$, which permits the definition of a restriction functor $\Res:({\cal A},I)\g\Mod\rightmap{}({\cal A}^\cirmin ,I^\cirmin )\g\Mod$.  We show how the reduction procedures and functors are compatible
with this construction and their module categories, respectively.

\begin{definition}\label{D: source idempotent of (A,I)}
 Assume that $({\cal A},I)$ is a triangular interlaced weak
 ditalgebra with layer $(R,W)$ where $R$ is a minimal algebra. Assume that
 $1=\sum_{i\in {\cal P}}e_i$ is the decomposition of the unit of $R$ as  sum of
 primitive orthogonal idempotents. Then, the subset $\Omega\subseteq {\cal P}$ is  called a \emph{multiple source of $({\cal A},I)$} iff $\Omega$ is a multiple source  of  ${\cal A}$, as in (\ref{D: source point}),  and
$e_\omega\not\in I$, for all $\omega\in \Omega$. We have \emph{the associated multiple source idempotent $e_\Omega=\sum_{\omega \in \Omega}e_\omega$}.
\end{definition}

From now on, the layers $(R,W)$ of the triangular interlaced weak
ditalgebras we consider always have an $R$ which is a minimal algebra.

\begin{lemma}[suppression of the multiple source idempotent]\label{L: (Acir,Icir)}
Assume that $({\cal A},I)$ is a triangular interlaced weak ditalgebra
with a fixed multiple source $\Omega$. Let $(R,W)$ be the triangular layer of ${\cal
A}=(T,\delta)$.
Define $f:=1-e_{_\Omega}$, where $e_{_\Omega}=\sum_{\omega\in \Omega}e_\omega$. Consider the tensor algebra
$T^{\cirmin }:=fTf\cong T_{Rf}(fWf)$, which is a vector subspace of $T=T_R(W)$
invariant under $\delta$. Denote by $\delta^{\cirmin }$ the restriction of $\delta$
to $T^{\cirmin }$. Then, we obtain a new weak ditalgebra
${\cal A}^{\cirmin }=(T^\cirmin ,\delta^\cirmin )$ with layer $(R^\cirmin ,W^\cirmin ):=(Rf,fWf)$,
which is also
triangular.

Notice that $fTf$ is an algebra with unit $f$ and we have a surjective morphism
of algebras $\phi:T\rightmap{}fTf=T^\cirmin $ given by $\phi(t)=ftf$, for $t\in T$.
Observe also that $\phi\delta(t)=\delta^{\cirmin }\phi(t)$, for $t\in T$.

The space $I^\cirmin =fIf$ is an ideal of
$A^\cirmin :=[T^\cirmin ]_0$ and
$({\cal A}^\cirmin ,I^\cirmin )$ is a triangular interlaced weak ditalgebra, which is a Roiter interlaced weak ditalgebra whenever $({\cal A},I)$ is so.

Moreover, we have the restriction functor
$\Res: ({\cal A},I)\g\Mod\rightmap{}({\cal A}^\cirmin ,I^\cirmin )\g\Mod$ such that
$\Res(M)=fM$ and, for $h=(h^0,h^1)\in \Hom_{({\cal A},I)}(M,N)$, its restriction is
defined by
$\Res(h^0,h^1)=(\Res(h^0),\Res(h^1))$, where $\Res(h^0):fM\rightmap{}fN$ is the
restriction of the map $h^0$ and
$\Res(h^1):fVf\rightmap{}\Hom_k(fM,fN)$ is obtained from the restriction of
$h^1:V\rightmap{}\Hom_k(M,N)$.
\end{lemma}

\begin{proof} Notice that $W=fW\oplus e_{_\Omega}W_1e_{_\Omega}$, which implies the  ``convexity property'' $fWW\cdots Wf=fWfWf\cdots fWf$, and
$T^\cirmin \cong T_{Rf}(fWf)$. The triangularity of the layer $(Rf,fWf)$ for
$T^{\cirmin }$ clearly
follows from the triangularity of the layer $(R,W)$ of $T$. The rest of the statement is also clear.
\end{proof}

\begin{remark}\label{L: Acir wild implies A wild}
With the notation of the preceding statement,  notice  that, up to isomorphism, the weak interlaced ditalgebra $({\cal A}^\cirmin ,I^\cirmin )$ is obtained from $({\cal A},I)$ by deletion of
the multiple source idempotent $e_{_\Omega}$, as in (\ref{P: idemp deletion}). We adopt here this special notation and terminology in order to stress the particularities derived from the properties of the  multiple source idmpotent $e_\Omega$ . Notably, the \emph{convexity property}, that is, with the preceding notation, we have that  $fghf=fgfhf$, for any $g,h\in T$, where $T$ is the underlying tensor algebra of ${\cal A}$. This permits to realize the interlaced weak  ditalgebra  obtained by deletion of the idempotent $e_{_\Omega}$ within $({\cal A},I)$ and to define the restriction functor $\Res$. This construction plays an essential role in our arguments.

From (\ref{L: parameriz de funtores de reduccion}), we have that, if $({\cal A},I)$ is not wild, neither is $({\cal A}^\cirmin ,I^\cirmin )$.
\end{remark}

The following series of lemmas guarantees that, given a triangular interlaced weak ditalgebra $({\cal A},I)$ with multiple source $\Omega$, whenever we can perform a reduction of type $z\in \{d,r,q,a,X\}$ on $({\cal A}^{\cirmin },I^{\cirmin })$, then we can perform a corresponding reduction of the same type $z$ on $({\cal A},I)$. Moreover, important relations between their module categories are established. We start with the following general remark which will then be particularized to the special cases $z\in \{d,r,q\}$.

\begin{remark}\label{L: (Az,Iz) vs (Acirz,Icirz)}
Assume that $({\cal A},I)$ is a triangular interlaced weak ditalgebra with multiple source $\Omega$. Let $(R,W)$ be the triangular layer of ${\cal A}=(T,\delta)$. Since $R$ is a minimal algebra, we have $R=\prod_{i\in {\cal P}}Re_i$.
Assume that $({\cal A}',I')$ is a triangular interlaced weak ditalgebra with layer
$(R',W')$, and that $\phi:({\cal A},I)\rightmap{}({\cal A}',I')$ is a morphism of interlaced weak ditalgebras as in (\ref{L: context of reduction by a quotient})(4), such that:
${\cal P}=D\uplus C$, where $D$ may be empty,
$\Omega\subseteq C$, and  the minimal algebra $R'$ is $Re_C$, where $e_C:=\sum_{i\in C}e_i\in R$. Moreover, assume that  $\phi_{\flat}:R\rightmap{}R'$ is the canonical projection.   Defining $e'_i:=\phi_{\flat}(e_i)$, for $i\in C$, we obtain $R'=\prod_{i\in C}R'e'_i$, with $R'e'_i\cong Re_i$. We will furthermore assume that $\phi(I)=I'$.

Consider the idempotents $e_{_\Omega}=\sum_{i\in \Omega}e_i\in R$ and its image  $e'_{_\Omega}=\phi_{\flat}(e_{_\Omega})=\sum_{i\in \Omega}e'_i\in R'$, as well as
$f_{_\Omega}=1-e_{_\Omega}\in R$ and its image $f'_{_\Omega}=\phi_{\flat}(f_\Omega)=1-e'_{_\Omega}\in R'$.
Then, $\Omega\subseteq C$ is also a multiple source of $({\cal A}',I')$.

Indeed, we clearly have that  $e'_{_\Omega}W_0'=0$ and $e'_{_\Omega}W_1'=e'_{_\Omega}W_1'e'_{_\Omega}$. If we had $e'_\omega \in I'$, for some $\omega\in \Omega$,
since $\phi(I)=I'$ and $e_\omega W_0=0$, we would get  $e_\omega \in I$, which is not the case.

So, we can consider the
triangular interlaced weak ditalgebras
$({\cal A}^{\cirmin  },I^{\cirmin  })$ and $({\cal A}^{\prime\cirmin  },I^{\prime\cirmin  })$ obtained, respectively, from $({\cal A},I)$ and $({\cal A}',I')$ by suppression of the multiple source idempotents $e_{_\Omega}$ and $e'_{_\Omega}$. By definition, we have
$({\cal A}^\cirmin,I^{\cirmin})=((f_{_\Omega}T_R(W)f_{_\Omega}),\delta^\cirmin),f_{_\Omega}If_{_\Omega})$ and $({\cal A}^{\prime\cirmin},I^{\prime\cirmin})=((f'_{_\Omega}T_{R'}(W')f'_{_\Omega}),\delta^{\prime\cirmin}),f'_{_\Omega}I'f'_{_\Omega})$.

It is easily seen that the given morphism $\phi:({\cal A},I)\rightmap{}({\cal A}',I')$ induces, by restriction, a morphism of interlaced weak ditalgebras $\phi^\cirmin:({\cal A}^\cirmin,I^\cirmin)\rightmap{}({\cal A}^{\prime\cirmin},I^{\prime\cirmin})$ and that we have the following commutative diagram of functors:
$$\begin{matrix}
   ({\cal A}',I')\g\Mod&\rightmap{F_\phi}&({\cal A},I)\g\Mod\\
   \shortlmapdown{\Res}&&\shortlmapdown{\Res}\\
   ({\cal A}^{\prime\cirmin },I^{\prime\cirmin })\g\Mod&\rightmap{F_{\phi^\cirmin}}&({\cal
A}^{\cirmin },I^{\cirmin })\g\Mod.\\
  \end{matrix}$$
\end{remark}

\begin{lemma}[deletion of idempotents with multiple source]\label{L: (Ad,Id) vs
(Acird,Icird)}
Assume that $({\cal A},I)$ is a triangular interlaced weak ditalgebra
with multiple source $\Omega$. Let $(R,W)$ be the triangular layer of ${\cal
A}=(T,\delta)$, with $R=\prod_{i\in {\cal P}}R_i$. Assume that ${\cal P}=\Omega\uplus C\uplus D$, a disjoint union of non-empty subsets.
 Thus we have the decomposition $1=e_{_\Omega}+e_C+e_D$, a sum of orthogonal idempotents of $R$.
Then, we can form the triangular interlaced weak ditalgebra $({\cal A}^d,I^d)$
obtained from $({\cal A},I)$ by deletion of the idempotent $e_D$, as in
(\ref{P: idemp deletion}). Thus ${\cal A}^d=(T^d,\delta^d)$,
where $T^d=T_{R^d}(W^d)$,  with $R^d=Rf_D$, $W^d=f_DW_0f_D\oplus f_DW_1f_D$, and $f_D:=1-e_D\in R$. Moreover, $I^d=\phi(I)$, where $\phi:({\cal A},I)\rightmap{}({\cal A}^d,I^d)$ is the morphism of triangular interlaced weak algebras determined by the
canonical projections $\phi_{\flat}:R\rightmap{}Rf_D$,
$\phi_0:W_0\rightmap{}f_DW_0f_D$, and $\phi_1:W_1\rightmap{}f_DW_1f_D$.
From (\ref{L: (Az,Iz) vs (Acirz,Icirz)}),
we have that $\Omega$ is a multiple source of  $({\cal A}^d,I^d)$.
So, we can consider the triangular interlaced weak ditalgebras $({\cal A}^{\cirmin  },I^{\cirmin  })$ and $({\cal A}^{d\cirmin  },I^{d\cirmin  })$ obtained respectively from $({\cal A},I)$ and $({\cal A}^d,I^d)$ by suppression of the multiple source idempotents $e_{_\Omega}\in R$ and $e_{_\Omega}\in Rf_D$, respectively.

We can also consider the triangular interlaced weak ditalgebra
$({\cal A}^{\cirmin  d},I^{\cirmin  d})$, obtained from $({\cal A}^\cirmin ,I^{\cirmin })$
by deletion of the idempotent $e_D\in R^{\cirmin}=Rf_{\Omega}$, where $f_\Omega=1-e_\Omega\in R$.  We have the corresponding morphism of triangular interlaced weak ditalgebras $\phi^d:({\cal A}^{\cirmin},I^\cirmin{})\rightmap{}({\cal A}^{\cirmin d},I^{\cirmin d})$  determined by the
canonical projections $\phi^d_{\flat}:R^{\cirmin}\rightmap{}R^{\cirmin}f$,
$\phi^d_0:W_0^{\cirmin}\rightmap{}fW^\cirmin_0f$, and $\phi^d_1:W^\cirmin_1\rightmap{}fW^\cirmin_1f$, where $f=f_\Omega-e_D\in R^{\cirmin}$.
Then, we have
$({\cal A}^{\cirmin  d},I^{\cirmin  d})=
({\cal A}^{d\cirmin  },I^{d\cirmin  })$, $\phi^\cirmin=\phi^d$, and
the commutative diagram of functors
$$\begin{matrix}
   ({\cal A}^d,I^d)\g\Mod&\rightmap{F^d}&({\cal A},I)\g\Mod\\
   \shortlmapdown{\Res}&&\shortlmapdown{\Res}\\
   ({\cal A}^{d\cirmin },I^{d\cirmin })\g\Mod&\rightmap{F^{\cirmin  d}}&({\cal
A}^{\cirmin },I^{\cirmin })\g\Mod,\\
  \end{matrix}$$
  where $F^d=F_\phi$ and $F^{\cirmin d}:=F_{\phi^d}=F_{\phi^\cirmin}$ are the corresponding reduction functors.
  Moreover, if $M\in ({\cal A},I)\g\Mod$ is such that
  $\Res(M)\cong F^{\cirmin  d}(N')$, for some
  $N'\in({\cal A}^{d\cirmin },I^{d\cirmin })\g\Mod$, then $M\cong F^d(N)$,
  for some $N\in ({\cal A}^{d},I^{d})\g\Mod$.
\end{lemma}

\begin{proof} The relevant idempotents in the construction of ${\cal A}^{\cirmin d}=(T^{\cirmin d},\delta^{\cirmin d})$ and ${\cal A}^{d \cirmin}=(T^{d\cirmin},\delta^{\cirmin d})$ are $f_D$, $f_\Omega$, $f$, as defined above, and $f':=f_D-e_{\Omega}\in R^d=Rf_D$. Observe that their product in $R$ satisfies $f_Df'=e_C=ff_\Omega$. Therefore, we get $R^{d\cirmin}=Rf_Df'=Rf_\Omega f=R^{\cirmin d}$ and $W_i^{d\cirmin}=f'f_DW_if_Df'=ff_\Omega W_if_\Omega f=W_i^{\cirmin d}$, for $i\in \{0,1\}$. Hence, we have
$T^{\cirmin d}=T_{R^{\cirmin d}}(W^{\cirmin d})
=T_{R^{d\cirmin}}(W^{d\cirmin})=T^{d\cirmin}$. Observe that the restrictions of the map $\phi^\cirmin:T^{\cirmin}\rightmap{}T^{d\cirmin}$ to $R^{\cirmin}$, $W_0^{\cirmin}$, and $W_1^{\cirmin}$ coincide, respectively, with the maps
$\phi^d_{\flat}$, $\phi^d_0$ and $\phi^d_1$ which determine $\phi^d$. Thus, we get $\phi^\cirmin=\phi^d$, as claimed.

Now, recall that the derivations $\delta^d$ and $\delta^{\cirmin d}$ are uniquely determined by the commutativity of the squares
$$\begin{matrix}
   W&\rightmap{\delta}&T\\
   \shortlmapdown{\phi}&&\shortlmapdown{\phi}\\
   W^d&\rightmap{\delta^d}&T^d\\
  \end{matrix} \hbox{ \ \  \ \ and \  \ \ \ }
  \begin{matrix}
   W^{\cirmin}&\rightmap{\delta^{\cirmin}}&T^{\cirmin}\\
   \shortlmapdown{\phi^d}&&\shortlmapdown{\phi^d}\\
   W^{\cirmin d}&\rightmap{\delta^{\cirmin d}}&T^{\cirmin d}.\\
  \end{matrix}
  $$
Since $\phi^\cirmin=\phi^d$, aplying $\cirmin$ to the first diagram, we obtain the second one, with $\delta^{d\cirmin}=\delta^{\cirmin d}$.
Thus, ${\cal A}^{d\cirmin }=(T^{d\cirmin},\delta^{d\cirmin})=(T^{\cirmin d},\delta^{\cirmin d})={\cal A}^{\cirmin d}$. Since $\phi(f_\Omega)=f'$, we have $I^{\cirmin d}= \phi^d(f_\Omega If_\Omega)=\phi(f_\Omega If_\Omega)=f'\phi(I)f'=f'I^df'=  I^{d\cirmin}$.

 The commutativity of the diagram follows from (\ref{L: (Az,Iz) vs (Acirz,Icirz)}).
  Finally, given $M\in ({\cal A},I)\g\Mod$, we know that $F^{\cirmin  d}(N')\cong \Res(M)$, for some $N'\in ({\cal A}^{\cirmin  d},I^{\cirmin  d})\g\Mod$, means that  $e''fM=0$, so $e''M=0$, and $M\cong F^d(N)$, for some $N\in ({\cal A}^d,I^d)\g\Mod$.
\end{proof}

\begin{lemma}[regularization with multiple source]\label{L: (Ar,Ir) vs
(Acirr,Icirr)}
Assume that $({\cal A},I)$ is a triangular interlaced weak ditalgebra
with multiple source  $\Omega$. Let $(R,W)$ be the triangular layer of ${\cal
A}=(T,\delta)$. Suppose that we have $R$-$R$-bimodule decompositions
$W_0=W'_0\oplus W''_0$ and $W_1=\delta(W'_0)\oplus W''_1$.
Then, we can form the triangular interlaced weak ditalgebra $({\cal A}^r,I^r)$
obtained by regularization of the bimodule $W'_0$, as in
(\ref{P: regularization}). Thus ${\cal A}^r=(T^r,\delta^r)$ and $I^r=\phi(I)$,
where $T^r=T_{R}(W^r)$, with  $W^r=W''_0\oplus W''_1$,
and $\phi:({\cal A},I)\rightmap{}({\cal A}^r,I^r)$ is the morphism of triangular interlaced weak ditalgebras determined by the
canonical projections $\phi_{\flat}:R\rightmap{}R$,
$\phi_0:W_0\rightmap{}W''_0$, and $\phi_1:W_1\rightmap{}W''_1$.
From (\ref{L: (Az,Iz) vs (Acirz,Icirz)}), we have that $\Omega$ is a multiple source of   $({\cal A}^r,I^r)$.

Assume furthermore that $W'_0e_{_\Omega}=0$. If we take $f=1-e_{_\Omega}$, we get
$fW'_0f=W'_0$,
$W_0^{\cirmin}=fW_0f=W'_0\oplus fW''_0f$, and $W_1^\cirmin=fW_1f=\delta(W'_0)\oplus fW''_1f$. So,
we can consider the
triangular interlaced weak ditalgebra $({\cal A}^{\cirmin r},I^{\cirmin r})$
obtained from $({\cal A}^\cirmin,I^\cirmin)$ by regularization of the $R^{\cirmin}\g R^\cirmin$-bimodule $W'_0$, as in
(\ref{P: regularization}). We have the corresponding morphism of triangular interlaced weak ditalgebras $\phi^r: ({\cal A}^{\cirmin},I^{\cirmin})\rightmap{}({\cal A}^{\cirmin r},I^{\cirmin r})$ determined by the canonical projections
$\phi^r_{\flat}:R^{\cirmin}\rightmap{}R^{\cirmin}$,
$\phi^r_0:W^{\cirmin}_0\rightmap{}fW''_0f=W_0^{\cirmin r}$, and
$\phi^r_1:W_1\rightmap{}fW''_1f=W_1^{\cirmin r}$.
Then, we have
$({\cal A}^{\cirmin  r},I^{\cirmin  r})=
({\cal A}^{r\cirmin  },I^{r\cirmin  })$, $\phi^\cirmin=\phi^r$, and
the commutative diagram of functors
$$\begin{matrix}
   ({\cal A}^r,I^r)\g\Mod&\rightmap{F^r}&({\cal A},I)\g\Mod\\
   \shortlmapdown{\Res}&&\shortlmapdown{\Res}\\
   ({\cal A}^{r\cirmin },I^{r\cirmin })\g\Mod&\rightmap{F^{\cirmin  r}}&({\cal
A}^{\cirmin },I^{\cirmin })\g\Mod,\\
  \end{matrix}$$
  where $F^r=F_\phi$ and $F^{\cirmin r}:=F_{\phi^r}=F_{\phi^\cirmin}$ are the associated reduction functors.
   Moreover, if $({\cal A},I)$ is a Roiter interlaced weak ditalgebra and $M\in ({\cal A},I)\g\Mod$ is such that
  $\Res(M)\cong F^{\cirmin  r}(N')$, for some
  $N'\in({\cal A}^{r\cirmin },I^{r\cirmin })\g\Mod$, then $M\cong F^r(N)$,
  for some $N\in ({\cal A}^{r},I^{r})\g\Mod$.
\end{lemma}

\begin{proof} We have $R^{\cirmin r}=R^\cirmin=Rf=R^{\cirmin r}$, $W_0^{\cirmin r}=fW''_0f=W_0^{r\cirmin}$, and $W_1^{\cirmin r}=fW''_1f=W_1^{r\cirmin}$. Hence, we get $T^{\cirmin r}=T^{r \cirmin}$. As before, we notice that
 the restrictions of the map $\phi^\cirmin:T^{\cirmin}\rightmap{}T^{r\cirmin}$ to $R^{\cirmin}$, $W_0^{\cirmin}$, and $W_1^{\cirmin}$ coincide, respectively, with the maps
$\phi^r_{\flat}$, $\phi^r_0$ and $\phi^r_1$ which determine $\phi^r$. Thus, we get $\phi^\cirmin=\phi^r$.
The same argument involving the commutative squares in the proof of (\ref{L: (Ad,Id) vs
(Acird,Icird)}), replacing $d$ by $r$, gives that $\delta^{\cirmin r}=\delta^{r\cirmin}$. So,
${\cal A}^{\cirmin r}=(T^{\cirmin r},\delta^{\cirmin r})=(T^{r\cirmin},\delta^{r\cirmin})={\cal A}^{r\cirmin}$. Clearly, we have
 $I^{\cirmin r}=\phi^r(fIf)=\phi(fIf)=f\phi(I)f=I^{r\cirmin}$.

The commutativity of the diagram follows from (\ref{L: (Az,Iz) vs (Acirz,Icirz)}).
If $\Res(M)\cong F^{\cirmin  r}(N')$, then $(\Ker \delta^\cirmin \cap W'_0)fM=0$.
So, $(\Ker \delta^\cirmin \cap W'_0)M=0$, because $M=e_{_\Omega}M\oplus fM$ and $W'_0e_{_\Omega}=0$.
But $\delta^\cirmin _{\vert W'_0}=\delta_{\vert W'_0}$ implies that
$\Ker \delta\cap W'_0=\Ker \delta^\cirmin \cap W'_0$.
It follows that $F^r(N)\cong M$, for some $N\in ({\cal A}^r,I^r)\g\Mod$.
\end{proof}

\begin{lemma}[factoring out a direct summand of $W_0$ with
 multiple source]\label{L: cocientes con fuente}
 Let $({\cal A},I)$ be a triangular interlaced weak ditalgebra with
 multiple source $\Omega$ and triangular layer $(R,W)$. Define $f=1-e_{_\Omega}$.
 Assume that there is a decomposition $fW_0f=W'_0\oplus W''_0$
 as $Rf$-$Rf$-bimodules such that $W'_0\subseteq fIf$ and
 $\delta(W'_0)\subseteq A^\cirmin  W'_0fVf+fVfW'_0A^\cirmin $.
 Thus $W_0=W_0'\oplus(W''_0\oplus W_0e_{_\Omega})$ is a decomposition of
$R$-$R$-bimodules with $W'_0\subseteq I$ and $\delta(W'_0)\subseteq
AW'_0V+VW'_0A$.
Then, we can form the triangular interlaced weak ditalgebra
$({\cal A}^q,I^q)$, determined by this last decomposition, as in
(\ref{L: reduction by a quotient}).  Thus ${\cal A}^q=(T^q,\delta^q)$ and $I^q=\phi(I)$, where
$T^q=T_R(W^q)$, with $W^q_0=W''_0\oplus W_0e_{_\Omega}$ and $W_1^q= W_1$, and
$\phi:({\cal A},I)\rightmap{}({\cal A}^q,I^q)$ is the morphism determined by the canonical projections $\phi_{\flat}:R\rightmap{}R^q$, $\phi_0:W_0\rightmap{}W^q_0$, and $\phi_1:W_1\rightmap{}W^q_1$. From (\ref{L: (Az,Iz) vs (Acirz,Icirz)}), we have that
$\Omega$ is a multiple source in
$({\cal A}^q, I^q)$.

 We can also consider
the triangular interlaced weak ditalgebra $({\cal A}^{\cirmin q},I^{\cirmin  q})$, obtained from $({\cal A}^\cirmin,I^\cirmin)$
by factoring out the direct summand $W'_0$ of the $R^\cirmin\g R^\cirmin$-bimodule decomposition  $W_0^\cirmin=fW_0f=W_0'\oplus W''_0$. We have the corresponding morphism of triangular interlaced weak ditalgebras $\phi^q: ({\cal A}^{\cirmin},I^{\cirmin})\rightmap{}({\cal A}^{\cirmin q},I^{\cirmin q})$ determined by the canonical projections
$\phi^q_{\flat}:R^{\cirmin}\rightmap{}R^{\cirmin q}$,
$\phi^q_0:W^{\cirmin}_0\rightmap{}fW''_0f=W_0^{\cirmin q}$, and
$\phi^q_1:W^{\cirmin}_1\rightmap{}fW''_1f=W_1^{\cirmin q}$.
Then, we have
$({\cal A}^{\cirmin  q},I^{\cirmin  q})=
({\cal A}^{q\cirmin  },I^{q\cirmin  })$, $\phi^\cirmin=\phi^q$, and
the commutative diagram of functors
$$\begin{matrix}
   ({\cal A}^q,I^q)\g\Mod&\rightmap{F^q}&({\cal A},I)\g\Mod\\
   \shortlmapdown{\Res}&&\shortlmapdown{\Res}\\
   ({\cal A}^{q\cirmin },I^{q\cirmin })\g\Mod&\rightmap{F^{\cirmin  q}}&({\cal
A}^{\cirmin },I^{\cirmin })\g\Mod.\\
  \end{matrix}$$
   where $F^q=F_\phi$ and $F^{\cirmin q}:=F_{\phi^q}=F_{\phi^\cirmin}$ are the associated reduction functors.
\end{lemma}

\begin{proof}  We have  $T^{\cirmin  q}=T_{Rf}((fWf)^q)=T_{Rf}(W''_0\oplus fW_1f)=T^{q\cirmin }$. As in the preceding lemmas, the equality $\phi^\cirmin=\phi^r$ is easy to show and, from this, again as before, que get ${\cal A}^{\cirmin q}=(T^{\cirmin q},\delta^{\cirmin q})=(T^{q\cirmin },\delta^{q\cirmin })={\cal A}^{q\cirmin}$. Similarly, we have
$I^{\cirmin  q}=\phi^q(fIf)=\phi(fIf)=f\phi(I)f=I^{q\cirmin }$.
The commutativity of the diagram follows from (\ref{L: (Az,Iz) vs (Acirz,Icirz)}).
\end{proof}

\begin{lemma}[absorption of a loop with multiple source]\label{L: absorption with source}
 Let $({\cal A},I)$ be a triangular interlaced weak ditalgebra with
 multiple source $\Omega$, where ${\cal A}=(T,\delta)$ admits the triangular
 layer $(R,W)$, with $Re_i=ke_i$ for some  $i\in {\cal P}\setminus \Omega$.
 Set $f=1-e_{_\Omega}$. Assume that there is a decomposition of
 $Rf$-$Rf$-bimodules $fW_0f=W'_0\oplus W''_0$ with
 $W'_0\cong e_iR\otimes_kRe_i$ and $\delta(W'_0)=0$.
 Then, we can consider the triangular interlaced weak ditalgebra
 $({\cal A}^\cirmin ,I^\cirmin )$ and the triangular interlaced weak ditalgebra
 $({\cal A}^{\cirmin  a},I^{\cirmin  a})$ obtained by the absorption of the
 $Rf$-$Rf$-subbimodule $W'_0$, as in (\ref{L: absorption}).
 We also have the decomposition
 of $R$-$R$-bimodules $W_0=W'_0\oplus (W''_0\oplus W_0e_{_\Omega})$, and we can
 consider the triangular interlaced weak ditalgebra $({\cal A}^a,I^a)$
 obtained by absorption of the $R$-$R$-bimodule $W'_0$. Then, $\Omega$
 is a multiple source in $({\cal A}^a,I^a)$ and we can consider
 the triangular interlaced weak ditalgebra
 $({\cal A}^{a\cirmin },I^{a\cirmin })$. We have that
 $({\cal A}^{\cirmin  a},I^{\cirmin  a})=({\cal A}^{a\cirmin },I^{a\cirmin })$ and there is a
 commutative square of functors
$$\begin{matrix}
   ({\cal A}^a,I^a)\g\Mod&\rightmap{F^a}&({\cal A},I)\g\Mod\\
   \shortlmapdown{\Res}&&\shortlmapdown{\Res}\\
   ({\cal A}^{a\cirmin },I^{a\cirmin })\g\Mod&\rightmap{F^{\cirmin  a}}&({\cal
A}^{\cirmin },I^{\cirmin })\g\Mod.\\
  \end{matrix}$$
\end{lemma}

\begin{proof} Here $1=\sum_{j\in {\cal P}}e_j$ is a decomposition of
the unit of the minimal algebra $R^a\cong T_R(W'_0)$ as a sum of orthogonal primitive idempotents, which contains the algebra $R$.
The layer of ${\cal A}^a$ is $(R^a,W^a)$ with $R^ae_i\cong k[x]$,
$W_0^a=R^a[W_0''\oplus W_0e_{_\Omega}]R^a$, and $W_1^a=R^aW_1R^a$.
Thus,
$e_{_\Omega}W_0^a=e_{_\Omega}R^aW_0''R^a\oplus e_{_\Omega}R^aW_0e_{_\Omega}R^a=0$ and
$e_{_\Omega}W_1^a=e_{_\Omega}R^aW_1R^a=R^ae_{_\Omega}W_1e_{_\Omega}R^a
=e_{_\Omega}W_1^ae_{_\Omega}$. Moreover, $R^ae_\omega=ke_\omega$, for $\omega\in \Omega$, and we already know
that $e_\omega\not\in I=I^a$. Thus, $\Omega$ is a multiple source in $({\cal A}^a,I^a)$.

Now,
$fW_0^af= fR^aW_0''R^af\oplus fR^a(W_0e_{_\Omega})R^af=fR^aW_0''R^af$ and
$fW_1^af=fR^aW_1R^af$. Then,
$(R^a)f=\prod_{j\in {\cal P}\setminus\Omega\cup \{i\}}Re_j\times R^ae_i=(Rf)^a$, and we get
$$T_{R^af}(fR^aW_0''R^af\oplus fR^aW_1R^af)=
T_{(Rf)^a}( (Rf)^aW''_0(Rf)^a\oplus (Rf)^aW_1(Rf)^a).$$
That is $T^{a\cirmin }=T^{\cirmin  a}$.
Moreover, $\delta^{a\cirmin }=(\delta^a)_\vert=\delta_\vert=
\delta^{\cirmin }=\delta^{\cirmin  a}$, so ${\cal A}^{a\cirmin }={\cal A}^{\cirmin  a}$,
and clearly $I^{a\cirmin }=fIf=I^{\cirmin  a}$.
The rest of the proof is straightforward.
\end{proof}

\begin{lemma}[reduction with an admissible module with multiple source]\label{L: (AX,IX) vs (AcirX,IcirX)}
Assume that $({\cal A},I)$ is a triangular interlaced weak ditalgebra
with multiple source  $\Omega$. Let $(R,W)$ be the triangular layer of ${\cal
A}=(T,\delta)$.
Write $f=1-e_{_\Omega}$ and assume that we have a decomposition $fW_0f=W'_0\oplus W''_0$
of $Rf$-$Rf$-bimodules with $\delta(W'_0)=0$. Thus we have the
decomposition
of $R$-$R$-bimodules $W_0=W'_0\oplus (W''_0\oplus W_0e_{_\Omega})$, with $\delta(W'_0)=0$.

Consider the $k$-subalgebra $B^{\cirmin }=T_{Rf}(W'_0)$ of $T^\cirmin =fTf$
and the $k$-subalgebra
$B=T_R(W'_0)=Re_{_\Omega}\times B^\cirmin $ of $T$.  Suppose that
$X^\cirmin $ is an complete
triangular admissible $B^\cirmin $-module, thus we have a splitting
$\Gamma^\cirmin =\End_{B^\cirmin }(X^\cirmin )^{op}=S^\cirmin \oplus P^\cirmin $, where
$S^\cirmin $ is a
subalgebra of $\Gamma^\cirmin $,  $P^\cirmin $ is an ideal of
$\Gamma^\cirmin $,  and
$X^\cirmin $
and $P^\cirmin $ are
finitely generated projective right $S^\cirmin $-modules.

Now, consider the $B$-module $X=Re_{_\Omega}\oplus X^\cirmin $, where $X^\cirmin $ is considered as a $B$-module by restriction through the projection map  $B\rightmap{}B^\cirmin $. Then, we have a
splitting
$\Gamma=\End_B(X)^{op}=S\oplus P$, where $S=Re_{_\Omega}\times S^\cirmin $ is a
subalgebra of $\Gamma$,  $P$ is an ideal of $\Gamma$, which is identified with the $S$-$S$-bimodule obtained from $P^\cirmin $ through the projection map $S\rightmap{}S^\cirmin $, and $X$ and
$P$ are finitely generated projective right $S$-modules.
Then, $X$ is a complete triangular admissible $B$-module,
and we can form the triangular interlaced weak ditalgebra $({\cal A}^X,I^X)$
obtained from $({\cal A},I)$ by reduction using the admissible $B$-module $X$, as in (\ref{P: (AX, IX)}).

We have that $\Omega$ is a multiple source  in $({\cal A}^X,I^X)$ and we can form
the triangular interlaced weak ditalgebra $({\cal A}^{X\cirmin  },I^{X\cirmin })$.
We can also form the triangular interlaced weak ditalgebra
$(({\cal A}^{\cirmin })^{X^\cirmin },(I^{\cirmin })^{X^\cirmin })$
obtained from
$({\cal A}^{\cirmin  },I^{\cirmin })$ by reduction using the admissible $B^\cirmin $-module $X^\cirmin $.

There is a canonical isomorphism of interlaced weak ditalgebras between  $({\cal A}^{X\cirmin  },I^{X\cirmin  })$ and
$(({\cal A}^{\cirmin })^{X^\cirmin },(I^{\cirmin })^{X^\cirmin })$, which allows us to identify them and we have the
following commutative diagram of functors with full and faithful rows
$$\begin{matrix}
   ({\cal A}^X,I^X)\g\Mod&\rightmap{F^X}&({\cal A},I)\g\Mod\\
   \shortlmapdown{\Res}&&\shortlmapdown{\Res}\\
   (({\cal A}^{\cirmin })^{X^\cirmin },(I^{\cirmin })^{X^\cirmin })\g\Mod&
   \rightmap{F^{X^\cirmin  }}&({\cal A}^{\cirmin },I^{\cirmin })\g\Mod.\\
  \end{matrix}$$
    Moreover, if $M\in ({\cal A},I)\g\Mod$ is such that
  $\Res(M)\cong F^{X^\cirmin }(N')$, for some
  $N'\in(({\cal A}^{\cirmin })^{X^\cirmin },(I^{\cirmin })^{X^\cirmin })\g\Mod$, then $M\cong F^X(N)$,
  for some $N\in ({\cal A}^{X},I^{X})\g\Mod$.
\end{lemma}

\begin{proof} Assume that $X^\cirmin $ is complete triangular admissible  $B^\cirmin $-module. Then, directly from (\ref{P: completez de X para la reduccion}), we have the associated  full and faithful functor $F^{X^\cirmin }:(({\cal A}^\cirmin )^{X^\cirmin },(I^\cirmin )^{X^\cirmin })\g\Mod\rightmap{}({\cal A}^\cirmin ,I^\cirmin )\g\Mod$,
 where $I^\cirmin =fIf$.

We have the projection map $\pi:B\rightmap{}B^\cirmin $, which enable us to consider $X^\cirmin $ as a $B$-module, so we have the splitting
$$\Gamma^\cirmin :=
\End_{B^\cirmin }(X^\cirmin )^{op}=\End_B(X^\cirmin )^{op}=S^\cirmin \oplus P^\cirmin ,$$
and $X^\cirmin $ is a triangular admissible $B^\cirmin $-module, relative to the given splitting of $\Gamma^\cirmin $. We have $P^{\cirmin  *}=\Hom_{S^\cirmin }(P^\cirmin ,S^\cirmin )$ and the corresponding comultiplication $\mu^\cirmin :P^{\cirmin  *}\rightmap{}P^{\cirmin  *}\otimes_{S^\cirmin }P^{\cirmin  *}$, which determines the structure of the ditalgebra $(B^\cirmin ,0)^{X^\cirmin }=(T_{S^\cirmin }(P^{\cirmin  *}),\delta(\mu^\cirmin ))=(B,0)^{X^\cirmin }$, see (\ref{R: complete B-mods}).
Since $\,_{B^\cirmin }X^\cirmin $ is complete, the first row in the following commutative diagram is a full and faithful functor
$$\begin{matrix}
(B^\cirmin ,0)^{X^\cirmin }\,\g\Mod&
\rightmap{F^{\prime X^\cirmin }}&
B^\cirmin \g\Mod\\
   \hbox{\,\,\,\,\,\,}\parallel&&\shortlmapdown{F_\pi}\\
  (B,0)^{X^\cirmin }\,\g\Mod&
  \rightmap{F^{X^\cirmin }}&B\g\Mod,\\
  \end{matrix}$$
  so the second one is full and faithful and $X^\cirmin $ is a complete admissible  $B$-module.

Since we have a surjective morphism of algebras $B\rightmap{}Re_{_\Omega}$, from (\ref{R: complete B-mods})(2), the admissible $B$-module
 $Re_{_\Omega}$ is complete. Since the $B$-module $X:=Re_{_\Omega}\oplus X^\cirmin $ satisfies
 $\Hom_B({\cal I}_{Re_{_\Omega}},{\cal I}_{X^\cirmin })=0$ and
 $\Hom_B({\cal I}_{X^\cirmin },{\cal I}_{Re_{_\Omega}})=0$, by (\ref{R: complete B-mods})(3),  it  is complete.
 So,  the associated functor
 $F^{\prime X}:(B,0)^X\g\Mod\rightmap{}B\g\Mod$ is full and faithful.
 We have the ideal $I_0=B\cap I$ of $B$ and the associated full and faithful functors
 $$F^{\prime X}:((B,0)^X,I_0^X)\g\Mod\rightmap{}(B/I_0)\g\Mod,$$
and $F^{X}:({\cal A}^X,I^X)\g\Mod\rightmap{}({\cal A},I)\g\Mod$,
 see (\ref{P: completez de X para la reduccion}).

\medskip
In order to compare $(({\cal A}^{\cirmin })^{X^\cirmin },(I^{\cirmin })^{X^\cirmin })$ and
$(({\cal A}^X)^\cirmin ,(I^X)^\cirmin )$, in the following we recall their constituents.
We need to consider the splitting $\Gamma=\End_B(X)=S\oplus P$, where $S=Re_{_\Omega}\times S^\cirmin $ and $P$ denotes the $S$-$S$-bimodule obtained from the $S^\cirmin \g S^\cirmin $-bimodule $P$ by restriction through the projection map $S\rightmap{}S^\cirmin $.
Observe that $W_0=fW_0=fW_0f\oplus fW_0e_{_\Omega}=W'_0\oplus
(W''_0\oplus W_0e_{_\Omega})$.
The layer of the weak ditalgebra ${\cal A}^X$ is $(S,W^X)$, where
$$W_0^X=X^*\otimes_R(W''_0\oplus W_0e_{_\Omega})\otimes_RX
\hbox{  \ and \  }
 W_1^X=(X^*\otimes_RW_1\otimes_RX)\oplus P^*,$$
 where $X^*=\Hom_S(X,S)$ and $P^*=\Hom_S(P,S)$.
 Notice that the $S$-$S$-bimodule $P^*$ can be identified with the  $S$-$S$-bimodule obtained from  the $S^\cirmin $-$S^\cirmin $-bimodule
 $P^{\cirmin  *}=\Hom_{S^\cirmin }(P^\cirmin ,S^\cirmin )$
 by restriction through the projection map $S\rightmap{}S^\cirmin $.
 Moreover, if we denote by $X^{\cirmin  *}=\Hom_{S^\cirmin }(X^\cirmin ,S^\cirmin )$, then  the  $S$-$R$-bimodule $X^*$ can be identified with the direct sum $Re_{_\Omega}\oplus X^{\cirmin  *}$, where $X^{\cirmin  *}$ is considered as an $S$-$R$-bimodule by restriction through the projections $S\rightmap{}S^\cirmin $ and $R\rightmap{}Rf$.
  Hence,
  $$\begin{matrix}W_0^X&=& X^*\otimes_R(W''_0\oplus W_0e_{_\Omega})\otimes_RX\hfill\\
   &=&(Re_{_\Omega}\oplus X^{\cirmin  *})\otimes_R(W''_0\oplus W_0e_{_\Omega})
   \otimes_R(Re_{_\Omega}\oplus X^\cirmin )\hfill\\
   &=& (X^{\cirmin  *}\otimes_RW''_0\otimes_RX^\cirmin ) \oplus
   (X^{\cirmin  *}\otimes_RW_0e_{_\Omega}\otimes_RRe_{_\Omega})\hfill\\
  \end{matrix}$$
  and
  $$\begin{matrix}W_1^X&=& [X^*\otimes_RW_1\otimes_RX]\oplus P^*\hfill\\
   &=&[(Re_{_\Omega}\oplus X^{\cirmin  *})\otimes_RW_1
   \otimes_R(Re_{_\Omega}\oplus X^\cirmin )]\oplus P^{\cirmin  *}\hfill\\

   &=&
   Re_{_\Omega}\otimes_RW_1
   \otimes_R(Re_{_\Omega}\oplus X^\cirmin )\hfill\\
   &&
   \oplus [X^{\cirmin  *}\otimes_RW_1
   \otimes_R(Re_{_\Omega}\oplus X^\cirmin )]\oplus P^{\cirmin  *}\hfill\\
   &=&

    Re_{_\Omega}\otimes_Re_{_\Omega}W_1e_{_\Omega}
   \otimes_R(Re_{_\Omega}\oplus X^\cirmin )\hfill\\
   &&
   \oplus [X^{\cirmin  *}\otimes_RW_1
   \otimes_R(Re_{_\Omega}\oplus X^\cirmin )]\oplus P^{\cirmin  *}\hfill\\

    &=&
    Re_{_\Omega}\otimes_Re_{_\Omega}W_1e_{_\Omega}
   \otimes_RRe_{_\Omega}\hfill\\
   &&
   \oplus [X^{\cirmin  *}\otimes_RW_1
   \otimes_RRe_{_\Omega}]\hfill\\
   &&
   \oplus [X^{\cirmin  *}\otimes_RW_1
   \otimes_R X^\cirmin ]\oplus P^{\cirmin  *}\hfill\\
  \end{matrix}$$
  Hence,
 $e_{_\Omega}W_0^X=0$ and $e_{_\Omega}W_1^X=Re_{_\Omega}\otimes_Re_{_\Omega}W_1e_{_\Omega}
   \otimes_RRe_{_\Omega}=
 e_{_\Omega}W_1^Xe_{_\Omega}$.

Let us show that $e_\omega\not\in I^X$, for $\omega\in \Omega$. For this notice that
$e_\omega W_0=0$ and $e_\omega\not\in I$ implies that $e_\omega I=0$. Thus, $fI=I$.
We claim that, for any $\omega\in \Omega$, $\nu\in X^*$, $x\in X$, and $h\in I$, we have that
$e_\omega\sigma_{\nu,x}(h)=0$.
Indeed, for $h\in I=fI$, we have $h=fa$ with
$a\in I\subseteq A=T_R(W_0)$, thus $a=r+\gamma$, with $r\in R$ and
$\gamma\in [\oplus_nW_0^{\otimes n}]$, so $h=fr+f\gamma$. Since,
$e_\omega\sigma_{\nu,x}(fr)=e_\omega\nu(frx)$ with $frx\in X^\cirmin $, we get
$e_\omega\nu(frx)=0$. Moreover, from \cite{BSZ}(12.8)(3), we get
$e_\omega\sigma_{\nu,x}(f\gamma)\in e_\omega[\oplus_n(W_0^X)^{\otimes n}]=0$. Hence
$e_\omega\sigma_{\nu,x}(h)=e_\omega\sigma_{\nu,x}(fr)+e_\omega\sigma_{\nu,x}(f\gamma)=0$, as
claimed.
Now, if $e_\omega\in I^X$, we get
$e_\omega=\sum_v(s_v+\gamma_v)\sigma_{\nu_v,x_v}(h_v)\zeta_v$,
where $s_v\in S^{\cirmin }$, $\gamma_v\in [\oplus_n(W_0^X)^n]$,
$\zeta_v\in A^X$, $\nu_v\in X^*$, $x_v\in X$, and $h_v\in I$.
Thus, from our precedent claim, we obtain
$e_\omega=\sum_vs_ve_\omega\sigma_{\nu_v,x_v}(h_v)\zeta_v=0$, a contradiction.
Thus $e_\omega\not\in I^X$, for all $w\in \Omega$.

Then, $\Omega$ is a multiple source of $({\cal A}^X,I^X)$ and we have the interlaced weak ditalgebra $({\cal A}^{X\cirmin },I^{X\cirmin })$ obtained by suppression of the multiple source idempotent $e_{_\Omega}\in S$. Here, we have $T^{X\cirmin }=T_{S\hat{f}}(\hat{f}(W_0^X\oplus W_1^X)\hat{f})$, where $\hat{f}=1-e_{_\Omega}\in S$.

We also have $T^\cirmin =T_{Rf}(fWf)$ and $(T^{\cirmin })^{X^\cirmin }=T_{S^\cirmin }((fW_0f)^{X^\cirmin }\oplus (fW_1f)^{X^\cirmin })$. From the description given before for $W_0^X$, we get:
$$(fW_0f)^{X^\cirmin }=
X^{\cirmin  *}\otimes_{Rf}W''_0\otimes_{Rf}X^\cirmin =
X^{\cirmin  *}\otimes_{R}W''_0\otimes_{R}X^\cirmin =\hat{f}W_0^X\hat{f}.$$
From the description given before for $W_1^X$, we have
  $(W_1^\cirmin )^{X^\cirmin }=(fW_1f)^{X^\cirmin }=[X^{\cirmin  *}\otimes_{Rf}fW_1f\otimes_{Rf}X^\cirmin ]\oplus P^{\cirmin  *}= \hat{f}W_1^X\hat{f}$.
    Therefore, we get $(T^{\cirmin })^{X^\cirmin }=T^{X\cirmin }$.

 The choose, naturally, the fixed dual basis for $X^*$ which is obtained from the dual basis of
 $X^{\cirmin  *}$ by adding, for each $\omega\in \Omega$, the elements $x_\omega:=e_\omega\in X$ and $\nu_\omega\in X^*$ given by $\nu_\omega(e_\omega)=e_\omega$ and $\nu_\omega(X^\cirmin +\sum_{w'\not=w}ke_{\omega'})=0$. Then, we have the commutative squares
  $$\begin{matrix}
    X^*&\rightmap{\lambda}&P^*\otimes_SX^*\\
    \shortlmapup{}&&\shortlmapup{}\\
    X^{\cirmin  *}&\rightmap{\lambda^\cirmin }&P^{\cirmin  *}\otimes_{S^\cirmin }X^{\cirmin  *}\\
   \end{matrix}
   \hbox{  \ \ and \ \  }
   \begin{matrix}
    X&\rightmap{\rho}&X\otimes_SP^*\\
    \shortlmapup{}&&\shortlmapup{}\\
    X^\cirmin &\rightmap{\rho^\cirmin }&X^\cirmin \otimes_{S^\cirmin }P^{\cirmin  *}\\
   \end{matrix}$$
   where the superscript $\cirmin $ has been attached to the maps associated to the admissible $B^\cirmin $-module $X^\cirmin $ and the vertical arrows denote inclusions. Indeed, this follows immediately from the expressions for these maps given in (7.6) in terms of the chosen dual basis.
  Similarly, for $\nu\in X^{\cirmin  *}$ and $x\in X^\cirmin $, we have the commutative square
 $$\begin{matrix}
    T&\rightmap{\sigma_{\nu,x}}&T^X\\
    \shortlmapup{}&&\shortlmapup{}\\
    T^{\cirmin }&\rightmap{\sigma^\cirmin _{\nu,x}}&(T^{\cirmin })^{X^\cirmin }\\
   \end{matrix}$$
  where, again, the vertical arrows denote inclusions.

Then, for  $\nu\in X^{\cirmin  *}$, $x\in X^\cirmin $, and
$ w\in B^\cirmin  W''_0B^\cirmin \cup B^\cirmin  W_1B^\cirmin $, we get
$$\begin{matrix}\delta^{X\cirmin }(\nu\otimes w\otimes x)&=&
   \lambda(\nu)\otimes w\otimes x+\sigma_{\nu,x}(\delta(w))+
   (-1)^{\deg w+1}\nu\otimes w\otimes \rho(x)\hfill\\
   &=&(\delta^{\cirmin })^{X^\cirmin }(\nu\otimes w\otimes x)\hfill\\
  \end{matrix}$$
and we obtain
${\cal A}^{X\cirmin }
=
(T^{X\cirmin },\delta^{X\cirmin })
=
((T^{\cirmin })^{X^\cirmin },(\delta^{\cirmin })^{X^\cirmin })
=
({\cal A}^{\cirmin })^{X^\cirmin }$.

Now, let us prove that $(I^\cirmin )^{X^\cirmin }=(I^X)^\cirmin $, that is $(fIf)^{X^\cirmin }=\hat{f}I^X\hat{f}$. As before,
since $e_{\Omega}W_0^X=0$ and $e_\omega\not\in I^X$, we obtain that $\hat{f}I^X=I^X$.
So we want to show the equality $I^X\hat{f}=(If)^{X^\cirmin }$. Given a generator
$\sigma_{\nu,x}(hf)$ of $(If)^{X^\cirmin }$, with $x\in X^{\cirmin }$,  $\nu\in X^{\cirmin  *}$, and $h\in I$, from \cite{BSZ}(12.8)(3), we obtain
$\sigma_{\nu,x}(hf)=\sum_i\sigma_{\nu,x_i}(h)\sigma_{\nu_i,x}(f)=
\sum_i\sigma_{\nu,x_i}(h)\nu_i(fx)$. But, since $fx\in X^\cirmin $,
we have $\nu_i(fx)=\nu_i(fx)\hat{f}$. Thus $\sigma_{\nu,x}(hf)\in I^X\hat{f}$,
and
$(If)^{X^\cirmin }\subseteq I^X\hat{f}$. Now, consider a generator
$\sigma_{\nu,x}(h)$ of the ideal $I^X$ of $A^X$, thus $h\in I$, $x\in X$, and $\nu\in X^*$. We know that $h=\sum_{\omega\in \Omega}he_\omega+hf$, so
$$\begin{matrix}
\sigma_{\nu,x}(h)\hat{f}
&=&
[\sum_\omega\sigma_{\nu,x}(he_\omega)+\sigma_{\nu,x}(hf)]\hat{f}\hfill\\
&=&
[\sum_{\omega,i}\sigma_{\nu,x_i}(h)\sigma_{\nu_i,x}(e_\omega)
+
\sum_i\sigma_{\nu,x_i}(hf)\sigma_{\nu_i,x}(f)]\hat{f}\hfill\\
&=&
[\sum_i\sigma_{\nu,x_i}(hf)\nu_i(fx)]\hat{f}
=
\sum_i\sigma_{\nu,x_i}(hf)\nu_i(fx)\in (If)^{X^\cirmin } \hfill\\
\end{matrix}$$
since $\sigma_{\nu_i,x}(e_\omega)\in ke_\omega$ and $\nu_i(fx)\in S^{\cirmin }\hat{f}$.
Now, notice that since $A^{X\cirmin }$ is a convex subalgebra of $A^X$, the ideal $I^X\hat{f}=\hat{f}I^X\hat{f}$ of $\hat{f}A^X\hat{f}=A^{X\cirmin }$ is generated by the elements $\sigma_{\nu,x}(x)\hat{f}$ analyzed before. So
$I^X\hat{f}=(If)^{X^\cirmin }$.

From the preceding facts,  we get
$({\cal A}^{X\cirmin },I^{X\cirmin })=(({\cal A}^{\cirmin })^{X^\cirmin }, (I^{\cirmin })^{X^\cirmin })$, as claimed.
 Suppose that $M\in ({\cal A},I)\g\Mod$ is such that
  $\Res(M)\cong F^{X^\cirmin }(N')$, for some
  $N'\in({\cal A}^{X\cirmin },I^{X\cirmin })\g\Mod$. So we have an isomorphism of left $B^\cirmin $-modules $fM\cong X^\cirmin \otimes_{S^\cirmin }N'$.
  But the $B$-module $M$ admits the decomposition $M=e_\Omega M\oplus fM$ and
  $X\otimes_Se_\Omega M\cong Re_\Omega\otimes_Se_\Omega M\cong e_\Omega M$.
  Consider the $S$-module $N:=e_\Omega M\oplus N'$, where $N'$ is considered as an $S$-module by restriction through the projection map $S\rightmap{}S^\cirmin $. Therefore we have an isomorphism of left $B$-modules $X\otimes_SN=X\otimes_{S}(e_\Omega M\oplus N')\cong (Re_{_\Omega}\otimes_Se_{_\Omega}M)\oplus (X^\cirmin \otimes_{S^\cirmin }N')\cong e_\Omega M\oplus fM$,
  which, by (\ref{P: (AX, IX)})(4), implies that $M\cong F^X(\overline{N})$,
  for some $\overline{N}\in ({\cal A}^{X},I^{X})\g\Mod$.
\end{proof}

 \begin{example} Let us start from a very simple interlaced weak  ditalgebra $({\cal A},I)$, where ${\cal A}$ is the tensor algebra of the following biquiver (without dashed arrows and with trivial derivation) and ideal $I$ generated by the path $\beta\alpha$.
 $${\cal A} : \quad \quad \xymatrix{ &  _1 \ar[r]^{\alpha} &  _2 \ar[rd]^{\beta} & \\
 _0 \ar[ru]^{\lambda} & & &  _3} \quad \quad \beta \alpha = 0$$
 In a first step, we apply the edge-reduction of $\beta$ to $({\cal A},I)$, which corresponds to the reduction of $({\cal A},I)$ by a suitable admissible module $X$ see \cite{BSZ}(23.18). The interlaced weak ditalgebra obtained $({\cal A}^\beta,I^\beta)$ has the following biquiver with differential $\delta$ and ideal $I^\beta$ determined by the following data.
$${\cal A}^{\beta} : \quad \quad \xymatrix{
                               &                                                           &  _2                          & \\
                               &  _1 \ar[ru]^{\alpha _1} \ar[r]_{\alpha _2} &  _{3'}     \ar@{-->}[u]_{\xi}           & \\
 _0 \ar[ru]^{\lambda} & & &  _3 \ar@{-->}[lu]_{\eta}} \quad \quad \delta \left( \alpha _1 \right) = \xi \alpha _2 , \quad \alpha _2 = 0$$
Now, we can factor out from $({\cal A}^\beta,I^\beta)$ the solid arrow $\alpha_2$, which belongs to the interlaced ideal $I^\beta$, to obtain the interlaced weak ditalgebra $({\cal A}^{\beta q},I^{\beta q})$, with $I^{\beta q}=0$, determined by the following data

$${\cal A}^{\beta q} : \quad \quad \xymatrix{
                               &                                                           &  _2                          & \\
                               &  _1 \ar[ru]^{\alpha _1}                          &  _{3'}     \ar@{-->}[u]_{\xi}           & \\
 _0 \ar[ru]^{\lambda} & & &  _3 \ar@{-->}[lu]_{\eta}} \quad \quad \delta \left( \alpha _1 \right) = 0$$
Then, we can proceed to apply the edge-reduction of $\alpha_1$ to the last  ditalgebra to obtain $({\cal A}^{\beta q\alpha_1}, I^{\beta q\alpha_1})$
given by the following data (with trivial ideal)

$${\cal A}^{\beta q \alpha _1} : \quad \quad \xymatrix{
& _1 & _2 \ar@{-->}[ld]^{\eta `} & & \\
_0 \ar[ru]^{\lambda _1} \ar[r]^{\lambda _2} & _{2'} \ar@{-->}[u] _{\zeta `} & & _{3'} \ar@{-->}[ll]^{\xi _1} \ar@{-->}[lu]_{\xi _2} & _{3} \ar@{-->}[l]_{\eta}} \quad \delta \left( \lambda _1 \right) = \zeta ` \lambda _2 $$
The biquiver of the ditalgebra obtained from the preceding one by supressing the source idempotent $e_0$ is the following minimal seminested ditalgebra
$${\cal A}^{\beta q\alpha _1\cirmin} : \quad \quad \xymatrix{
 _1 & _2 \ar@{-->}[ld]^{\eta `} & & \\
 _{2'} \ar@{-->}[u] _{\zeta `} & & _{3'} \ar@{-->}[ll]^{\xi _1} \ar@{-->}[lu]_{\xi _2} & _{3} \ar@{-->}[l]_{\eta}}$$

On the other hand, we can start from the original interlaced weak ditalgebra $({\cal A},I)$ and suppress first the source idempotent $e_0$ to obtain
$${\cal A}^{\cirmin} : \quad \quad \xymatrix{  _1 \ar[r]^{\alpha} &  _2 \ar[rd]^{\beta} & \\
  & &  _3} \quad \quad \beta \alpha = 0$$
Then we can apply edge-reduction of the arrow $\beta$ to obtain
$${\cal A}^{\cirmin \beta} : \quad \quad \xymatrix{
                                                         &  _2                          & \\
 _1 \ar[ru]^{\alpha _1} \ar[r]_{\alpha _2} &  _{3'}     \ar@{-->}[u]_{\xi}           & \\
  &   &  _3 \ar@{-->}[lu]_{\eta}}
  \quad \quad \delta \left( \alpha _1 \right) = \xi \alpha _2 , \quad \alpha _2 = 0$$
 and then factor out the arrow $\alpha_2$ to get
 $${\cal A}^{\cirmin \beta q} : \quad \quad \xymatrix{
   &  _2                          & \\
  _1 \ar[ru]^{\alpha _1}                          &  _{3'}     \ar@{-->}[u]_{\xi}           & \\
 & &  _3 \ar@{-->}[lu]_{\eta}} \quad \quad \delta \left( \alpha _1 \right) = 0$$
Finally, if we apply edge-reduction of the arrow $\alpha_1$ to the preceding ditalgebra, we obtain
$${\cal A}^{\cirmin \beta q \alpha _1} : \quad \quad \xymatrix{
 _1 & _2 \ar@{-->}[ld]^{\eta `} & & \\
 _{2'} \ar@{-->}[u] _{\zeta `} & & _{3'} \ar@{-->}[ll]^{\xi _1} \ar@{-->}[lu]_{\xi _2} & _{3} \ar@{-->}[l]_{\eta}}$$
which coincides with ${\cal A}^{ \beta q\alpha _1\cirmin}$.
 \end{example}

   \section{Reduction to minimal ditalgebras}

This section is devoted to the proof of the following theorem, our main result for
non-wild ${\cal P}$-oriented interlaced weak ditalgebras $({\cal A},I)$, see (\ref{D: biquiver P-orientado}).
This result reduces the study of the  $({\cal A},I)$-modules  with  dimension bounded by some $d\in\hueca{N}$
 to the study of modules over a minimal ditalgebra
obtained from $({\cal A},I)$ by a finite number of reductions. This
is the main step to the proof of the tame-wild dichotomy for this special type of weak interlaced ditalgebras $({\cal A},I)$.

\begin{remark}  Assume that ${\cal P}$ is a finite preordered set and that ${\cal A}=(T,\delta)$ is a ${\cal P}$-oriented weak ditalgebra with layer $(R,W)$, where $R=\prod_{i\in {\cal P}}ke_i$. Then, if $\Omega\in \overline{\cal P}$ is any minimal element, we have that  $\Omega\subseteq {\cal P}$ satisfies $e_{_\Omega} W_0=0$, $e_{_\Omega} W_1=e_{_\Omega} W_1e_{_\Omega}$, and  $Re_\omega=ke_\omega$, for all $\omega\in \Omega$. That is, $\Omega$ is a multiple source of ${\cal A}$.
\end{remark}

\begin{theorem}\label{T: parametrización de módulos de dim acotada}
 Assume that the ground field $k$ is algebraically closed and let ${\cal P}$ be a finite preordered set.
 Assume $({\cal A},I)$ is a ${\cal P}$-oriented triangular interlaced weak ditalgebra as in
(\ref{D: triangular interlaced weak ditalgebras}), where $I$ is an
 ideal of $A$ contained in the radical of $A$.
Suppose that $({\cal A},I)$ is not wild and take $d\in \hueca{N}$.
Then, there is a finite sequence of reductions
$$({\cal A},I)\mapsto ({\cal A}^{z_1},I^{z_1})\mapsto\cdots\mapsto
({\cal A}^{z_1z_2\cdots z_t},I^{z_1z_2\cdots z_t})$$
of type $z_i\in \{a,d,r,q,X\}$ such that ${\cal A}^{z_1z_2\cdots z_t}$ is a
minimal ditalgebra, we have  $I^{z_1z_2\cdots z_t}=0$, and almost every
$({\cal A},I)$-module $M$ with  $\dim_k
M\leq d$ has the form
$M\cong F^{z_1}\cdots F^{z_t}(N)$, for some $N\in ({\cal A}^{z_1\cdots
z_t},I^{z_1\cdots z_t})\g\Mod$.
\end{theorem}

\begin{proof} Let $(R,W)$ be the layer of ${\cal A}$. By assumption $R=\prod_{i\in {\cal P}}ke_i$ is a
product of fields. Since $I\subseteq \rad A$, we have
 that no idempotent $e_i$ belongs to $I$.

Since ${\cal A}$ is ${\cal P}$-oriented, if we choose a minimal element $\Omega$ in the poset $\overline{\cal P}$, we obtain a multiple source $\Omega\subseteq {\cal P}$ for the interlaced weak ditalgebra
$({\cal A},I)$. So, $e_{_\Omega} W_0=0$ , $e_{_\Omega} W_1=e_{_\Omega} W_1e_{_\Omega}$, and  $Re_\omega=ke_\omega$ with $e_\omega\not\in I$, for all $\omega\in \Omega$.

 We will  proceed by induction on the number of points of the interlaced weak
 ditalgebra $({\cal A},I)$.
The base of the induction is clear: if there is only one point $i$ in ${\cal A}$, we have ${\cal P}=\Omega=\{i\}$ and ${\cal A}$ admits no solid arrows, so ${\cal A}$ is a minimal ditalgebra, and we are done. The same holds whenever ${\cal P}=\Omega$, so we can assume that this is not the case.

Assume $n\in \hueca{N}$, with $n>1$. Suppose that the theorem holds for any
non-wild ${\cal P}'$-oriented  triangular
interlaced weak ditalgebra $({\cal A}',I')$ with $m<n$ points, where $I'$ is an ideal of $A'$ such that $I'\subseteq \rad A'$. Suppose that
$({\cal A},I)$ is a non-wild ${\cal P}$-oriented triangular interlaced weak ditalgebra
with multiple source $\Omega$ and $n$ points,
where $I$ is an ideal of $A$ with $I\subseteq \rad A$, and take $d\in \hueca{N}$.
Now,  consider the non-wild triangular interlaced weak ditalgebra
$({\cal A}^\cirmin ,I^\cirmin )$ obtained from $({\cal A},I)$ by supression of the multiple source idempotent $e_{_\Omega}$. Consider the non-empty subset
${\cal P}^\cirmin ={\cal P}\setminus \Omega$ of ${\cal P}$ with the induced preorder and notice that the biquiver $\hueca{B}^\cirmin $ of
${\cal A}^\cirmin $ is ${\cal P}^\cirmin $-oriented.
So $({\cal A}^ \cirmin ,I^\cirmin )$ is a
${\cal P}^\cirmin$-oriented non-wild triangular interlaced weak ditalgebra where $I^\cirmin$ is an ideal of $A^\cirmin$ with $I^\cirmin\subseteq \rad A^\cirmin$, to which we can apply our induction hypothesis.
So, for a fixed $d\in \hueca{N}$, to obtain a finite sequence
of reductions
$$({\cal A}^\cirmin ,I^\cirmin )\mapsto ({\cal A}^{\cirmin  z_1},I^{\cirmin  z_1})\mapsto
\cdots\mapsto
({\cal A}^{\cirmin  z_1z_2\cdots z_t},I^{\cirmin  z_1z_2\cdots z_t})$$
of type $z_i\in \{a,d,r,q,X\}$ such that
${\cal A}^{\cirmin  z_1z_2\cdots z_t}$ is a
minimal ditalgebra, we have  $I^{\cirmin  z_1z_2\cdots z_t}=0$,
and any $({\cal A}^\cirmin ,I^\cirmin )$-module $M''$ with  $\dim_k
M''\leq d$ has the form
$M''\cong F^{z_1}\cdots F^{z_t}(N'')$, for some $N''\in ({\cal A}^{\cirmin
z_1\cdots
z_t},I^{\cirmin  z_1\cdots z_t})\g\Mod$.

From the lemmas in \S\ref{Reducs with source}, we can perform a corresponding finite sequence of reductions $$({\cal A},I)\mapsto ({\cal A}^{z_1},I^{ z_1})\mapsto
\cdots\mapsto
({\cal A}^{z_1z_2\cdots z_t},I^{ z_1z_2\cdots z_t})$$
of the same type $z_i\in \{a,d,r,q,X\}$ as those applied successively to $({\cal A}^\cirmin ,I^\cirmin )$, such that
$({\cal A}^{z_1\cdots z_t\cirmin },I^{z_1\cdots z_t\cirmin })=
({\cal A}^{\cirmin  z_1\cdots z_t},I^{\cirmin  z_1\cdots z_t})$
and there is a commutative
diagram
$$\begin{matrix}({\cal
A}^{z_1\cdots z_t},I^{z_1\cdots z_t})\g\Mod&\rightmap{F^{z_t}}
&\cdots&\rightmap{F^{z_1}}&({\cal A},I)\g\Mod\\
\shortlmapdown{\Res}&&&&\shortlmapdown{\Res}\\
  ({\cal A}^{z_1\cdots z_t\cirmin },I^{z_1\cdots z_t\cirmin })\g\Mod&\rightmap{F^{\cirmin
z_t}}&\cdots&\rightmap{F^{\cirmin  z_1}}&
  ({\cal A}^\cirmin ,I^\cirmin )\g\Mod
  \end{matrix}$$
where ${\cal A}^{z_1\cdots z_t\cirmin }$  is a minimal ditalgebra,
$I^{z_1\cdots z_t\cirmin }=0$, and any
$M'\in({\cal A}^\cirmin ,I^\cirmin )\g\Mod$ with $\dim_k M'\leq d$ is of the form
$F^{\cirmin  z_1}\cdots F^{\cirmin  z_t}(N')\cong M'$.
If $M\in ({\cal A},I)\g\Mod$ and
$\dim_k M\leq d$, then $\dim_k\Res(M)\leq d$, and
$\Res(M)\cong F^{\cirmin  z_1}\cdots F^{\cirmin  z_t}(N')$. Thus, for some $N\in ({\cal
A}^{z_1\cdots z_t},I^{z_1\cdots z_t})\g\Mod$, we have
$F^{z_1}\cdots F^{z_t}(N)\cong M$. Moreover, $\dim_k N\leq d'$,
for some fixed $d'$ which depends on $d$. For the sake of simplicity, set
$({\cal A}^\prime,I^\prime)
:=({\cal A}^{z_1\cdots z_t},I^{z_1\cdots z_t})$.

Here, we have that ${\cal A}^{\prime\cirmin }$ is a minimal ditalgebra, $I^{\prime\cirmin }=0$,
and $({\cal A}',I')$ is a  triangular
interlaced weak ditalgebra with multiple source $\Omega$. If $(R',W')$ denotes the layer of
${\cal A}'$,
say with $R'=\prod_{i\in {\cal P}'}R'{e_i}$,
and we define $f=1-e_{_\Omega}=\sum_{j\in {\cal P}'\setminus\Omega}e_j\in R'$,
we know that $e{_\Omega}W'_0=0$ and $fW'_0f=0$, so $W'_0e_{_\Omega}=W'_0$.
Then, we get that  $({\cal A}',I')$ is a
stellar triangular interlaced weak ditalgebra with stars centers $\{e_\omega\}_{\omega\in \Omega}$.  Then, apply
(\ref{T: de stellar wditalg to seminested ditalg}) to $d'$, to obtain a
composition of full and faithful reduction functors
$$\begin{matrix}({\cal
A}'',0)\g\Mod&\rightmap{F^{y_s}}&\cdots&\rightmap{F^{y_1}}&({\cal
A}',I')\g\Mod\\
  \end{matrix}$$
such that ${\cal A}''$ is a non-wild seminested ditalgebra and any
$N\in ({\cal A}',I')\g\Mod$ with $\dim_k N\leq d'$ is of the form
$F^{y_1}\cdots
F^{y_s}(L)\cong N$ for some
$L\in ({\cal A}^{\prime\prime},0)\g\Mod$. Moreover, $\dim_k L\leq d''$,
for some fixed $d''$ depending on $d'$.
Then, from \cite{BSZ}(28.22), there is a minimal ditalgebra ${\cal B}$ and a
composition of full and faithful reduction functors $F:{\cal B}\g\Mod\rightmap{}{\cal A}''\g\Mod$
such that for any ${\cal A}''$-module $L$ with $\dim_k L\leq d''$,
there is a ${\cal
B}$-module $K$ such that $F(K)\cong L$.

Thus, given $M\in ({\cal A},I)\g\Mod$ with $\dim_k M\leq d$, we obtain
$$M\cong F^{z_1}\cdots F^{z_t}F^{y_1}\cdots F^{y_s}F(K).$$
\end{proof}

\section{Tame and wild dichotomy}

In this section we proceed to the proofs of our main results in the case of admissible homological systems.
In order to be precise, we need first to adapt some  known facts on tame
ditalgebras to the case of interlaced weak ditalgebras.

The following lemma is probably known, but we include a proof for the
sake of the reader.

\begin{lemma}\label{L: tame-wild disjointness for cats of modules}
Let $\Lambda$ be a finite-dimensional algebra over an algebraically closed
field.
Let ${\cal C}$ be a full subcategory of $\Lambda\g\mod$,
closed under direct summands and direct sums. Then, the category ${\cal C}$ can
not be simultaneously tame and wild.
\end{lemma}

\begin{proof} The proof is essentially the same proof given on page 366
of \cite{BSZ}. Consider the usual variety $\mod_\Lambda(\underline{d})$ of
$\Lambda$-modules with dimension vector $\underline{d}$, where the algebraic
group
$\hueca{G}_{\underline{d}}$ acts in such a way that two
$\underline{d}$-dimensional $\Lambda$-modules are isomorphic iff their
corresponding points (denoted with the same symbols $M$ and $N$) satisfy
that  $N=hM$, for some $h\in \hueca{G}_{\underline{d}}$. If ${\cal C}$
is wild, there is a morphism of varieties
$\varphi:k^2=\mod_{k\langle x,y\rangle}(1)\rightmap{}\mod_\Lambda(\underline{d})$
induced by the functor
$F:k\langle x,y\rangle\g\mod\rightmap{}\Lambda\g\mod$, with image
in ${\cal C}$, which is provided by the wildness of ${\cal C}$ as in
 (\ref{D: tame and wild}).  Since $F$
preserves indecomposables and isomorphism classes, each $\varphi(\lambda,\mu)$
represents an indecomposable $\Lambda$-module in ${\cal C}$ and
$\varphi(\lambda,\mu)\not\in
\hueca{G}_{\underline{d}}\varphi(\lambda',\mu')$,
for all different pairs  $(\lambda,\mu),(\lambda',\mu') \in k^2$.
Consider, for each $\lambda\in k$, the curve
$\varphi_\lambda:k\rightmap{}\mod_{\Lambda}(\underline{d})$, defined by
$\varphi_\lambda(\mu)=\varphi(\lambda,\mu)$, for $\mu\in k$.

 If ${\cal C}$  is tame,  there is a finite number of curves
$\{\gamma_i:E_i\rightmap{}\mod_{\Lambda}(\underline{d})\}_{i=1}^m$,
defined on cofinite subsets $E_i$ of $k$, such that every
$\underline{d}$-dimensional indecomposable $\Lambda$-module in ${\cal C}$
is represented by a point in $\cup_{i=1}^m\hueca{G}_{\underline{d}}
\gamma_i(E_i)$.

Then, for each $\lambda\in k$,    $\varphi_\lambda(k)\subseteq
\cup_{i=1}^m\hueca{G}_{\underline{d}}\gamma_i(E_i)$. It follows that
$\varphi_\lambda(D^\lambda)\subseteq
\hueca{G}_{\underline{d}} \gamma_i(E_i)$, for some cofinite subset
$D^\lambda\subseteq k$ and some $i$ depending on $\lambda$.
Then, $\gamma_i(E_i^\lambda)\subseteq
\hueca{G}_{\underline{d}}\varphi_\lambda(D^\lambda)$, for some cofinite subset
$E_i^\lambda\subset E_i$.
Since we are dealing with finitely many curves $\gamma_1,\ldots,\gamma_m$, then
there is   $\lambda'\not=\lambda$ such that
$\varphi_{\lambda'}(D^{\lambda'})\subseteq
\hueca{G}_{\underline{d}} \gamma_i(E_i)$, for the same $i$. Since $E_i^\lambda$
is cofinite in $E_i$ and $D^{\lambda'}$ is infinite,
there exists $\mu\in D^{\lambda'}$ such that
$\varphi_{\lambda'}(\mu)\in
\hueca{G}_{\underline{d}}\gamma_i(E_i^\lambda)\subseteq
\hueca{G}_{\underline{d}}\varphi_{\lambda}(D^\lambda)$. This entails a
contradiction.
\end{proof}

\begin{definition}\label{D: tame weak dit} An interlaced weak
ditalgebra $({\cal A},I)$ is called \emph{tame} iff, for each $d\in \hueca{N}$,
there is a finite collection $\{(\Gamma_i,Z_i)\}_{i=1}^m$, where
$\Gamma_i=k[x]_{f_i}$ and $Z_i$ is an $(A/I)$-$\Gamma_i$-bimodule, which is free
of finite rank as a right $\Gamma_i$-module, such that for every indecomposable
$M\in ({\cal A},I)\g\Mod$ with $\dim_k M\leq d$, there are an $i\in [1,m]$ and a
simple $\Gamma_i$-module $S$ such that
$Z_i\otimes_{\Gamma_i}S\cong M$ in $({\cal A},I)\g\Mod$.
\end{definition}

The condition given in the last definition can be rearranged into
 an equivalent one, where the simple $\Gamma_i$-module $S$
is replaced by an indecomposable $\Gamma_i$-module $N$.
See \cite{BSZ}(27.2).

\begin{definition}\label{D: strictly tame weak dit} An interlaced weak
ditalgebra $({\cal A},I)$ is called \emph{strictly tame} iff, for each $d\in \hueca{N}$,
there is a finite collection $\{(\Gamma_i,Z_i)\}_{i=1}^m$, where
$\Gamma_i=k[x]_{f_i}$ and $Z_i$ is an $(A/I)$-$\Gamma_i$-bimodule, which is free
of finite rank as a right $\Gamma_i$-module, such that for almost every indecomposable
$M\in ({\cal A},I)\g\Mod$ with $\dim_k M\leq d$, there are an $i\in [1,m]$ and an indecomposable
 $\Gamma_i$-module $N$ such that
$Z_i\otimes_{\Gamma_i}N\cong M$ in $({\cal A},I)\g\Mod$. Moreover, the functors
$$\begin{matrix}\Gamma_i\g\mod&\rightmap{ \ Z_i\otimes_{\Gamma_i}- \ }&(A/I)\g\mod&
\rightmap{ \ L_{({\cal A},I)} \ }&({\cal A},I)\g\mod\\ \end{matrix}$$
are required to preserve isoclasses and indecomposables.
\end{definition}

It can be seen that in the last definition, we can replace the condition of being
``free of finite rank right $\Gamma_i$-module'', by the apparently weaker one
of being ``finitely generated right $\Gamma_i$-module", see \cite{BSZ}(27.5).

The following result is just a more concise formulation of Theorem
(\ref{T: main theorem dits}).

\begin{theorem}\label{T: tame wild dichotomy for directed inter weak dits}
Assume that the ground field is algebraically closed and that $({\cal A},I)$
is a ${\cal P}$-oriented triangular interlaced weak ditalgebra where $I$ is an ideal of $A$ contained in the radical of $A$. Then, $({\cal A},I)$ is
either tame or wild, but not both. Moreover,  $({\cal A},I)$ is tame iff it is
strictly tame.
\end{theorem}

\begin{proof} If $({\cal A},I)$ was  wild and tame
simultaneously,  we can easily adapt the geometric algebraic argument
given in \cite{BSZ}\S33 (pages 360--366) to get a contradiction:
the condition $\delta^2=0$ is irrelevant there (it is only used to have
a well defined category of modules) and we can use
$A/I$ instead of $A$ in that argument. Indeed, we already know that $({\cal A},I)$ is a Roiter seminested weak ditalgebra.

If $({\cal A},I)$ is not wild and $d\in \hueca{N}$, then we can apply
(\ref{T: parametrización de módulos de dim acotada})
to obtain a finite sequence of reductions
$$({\cal A},I)\mapsto ({\cal A}^{z_1},I^{z_1})\mapsto\cdots\mapsto
({\cal A}^{z_1z_2\ldots z_t},I^{z_1z_2\ldots z_t})$$
of type $z_i\in \{d,r,q,a,X\}$ such that $I^{z_1z_2\ldots z_t}=0$,
 ${\cal B}={\cal
A}^{z_1\cdots
z_t}$ is a
minimal ditalgebra, and almost  any $({\cal A},I)$-module $M$ with  $\dim_k
M\leq d$ has the form
$M\cong F(N)$, where $F=F^{z_1}\cdots F^{z_t}$,
for some $N\in {\cal B}\g\Mod$.

Given a point $i$ of ${\cal B}$, we either
have $Be_i=ke_i$ or $\Gamma_i=Be_i\not=ke_i$. The finite-dimensional indecomposable
${\cal B}$-modules are of
the form $S_i=ke_i$, for $i$ such that $Be_i=ke_i$,
or $\Gamma_i/(x-\lambda)^t$, with $\lambda\in k$ and $t\in \hueca{N}$,
for any $i$ such that $\Gamma_i=Be_i\not=ke_i$.

Apply (\ref{L: parameriz de funtores de reduccion}) to each $\Gamma_i$ and $F$
to obtain the following diagram which commutes up to isomorphism
$$\begin{matrix}
   \Gamma_i\g\Mod&\rightmap{ \ \Gamma_i\otimes_{\Gamma_i}- \ }&B\g\Mod&
   \rightmap{ \ L_{\cal B} \ }&({\cal B},0)\g\Mod\hfill\\
   \parallel&&\shortlmapdown{\underline{F}}&&\shortlmapdown{F}\\
    \Gamma_i\g\Mod&\rightmap{ \ Z_i\otimes_{\Gamma_i}- \ }&(A/I)\g\Mod&
    \rightmap{ \ L_{({\cal A},I)} \ }
    &({\cal A},I)\g\Mod\\
  \end{matrix}$$
where $Z_i:=F(\Gamma_i)$.  Since $F$ is a composition of full and faithful functors, the composition functor $FL_{\cal
B}(\Gamma_i\otimes_{\Gamma_i}-)$ preserves indecomposablility and isoclasses, and
the functor
 $L_{({\cal A},I)}(Z_i\otimes_{\Gamma_i}-)$ has the same properties. Moreover,
$Z_i$ is an $(A/I)$-$\Gamma_i$-bimodule which is finitely generated
 projective by the right. But $\Gamma_i$ is a principal ideal domain, so $Z_i$
 is a finitely generated free right $\Gamma_i$-module.
 If $N\cong \Gamma_i/(x-\lambda)^t$, then
 $L_{({\cal A},I)}(Z_i\otimes_{\Gamma_i}N)\cong
 FL_{\cal B}(\Gamma_i\otimes_{\Gamma_i}N)\cong M$.  We have shown that $({\cal A},I)$ is strictly tame. Finally, from the remark between definitions (\ref{D: tame weak dit}) and (\ref{D: strictly tame weak dit}), we see that strict tameness implies tameness.
\end{proof}

\noindent{\bf Proof of Theorem (\ref{T: main theorem qh-algebas}) in the admissible homological system case.}
\label{C: tame an wild dichotomy for catfinstandmods of qh-algebras}
Consider the the category of
$\Delta$-filtered modules ${\cal F}(\Delta)$ for a finite-dimensional  algebra
$\Lambda$ with admissible homological system $({\cal P},\leq,\{\Delta_i\}_{i\in {\cal P}})$. From
(\ref{L: tame-wild disjointness for cats of modules}), we already know that
${\cal F}(\Delta)$ can not be simultaneously tame and wild. By (\ref{R: F(Delta) tame (wild) iff F(Delta') tame (wild)}), it will be enough to show that ${\cal F}(\Delta')$ is either tame or wild, and it is tame iff it is strictly tame.

The ditalgebra of ${\cal Q}$ of (\ref{T: KKO}) is the quotient ${\cal Q}={\cal A}/J$, where $({\cal A},I)$ is the ${\cal P}$-oriented triangular interlaced weak ditalgebra associated to the given admissible homological system, constructed in \cite{hsb}\S5, and $J$ is the ideal of ${\cal A}$ generated by $I$, see \cite{hsb}(5.22) and (\ref{R: cocientes por ideales balanceados}). As remarked in \cite{hsb}(13.1), we know that $I\subseteq \rad A$. Here,  ${\cal A}=(T,\delta)=(T_A(V),\delta)$ is a triangular weak ditalgebra and we can adopt the notation of (\ref{R: cocientes por ideales balanceados}). Thus ${\cal Q}=\overline{\cal A}=(\overline{T},\overline{\delta})$, with $\overline{T}=T_{\overline{A}}(\overline{V})$, where $\overline{A}=A/I$ and $\overline{V}=V/(IV+\delta(I)+VI)$.

From (\ref{T: KKO}) there is
an equivalence of categories
 $$F:\overline{\cal A}\g\mod\rightmap{}{\cal F}(\Delta'),$$
Moreover,
 $F(M)\cong {\Gamma}\otimes_{\overline{A}} M$, for each $M\in
\overline{\cal A}\g\mod$, where
${\Gamma}=\End_{\overline{\cal A}}(\overline{A})^{op}$ is the right algebra of $\overline{\cal A}$,
 and the right $\overline{A}$-module ${\Gamma}$ is finitely generated
projective.

By (\ref{modules of interlaced ditalgebras}),
there is an equivalence
 $G:\overline{\cal A}\g\mod\rightmap{}({\cal A},I)\g\mod$,
which is the identity on objects.
By (\ref{T: tame wild dichotomy for directed inter weak dits}),
we know that $({\cal A},I)$ is wild or strictly tame.
It follows immediately that $\overline{\cal A}$ is wild
or strictly tame.

 If $\overline{\cal A}$  is wild, then we have the composition
$$\begin{matrix}k\langle x,y\rangle\g\mod&
\rightmap{Z\otimes_{k\langle x,y\rangle }-}&\overline{A}\g\mod&\rightmap{L_{\overline{\cal A}}}&
\overline{\cal A}\g\mod&\rightmap{F}&
{\cal F}(\Delta')\end{matrix}$$
where $Z$ realizes the wildness of $\overline{\cal A}$.
If $N\in k\langle x,y\rangle\g\mod$ is indecomposable, so is
$$F(Z\otimes_{k\langle x,y\rangle}N)\cong
{\Gamma}\otimes_{\overline{A}}Z\otimes_{k\langle x,y\rangle}N,$$
where ${\Gamma}\otimes_{\overline{A}}Z$ is a finitely generated
projective right module. Clearly,  the tensor
product by the preceding bimodule preserves isoclasses.
So, ${\cal F}(\Delta')$ is wild.

Notice that, for any  $N\in \overline{A}\g\mod$, we have
$\dim_kN\leq \dim_k{\Gamma}\otimes_{\overline{A}}N$. Indeed, since
${\Gamma}\otimes_{\overline{A}}-:\overline{A}\g\mod\rightmap{}
{\Gamma}\g\mod$ is an exact functor, if we have $\dim_kN=\sum_{i\in {\cal P}}m_i$, where $m_i$ denotes the multiplicity of the simple $\overline{A}$-module $S_i$ in the composition series of the $\overline{A}$-module $N$, then we also have that
$\dim_k({\Gamma}\otimes_{\overline{A}}N)=\sum_{i\in {\cal P}}m_i\dim_k\Delta'_i$, because from (\ref{T: KKO})(2), we know that $\Delta'_i\cong {\Gamma}
\otimes_{\overline{A}}S_i$.

Then, if $\overline{\cal A}$ is strictly tame
and $d\in \hueca{N}$, we have the composition functors preserving isoclasses and indecomposables
$$\begin{matrix}\Gamma_i\g\mod&
\rightmap{Z_i\otimes_{\Gamma_i}-}&\overline{A}\g\mod&\rightmap{L_{\overline{\cal A}}}&
\overline{\cal A}\g\mod&\rightmap{F}&
{\cal F}(\Delta')\end{matrix}$$
where $Z_1,\ldots,Z_t$ are the bimodules parametrizing the indecomposable
$\overline{\cal A}$-modules $N$ with $\dim_k N\leq d$. Then, the bimodules
${\Gamma}\otimes_{\overline{A}}Z_1,\ldots,
{\Gamma}\otimes_{\overline{A}}Z_t$
parametrize the indecomposable ${\Gamma}$-modules in ${\cal F}(\Delta')$ with
dimension at most $d$.  Moreover, the functors
$${\Gamma}\otimes_{\overline{A}}Z_i\otimes_{\Gamma_i}-:\Gamma_i\g\mod\rightmap{}{\cal F}(\Delta')$$
preserve isoclasses and indecomposables, for all $i\in [1,t]$.
This finishes our proof.

$\,\hfill\hfill\square$

\noindent{\bf Proof of Theorem (\ref{T: Crawley para F(Delta)}) in the  admissible homological system case.}

We assume that $({\cal P},\leq,\{\Delta_i\}_{i\in {\cal P}})$ is an admissible homological system for the finite-dimensional algebra $\Lambda$  and ${\cal F}(\Delta)$ is tame. We
 proceed in four steps, with the notation of (\ref{T: KKO}).

\medskip
\noindent\emph{Step 1: If \hbox{\,} $E\rightmap{g}M$ is a minimal right  almost
split morphism in ${\cal I}(\Gamma)$ and
$$0\rightmap{}Dtr M\rightmap{u}H\rightmap{v}M\rightmap{}0$$
is an almost split sequence in
$\Gamma\g\mod$, then $\dim_k E\leq \dim_k(\Gamma\otimes_QH)$.}
\medskip

Indeed,  recall that the category  ${\cal I}(\Gamma)$ is the full subcategory of $\Gamma\g\mod$ which
consists of
the $\Gamma$-modules of the form $\Gamma\otimes_QN$, for some $N\in
Q\g  \mod$. Given
$H\in  \Gamma\g\mod$, we have the product map
$\mu:\Gamma\otimes_QH\rightmap{}H$, and every morphism
 $g:Z\rightmap{}H$ of  $\Gamma$-modules with $Z  \in {\cal I}(\Gamma)$
 factors through $\mu$. Then, given an almost split sequence in
 $\Gamma\g\mod$
$$0\rightmap{}Dtr M\rightmap{u}H\rightmap{v}M\rightmap{}0,$$  we obtain that
any non-retraction $h:Z\rightmap{}M$ in ${\cal I}(\Gamma)$ factors through $v$, say $h'v=h$,
for some $h':Z\rightmap{}H$; then, $h'$ factors through $\mu$. So the
composition
$v\mu:\Gamma\otimes_Q H\rightmap{}M$ is a right almost
split morphism in ${\cal I}(\Gamma)$. Since $g$ is minimal right almost split, we get
$\dim_k E\leq \dim_k (\Gamma\otimes_Q H)$.

\medskip
\noindent\emph{Step 2: There is a constant $c_0\in{\hueca
N}$ such that, for any almost
split conflation $\zeta: \tau
M\rightmap{f}E\rightmap{g}M$  in ${\cal Q}\g \mod$, we have
$\dim_k \tau M\leq \dim_k E\leq c_0\times \dim_kM.$}
\medskip

Indeed, by (\ref{T: KKO})(3), the exact full and faithful functor $F:{\cal Q}\g\mod\rightmap{}\Gamma\g\mod$ maps $\zeta$  on the
exact sequence
$$0\rightmap{}F(\tau M)\rightmap{ \ F(f) \ }FE\rightmap{ \ F(g) \ }FM\rightmap{}0.$$
Here, we have $FE\cong \Gamma\otimes_QE$ and $FM\cong  \Gamma\otimes_QM$. So, the morphism $F(g)$ belongs to ${\cal I}(\Gamma)$ and is minimal right almost split in this category because the restriction $F:{\cal Q}\g\mod\rightmap{}{\cal I}(\Gamma)\g\mod$ is an equivalence and $g$ is minimal right almost split in ${\cal Q}\g\mod$.

Now, consider the almost split sequence in $\Gamma\g\mod$ ending at $\Gamma\otimes_QM$:
$$0\rightmap{}Dtr(\Gamma\otimes_QM)\rightmap{}H\rightmap{}\Gamma\otimes_QM\rightmap{}0.$$
By Step 1, since $F(g):\Gamma\otimes_QE\rightmap{}\Gamma\otimes_QM$ is minimal right almost split in ${\cal I}(\Gamma)$, we obtain  $\dim_k (\Gamma\otimes_QE)\leq
\dim_k(\Gamma\otimes_QH)$.  From the general theory of almost split sequences, we also have that $\dim_kDtr(\Gamma\otimes_QM)\leq c\dim_k(\Gamma\otimes_QM)$, for some $c\in \hueca{N}$ which depends only on $\Gamma$. Therefore, we obtain
$$\dim_kH=\dim_kDtr(\Gamma\otimes_QM)+\dim_k(\Gamma\otimes_QM)
   \leq (c+1)\times \dim_k(\Gamma\otimes_QM).$$
Hence, we get
$$
\dim_k\tau M\leq \dim_k E\leq \dim_k(\Gamma\otimes_QE)\leq\dim_k(\Gamma\otimes_QH)\leq c_0\times \dim_kM,$$
where $c_0:=(c+1)\times (\dim_k \Gamma)^2$.

\medskip
\noindent\emph{Step 3: For any positive integer $d$, almost every $d$-dimensional indecomposable $M$ in ${\cal Q}\g\mod$ satisfies that $\tau M\cong M$.}
\medskip

It is well known that the category ${\cal Q}\g\mod$ has almost split sequences, see \cite{BK}, \cite{BB}, and \cite{Burt}.
The strategy for this Step is the same used by Crawley-Boevey in \cite{CB1}.
Fix $d\in \hueca{N}$ and let $\zeta:\tau M\rightmap{f}E\rightmap{g}M$ be an almost split conflation in ${\cal Q}\g\mod$ with $\dim_kM\leq d$. Since ${\cal F}(\Delta)$ is not wild, neither is ${\cal F}(\Delta')$ nor ${\cal Q}\g\mod$.
 Then, from (\ref{T: parametrización de módulos de dim acotada}), if we consider the number $d_0:=c_0d$,
there is a minimal seminested ditalgebra ${\cal B}$ and a full and faithful
functor $F':{\cal B}\g\mod\rightmap{}{\cal Q}\g\mod$,  such  that almost every $N\in {\cal Q}\g\mod$ with $\dim_kN\leq d_0$ is
of the form $F'(N')\cong N$,
for some $N'\in {\cal B}\g\mod$. This is the case for almost every $M$, $E$, and $\tau M$.
So $M\cong F'(M')$, $\tau M\cong F'(N')$,
and $E\cong F'(E')$. Since $F'$ is full, we have $f':N'\rightmap{}E'$ and
$g':E'\rightmap{}M'$ such that
$F'(f')=f$ and $F'(g')=g$. Since $F'$ is full and faithful, and $\zeta$ is an
almost split conflation,
it is not hard to see
that
$\zeta': N'\rightmap{f'}E'\rightmap{g'}M'$ is a conflation and, in fact, that it is an almost
split conflation in ${\cal B}\g\mod$. But the almost split conflations in ${\cal
B}\g\mod$ are well known,
see \cite{BSZ}(32.3). Then, we can apply Crawley-Boevey's analysis to show that,  after discarding a finite number of possibilities for the indecomposable $M\in {\cal Q}\g\mod$,  we
have that $\tau M\cong M$. For details, see the proof of \cite{BSZ}(32.6).

\medskip
\noindent\emph{Step 4: Final argument.}
\medskip

From (\ref{T: KKO})(6), we have an equivalence of categories $K:{\cal F}(\Delta)\rightmap{}{\cal Q}\g\mod$
which maps  short exact sequences onto conflations
and $K(\Delta_i)\cong S_i$, for each $i\in {\cal P}$. So $K$ maps almost split sequences of ${\cal F}(\Delta)$ on almost split conflations of ${\cal Q}\g\mod$, and $\tau KM=K\tau M$, for all indecomposable
$M\in {\cal F}(\Delta)$ at which an almost split sequence ends.

For $M\in {\cal F}(\Delta)$, we can denote by $m_i$ the multiplicity of $\Delta_i$ as a composition factor in the $\Delta$-filtration of $M$, thus $\dim_kM=\sum_{i\in {\cal P}}m_i\dim_k\Delta_i$. Since $K$ is exact, we get $\dim_kKM=\sum_{i\in {\cal P}}m_i\dim_kK(\Delta_i)=
\sum_{i\in {\cal P}}m_i$. So, we have $\dim_k KM\leq \dim_kM$.

 Then, given $d\in \hueca{N}$, for all indecomposable $M\in {\cal F}(\Delta)$ with $\dim_kM\leq d$,  such that an almost split sequence ends at $M$, we also have that $KM$ is indecomposable with $\dim_k KM\leq d$. From Step 3,  for almost all such indecomposable $KM$, we have $\tau KM\cong KM$. It follows that, for almost all such $M$, we have $\tau M\cong M$.

 The equivalence $\Theta:\Lambda\g\Mod\rightmap{}
 \Gamma\g\Mod$   of (\ref{T: KKO})(5) induces an equivalence of categories  $\Theta':{\cal F}(\Delta)\rightmap{}{\cal F}(\Delta')$ which preserves and reflects almost split sequences. Thus for almost all indecomposable $M\in {\cal F}(\Delta)$ there is an almost split sequence ending at $M$ iff the same property holds for ${\cal F}(\Delta')$. We will see now that the last case occurs, and from this fact we obtain what we wanted to show for ${\cal F}(\Delta)$.
Indeed, look again to the  equivalence of categories of (\ref{T: KKO})(4):
$$F:{\cal Q}\g\mod\rightmap{}{\cal I}(\Gamma)$$
such that composed with the inclusion into $\Gamma\g\mod$ gives an exact functor. If $\zeta:\tau M\rightmap{f}E\rightmap{g}M$ is an almost split conflation in ${\cal Q}\g\mod$ then its image $F\zeta:F\tau M\rightmap{}FE\rightmap{}FM$ under $F$ is an exact sequence in $\Gamma\g\mod$ with all its terms in ${\cal F}(\Delta')={\cal I}(\Gamma)$; moreover, the morphism  $F(f)$ is minimal left almost split in ${\cal I}(\Gamma)$ and $F(g)$ is minimal right almost split in  ${\cal I}(\Gamma)$, because $F$ is an equivalence.
We know, see \cite{hsb}(12.4), that ${\cal Q}\g\mod$ admits only finitely many indecomposable projectives relative to its exact structure.
Thus, with the exception of finitely many indecomposables $M$, every indecomposable $M\in {\cal Q}\g\mod$ admits an almost split conflation ending at $M$, see for instance \cite{BSZ}(7.18). This implies that
${\cal F}(\Delta')$ has the same property.

$\,\hfill\hfill\square$

\section{The case of  general homological systems}

This section is devoted to the extension of the tame-wild dichotomy for ${\cal F}(\Delta)$ to the general case of an arbitrary homological system. This extension is a consequence of our theorem (1.3) for the admissible homological case and the following theorem (\ref{T: el teorema de MSX}), due to Mendoza, S\'aenz, and Xi. The version we present here is a reformulation of some of their results in  \cite{MSX}\S3. We will briefly recall some of their arguments, in their dual form, for the sake of the reader.

Fix a general homological system $({\cal P},\leq,\{\Delta_i\}_{i\in {\cal P}})$ for a finite-dimensional algebra $\Lambda$.
The following statement is crucial in the proof of (\ref{T: el teorema de MSX}), its dual is verified in the first part of the proof of
\cite{MSX}(3.12).

\begin{lemma}\label{L: de las sucesiones especiales}
For each $ i \in {\cal P}$, there is an exact sequence $$0\rightmap{}V_i\rightmap{}U_i\rightmap{}\Delta_i\rightmap{}0$$ such that $U_i$ is an indecomposable
${\cal F}(\Delta)$-projective and $V_i\in {\cal F}(\Delta)$ has a $\Delta$-filtration with factors of the form $\Delta_{j}$ with $j>i$.
\end{lemma}

We fix a family of special exact sequences, as provided by the last lemma, for the rest of this section.

\begin{lemma}\label{L: suficientes proyectivos en F(Delta)}
For each $M\in {\cal F}(\Delta)$, there is an exact sequence
$$0\rightmap{}K\rightmap{}W\rightmap{}M\rightmap{}0$$
in ${\cal F}(\Delta)$,
such that $W\cong \bigoplus_{i \in {\cal P}}n_iU_i$. Moreover, each $n_i$ is the multiplicity of the factor $\Delta_i$ in any  $\Delta$-filtration of $M$.

The family $\{U_i\}_{i\in {\cal P}}$ is a complete set of representatives of the isomorphism classes of the indecomposable ${\cal F}(\Delta)$-projective modules.
\end{lemma}

\begin{proof} Consider the  commutative diagram
$$\begin{matrix}&&0\,\,&&0\,\,&&0\,\,&&\\
&&\shortlmapdown{}&&\shortlmapdown{}&&\shortlmapdown{}&&\\
  0&\rightarrow&K_L&\rightarrow&K_M&\rightarrow&K_N&\rightarrow&0\\

  &&\shortlmapdown{}&&\shortlmapdown{}&&\shortlmapdown{}&&\\

  0&\rightarrow&W_L&\rightarrow&W_L\oplus W_N&\rightarrow&W_N&\rightarrow&0\\
  &&\shortlmapdown{}&&\shortlmapdown{}&&\shortlmapdown{}&&\\
   0&\rightarrow& L&\rightarrow&M&\rightarrow&N&\rightarrow&0\\
 &&\shortlmapdown{}&&\shortlmapdown{}&&\shortlmapdown{}&&\\
 &&0\,\,&&0\,\,&&0\,\,&&\\
  \end{matrix}$$
where the lower row is any exact sequence in ${\cal F}(\Delta)$, the first column  is an exact sequence in ${\cal F}(\Delta)$ where $W_L$ is ${\cal F}(\Delta)$-projective, the third column  is an exact sequence in ${\cal F}(\Delta)$ where $W_N$ is ${\cal F}(\Delta)$-projective, and the central row is constructed as in the horseshoe lemma. So, all the rows and columns are exact. Therefore, the exact sequence in the central column belongs to ${\cal F}(\Delta)$ and $W_L\oplus W_N$ is ${\cal F}(\Delta)$-projective.

With the preceding argument, using the exact sequences of (\ref{L: de las sucesiones especiales}), it can be shown that each $M\in {\cal F}(\Delta)$ admits an exact sequence as described in the statement of this lemma.
We recall, from \cite{P}, that the multiplicity of the factors in a $\Delta$-filtration of $M$ is independent of the filtration.

It follows that any indecomposable ${\cal F}(\Delta)$-projective is a direct  summand of some $\bigoplus_{i\in {\cal P}}n_i U_i$, so it is isomorphic to some $U_i$. Finally, the fact that the family $\{U_i\}_{i\in {\cal P}}$ consists of pairwise non-isomorphic modules is proved in \cite{MSX}(3.7).
\end{proof}

\begin{definition}\label{D: Gamma y Theta}
Define $U:=\bigoplus_{i\in {\cal P}}U_i$ and $\Gamma:=\End_\Lambda(U)^{op}$. Moreover, consider the family
$\{\Theta_i\}_{i\in {\cal P}}$ of $\Gamma$-modules given by $\Theta_i:=\Hom_\Lambda(U,\Delta_i)$, for $i\in {\cal P}$.
\end{definition}

In the following, we review part of the proof of (\ref{T: el teorema de MSX}):  that $({\cal P},\leq,\{\Theta_i\}_{i\in {\cal P}})$ is an admissible homological system
such that ${\cal F}(\Delta)$ is equivalent to ${\cal F}(\Theta)$ as exact categories.
The quasi-inverse equivalences will be realized by the appropriate  restrictions of the  functors $H:=\Hom_\Lambda(U,-):\Lambda\g\mod\rightmap{}\Gamma\g\mod$ and $T:=U\otimes_\Gamma-:\Gamma\g\mod\rightmap{}\Lambda\g\mod$.
We start with the following.

\begin{lemma}\label{L: Hom(U,-) restringe a fiel-pleno-exact de F(Delta)} The functor $H$ restricts to a full and faithful exact functor
$$H:{\cal F}(\Delta)\rightmap{}\Gamma\g\mod.$$
\end{lemma}

\begin{proof}  Consider the full subcategory $\add(U)$ of
$\Lambda\g\mod$ formed by the direct summands of finite direct sums of $U$. Then, from
\cite{ARS}(II.2.1), we know that $H$ restricts to an equivalence of categories $\add(U)\rightmap{}\add(\Gamma)$, where $\add(\Gamma)$ coincides with the full subcategory of $\Gamma\g\mod$ formed by the finitely generated projectives. Moreover, for each $W\in \add(U)$ and $N\in \Lambda\g\mod$, the same functor $H$ restricts to an isomorphism
$\Hom_\Lambda(W,N)\rightmap{}\Hom_\Gamma(H(W),H(N))$.

Given $M\in {\cal F}(\Delta)$, from (\ref{L: suficientes proyectivos en F(Delta)}) we  derive the existence of an exact sequence of the form
$$W'\rightmap{}W\rightmap{\pi}M\rightmap{}0,$$
where $W,W'\in \add(U)$. This sequence is obtained by splicing the exact sequence $0\rightmap{}K\rightmap{}W\rightmap{}M\rightmap{}0$ produced for $M$ by this lemma, with the exact sequence  $0\rightmap{}K'\rightmap{}W'\rightmap{}K\rightmap{}0$ produced for $K$ by this lemma.  Since $U$ is ${\cal F}(\Delta)$-projective,
applying $H$, we obtain an exact sequence
$$H(W')\rightmap{}H(W)\rightmap{\pi_*}H(M)\rightmap{}0.$$

Then, for any $N\in \Lambda\g\mod$, we have the commutative diagram
$$\begin{matrix}
   0&\rightarrow&\Hom_\Lambda(M,N)&\rightarrow&\Hom_\Lambda(W,N)&\rightarrow&\Hom_\Lambda(W',N)\\
   &&\shortlmapdown{}&&\shortlmapdown{}&&\shortlmapdown{}\\
   0&\rightarrow&\Hom_\Gamma(H(M),H(N))&\rightarrow&\Hom_\Gamma(H(W),H(N))&\rightarrow&\Hom_\Gamma(H(W'),H(N))\\
  \end{matrix}$$
 where the second and third vertical arrows are isomorphisms, then so is the first one. This shows that the restriction of $H$ to the category ${\cal F}(\Delta)$ is full and faithful.  This restriction is exact because $U$ is an ${\cal F}(\Delta)$-projective module.
\end{proof}

\begin{theorem}\label{T: el teorema de MSX}
Given a homological system $({{\cal P}},\leq,\{\Delta_i\}_{i\in {\cal P}})$ for $\Lambda$,
consider the algebra $\Gamma=\End_\Lambda(U)^{op}$ as before and the triple $({{\cal P}},\leq,\{\Theta_i\}_{i\in {\cal P}})$, where $\Theta_i=\Hom_\Lambda(U,\Delta_i)$, for all $i\in {\cal P}$.
Then, $({{\cal P}},\leq,\{\Theta_i\}_{i\in {\cal P}})$ is an admissible homological system for $\Gamma$ and  the functors $H=Hom_\Lambda(U,-)$ and $T=U\otimes_\Gamma-$ induce   quasi-inverse equivalences of  exact categories between ${\cal F}(\Delta)$ and ${\cal F}(\Theta)$.
\end{theorem}

\begin{proof} With the preceding notations, we consider the following natural transformations
 $\alpha:U\otimes_\Gamma H\rightmap{}id_{\Lambda\g\mod}$
and $\beta:id_{\Gamma\g\mod}\rightmap{}H(U\otimes_\Gamma-)$ given by
\begin{enumerate}
 \item For $M\in \Lambda\g\mod$, the morphism $\alpha_M:U\otimes_\Gamma H(M)\rightmap{}M$ has the recipe  $\alpha_M(u\otimes g)=g(u)$.
 \item For $N\in \Gamma\g\mod$, the morphism $\beta_N:N\rightmap{}H(U\otimes_\Gamma N)$ has the recipe  $\beta_N(n)[u]=u\otimes n$.
\end{enumerate}

The notation ${\cal F}(\Theta)$ makes sense for any family of modules. So it will be enough to show that we have quasi-inverse exact equivalences
 $${\cal F}(\Delta)\rightmap{H} {\cal F}(\Theta)\hbox { \ and \ } {\cal F}(\Theta)\rightmap{T}{\cal F}(\Delta),$$
 since this clearly implies that $({{\cal P}},\leq,\{\Theta_i\}_{i\in {\cal P}})$ is a  homological system. It is admissible because $\Gamma=\bigoplus_{i\in {\cal P}}\Theta_i\in {\cal F}(\Theta)$.

 \medskip
 \noindent{\it Step 1. $\alpha_M$ is an isomorphism, for all $M\in {\cal F}(\Delta)$.}
 \medskip

 Since $\alpha_U$ is the composition $U\otimes_\Gamma H(U)=U\otimes_\Gamma \Gamma\cong U$, we know that $\alpha_U$ is an isomorphism. It follows that $\alpha_W$ is an isomorphism, for each $W\in \add(U)$. For a given $M\in {\cal F}(\Delta)$, consider an exact sequence
 $$W'\rightmap{}W\rightmap{}M\rightmap{}0,$$
 as before, in the proof of (\ref{L: Hom(U,-) restringe a fiel-pleno-exact de F(Delta)}), with $W,W'\in \add(U)$. Then, we have a commutative diagram
 $$\begin{matrix}
   U\otimes_\Gamma H(W')&\rightarrow&U\otimes_\Gamma H(W)
   &\rightarrow&U\otimes_\Gamma H(M)&\rightarrow&0\\
   \shortlmapdown{\alpha_{W'}}&&\shortlmapdown{\alpha_W}&&\shortlmapdown{\alpha_M}&&\\
   W'&\rightarrow&W&\rightarrow&M&\rightarrow&0\\
  \end{matrix}$$
  with exact rows, and $\alpha_N$ is an isomorphism.

 \medskip
 \noindent{\it Step 2. The functor $T=U\otimes_\Gamma-$ restricts to an exact functor ${\cal F}(\Theta)\rightarrow {\cal F}(\Delta)$.}
 \medskip

 Given one of the fixed special exact sequences
 $0\rightmap{}V_i\rightmap{}U_i\rightmap{}\Delta_i\rightmap{}0$
 given by (\ref{L: de las sucesiones especiales}), we can apply the functor $H$ to this sequence and obtain an exact sequence
 $0\rightmap{}H(V_i)\rightmap{}H(U_i)\rightmap{}H(\Delta_i)\rightmap{}0,$
 thanks to (\ref{L: Hom(U,-) restringe a fiel-pleno-exact de F(Delta)}). The middle term of this exact sequence is projective in $\Gamma\g\mod$, so from the long homology sequence associated to this sequence, we obtain an exact sequence
 $$0\rightarrow\Tor_1^{\Gamma}(U,\Theta_i)\rightarrow U\otimes_\Gamma H(V_i)\rightarrow U\otimes_\Gamma H(U_i)\rightmap{}U\otimes_\Gamma H(\Delta_i)\rightarrow 0.$$
 Comparing this sequence with the original special exact sequence using the isomorphism  $\alpha$, we obtain that $\Tor^{\Gamma}_1(U,\Theta_i)=0$. So this holds for all $i\in {\cal P}$.   A simple induction argument shows that $\Tor^{\Gamma}_1(U,N)=0$, for any $N\in {\cal F}(\Theta)$.

 Now, given an exact sequence $0\rightarrow \Theta_{j}\rightarrow N\rightarrow \Theta_i\rightarrow 0$ in $\Gamma\g\mod$, since $\Tor^{\Gamma}_1(U,\Theta_i)=0$, we get an exact sequence
 $$0\rightarrow U\otimes_\Gamma \Theta_{j}\rightarrow U\otimes_\Gamma N\rightmap{}U\otimes_\Gamma \Theta_i\rightarrow 0.$$
 Thus, we have the diagram
 $$\begin{matrix}
    0&\rightarrow &U\otimes_\Gamma \Theta_{j}&\rightarrow& U\otimes_\Gamma N&\rightarrow&U\otimes_\Gamma \Theta_i&\rightarrow &0\\
    &&\shortlmapdown{\alpha_{\Delta_{j}}}&&\shortlmapdown{id}&&\shortlmapdown{\alpha_{\Delta_i}}&&\\
    0&\rightarrow& \Delta_{j}&\rightarrow& U\otimes_\Gamma N&\rightarrow&\Delta_i&\rightarrow &0,\\
   \end{matrix}$$
where the vertical arrows are isomorphisms and the middle morphisms of the lower row are defined so that the diagram commutes. Since the first row is exact, so is the second one and $U\otimes_\Gamma N\in {\cal F}(\Delta)$.

A simple induction extends the preceding argument to show that $U\otimes_\Gamma N\in {\cal F}(\Delta)$ for any $N\in {\cal F}(\Theta)$.
Moreover, if $0\rightarrow L\rightarrow M\rightarrow N\rightarrow 0$ is an exact sequence in ${\cal F}(\Theta)$, using  that $\Tor_1^{\Gamma}(U,N)=0$, we obtain that the sequence $0\rightarrow U\otimes_\Gamma L\rightarrow U\otimes_\Gamma M\rightarrow U\otimes_\Gamma N\rightarrow 0$ is exact.

\medskip
 \noindent{\it Step 3. $\beta_N$ is an isomorphism, for all $N\in {\cal F}(\Theta)$.}
 \medskip

 For $i\in {\cal P}$, we have that the following composition is the identity map
 $$\Theta_i=H(\Delta_i)\rightmap{ \ \ \beta_{\Theta_i} \ \ }
 H(U\otimes_\Gamma H(\Delta_i))
 \rightmap{ \ \ (\alpha_{\Delta_i})_* \ \ } H(\Delta_i)=\Theta_i.$$
 So, we know that $\beta_{\Theta_i}$ is an isomorphism for all $i\in {\cal P}$. Again, if we consider an exact sequence of the form
 $0\rightarrow\Theta_{j}\rightarrow N\rightarrow \Theta_{i}\rightarrow 0$ in $\Gamma\g\mod$, we have a commutative diagram with exact rows
  $$\begin{matrix}
  0&\rightarrow& \Theta_{j}&\rightarrow& N &\rightarrow&\Theta_i&\rightarrow &0,\\ &&\shortlmapdown{\beta_{\Theta_{j}}}&&\shortlmapdown{\beta_N}&&\shortlmapdown{\beta_{\Theta_i}}&&\\
    0&\rightarrow &H(U\otimes_\Gamma \Theta_{j})&\rightarrow& H(U\otimes_\Gamma N)&\rightarrow&H(U\otimes_\Gamma \Theta_i)&\rightarrow &0.\\
   \end{matrix}$$
   So $\beta_N$ is an isomorphism. A simple induction, using the preceding argument shows that $\beta_N$ is an isomorphism, for each $N\in {\cal F}(\Theta)$.

   In summary, we have shown that we can restrict the functors $H$ and $T$ to functors
   $H:{\cal F}(\Delta)\rightmap{}{\cal F}(\Theta)$ and $T:{\cal F}(\Theta)\rightmap{}{\cal F}(\Delta)$ and that these restrictions are exact. Moreover, the natural transformations $\alpha$ and $\beta$ when restricted appropriately give rise to isomorphisms of functors $\alpha:TH\rightmap{}id_{{\cal F}(\Delta)}$ and $\beta:id_{{\cal F}(\Theta)}\rightmap{}HT$. So, the restrictions
   $H$ and $T$ are quasi-inverse equivalences between the exact categories ${\cal F}(\Delta)$ and ${\cal F}(\Theta)$.
\end{proof}

Finally, we can proceed to the following.
\medskip

\noindent{\bf Proof of Theorem (\ref{T: main theorem qh-algebas}) in the general homological system case.}\\
Consider the algebra $\Gamma$, the admissible  homological system
$({\cal P},\leq,\{\Theta_i\}_{i\in {\cal P}})$ for $\Gamma$, and the equivalence functor $U\otimes_\Gamma-:{\cal F}(\Theta)\rightmap{}{\cal F}(\Delta)$ as in (\ref{T: el teorema de MSX}).

From (\ref{L: tame-wild disjointness for cats of modules}), we already know that ${\cal F}(\Delta)$ can not be simultaneously tame and wild. From (\ref{T: main theorem qh-algebas}), we know that
${\cal F}(\Theta)$ is tame or wild.

If ${\cal F}(\Theta)$ is wild, there is a $\Gamma\g k\langle x,y\rangle$-bimodule $Z$, free of finite rank as a right $k\langle x,y\rangle$-module, such that the functor
$Z\otimes_{k\langle x,y\rangle}-:k\langle x,y\rangle\g\mod\rightmap{}{\cal F}(\Theta)$ preserves indecomposables and isomorphism classes. The $\Lambda\g k\langle x,y\rangle$-bimodule $\overline{Z}:=U\otimes_\Gamma Z$ is finitely generated by the right and
 $\overline{Z}\otimes_{k\langle x,y\rangle}-:k\langle x,y\rangle\g\mod\rightmap{}{\cal F}(\Delta)$ preserves indecomposables and isomorphism classes. The bimodule $\overline{Z}$ can be converted into a new $\Lambda\g k\langle x,y\rangle$-bimodule $\widehat{Z}$, which is free of finite rank by the right, and such that
$\overline{Z}\otimes_{k\langle x,y\rangle}-:k\langle x,y\rangle\g\mod\rightmap{}{\cal F}(\Delta)$ preserves indecomposables and isomorphism classes, see \cite{D}, \cite{CB1}\S6, or \cite{BSZ}(22.17).

Now, assume that ${\cal F}(\Theta)$ is tame. In order to show that the tameness of ${\cal F}(\Theta)$ transfers to ${\cal F}(\Delta)$, it is convenient to notice the following.

Let us call a \emph{multiplicity vector} $m=\{m_i\}_{i\in {\cal P}}$ any sequence with non-negative integer entries and define $\vert\vert m\vert\vert_\Delta=\sum_im_i\dim_k\Delta_i$. Notice that, for any $d\in \hueca{N}$, there are only finitely many multiplicity vectors with
$\vert\vert m\vert\vert_\Delta=d$. Any $M\in {\cal F}(\Delta)$ has a multiplicity vector $m_M=\{m_i\}_{i\in {\cal P}}$, where $m_i$ is the multiplicity of $\Delta_i$ as a factor in any $\Delta$-filtration of $M$, which satisfies that $\vert\vert m\vert\vert_\Delta=\dim_kM$.

Thus,  we can parametrize almost every indecomposable $\Lambda$-module in ${\cal F}(\Delta)$ with dimension $d$, for all $d\in \hueca{N}$, iff we can parametrize almost every  indecomposable $\Lambda$-module in ${\cal F}(\Delta)$ with multiplicity vector $m_M=m$, for all $m$ with $\vert\vert m\vert\vert_\Delta=d$.

Then, in order to show that ${\cal F}(\Delta)$ is tame, it will be enough to show that almost every indecomposable $\Lambda$-module in ${\cal F}(\Delta)$ with given multiplicity vector $m=\{m_i\}_{i\in {\cal P}}$ can be parametrized with a finite family of $\Lambda\g\Gamma_i$-bimodules. So fix any such  multiplicity vector $m$. Since ${\cal F}(\Theta)$ is tame, there are rational algebras $\Gamma_1,\ldots,\Gamma_t$ and $\Gamma\g \Gamma_i$-bimodules $Z_1,\ldots,Z_t$ finitely generated and free by the right, such that every indecomposable $\Gamma$-module $N$ in ${\cal F}(\Theta)$ with multiplicity vector $m_N=m$ is of the form
$N\cong Z_i\otimes_{\Gamma_i}S$, for some $i$ and some simple $\Gamma_i$-module $S$.

Now take any indecomposable $M\in {\cal F}(\Delta)$ with multiplicity vector $m$, then there is some indecomposable $N\in {\cal F}(\Theta)$ with $M\cong U\otimes_\Gamma N$. Since $U\otimes_\Gamma-$ is exact and  $U\otimes_\Gamma\Theta_i\cong \Delta_i$, for all
$i\in {\cal P}$, the multiplicity vectors of $N$ and $M$ coincide: $m_N=m_M=m$. Then, for almost every such module $M$, there is an index $i$ and a simple $\Gamma_i$-module $S$ such that
$M\cong U\otimes_\Gamma N\cong U\otimes_\Gamma  Z_i\otimes_{\Gamma_i} S$. Now, the
$\Lambda\g\Gamma_i$-bimodule $\overline{Z}_i:=U\otimes_\Gamma Z_i$ is finitely generated by the right. As usual, we can replace each one of these bimodules $\overline{Z}_i$ by new $\Lambda\g\Gamma_i$-bimodules $\widehat{Z}_i$ which are finitely generated and free by the right and which parametrize almost every $\Lambda$-module parametrized by $\overline{Z}_i$. So, ${\cal F}(\Delta)$ is also tame.
\hfill$\square$
\medskip

The proof of theorem (\ref{T: Crawley para F(Delta)}) in the general homological system case is similar to the preceding one: It is reduced to the  admissible homological system case using multiplicity vectors, which are preserved by the exact equivalence $U\otimes_\Gamma-:{\cal F}(\Theta)\rightmap{}{\cal F}(\Delta)$ given by theorem (\ref{T: el teorema de MSX}).

\begin{remark} In \cite{T}, H. Treffinger constructs a family of examples of non-admissible homological systems for finite-dimensional algebras $\Lambda$ with size which can be arbitrarily large in comparison to the number
of isomorphism classes of simple $\Lambda$-modules. The preceding theorem applies to ${\cal F}(\Delta)$ for each one of these homological systems.
\end{remark}

\noindent{\bf  Acknowledgements.} The first autor thanks S. Koenig for his
invitation to visit the Institute of
Algebra
and Number Theory, Stuttgart University, in November 2015, where this
research
started. He also thanks S. Koenig and J. K\"ulshammer by helpful comments,
specially for showing him several examples of reduction functors which
appear in \cite{K}.

\hskip2cm

\vbox{\noindent R. Bautista\\
Centro de Ciencias Matem\'aticas\\
Universidad Nacional Aut\'onoma de M\'exico\\
Morelia, M\'exico\\
raymundo@matmor.unam.mx\\}

\vbox{\noindent E. P\'erez\\
Facultad de Matem\'aticas\\
Universidad Aut\'onoma de Yucat\'an\\
M\'erida, M\'exico\\
jperezt@correo.uady.mx\\}

\vbox{\noindent L. Salmer\'on\\
Centro de Ciencias  Matem\'aticas\\
Universidad Nacional Aut\'onoma de M\'exico\\
Morelia, M\'exico\\
salmeron@matmor.unam.mx\\}

\end{document}